\pdfoutput=1
\documentclass[11pt]{article}
\usepackage[colorlinks,
linkcolor=red,
anchorcolor=blue,
citecolor=blue
]{hyperref}
\usepackage{enumerate}
\usepackage[OT1]{fontenc}
\usepackage[usenames]{color}
\usepackage{smile}
\usepackage{mathrsfs}
\usepackage{fullpage}
\usepackage[protrusion=true, expansion=true]{microtype}
\usepackage{float}
\usepackage{subfigure}
\usepackage{amsfonts,amsmath,amssymb,amsthm,url,xspace}
\usepackage{tikz}
\usepackage{verbatim}
\usetikzlibrary{arrows,shapes}
\usepackage{mathtools}
\usepackage{authblk}
\usepackage{enumitem}
\usepackage{mathabx}
\usepackage{centernot}
\usepackage{comment}
\usepackage{cancel}
\usepackage[bottom]{footmisc}
\usepackage{bibunits}
\usepackage{graphicx,caption}
\usepackage{epstopdf}
\usepackage{afterpage}

\usepackage{lscape}
\usepackage{chngcntr}
\counterwithout{equation}{section}
\usepackage{makecell}

\mathtoolsset{showonlyrefs}

\allowdisplaybreaks[1]

\usepackage{xargs}
\usepackage[colorinlistoftodos,prependcaption,textsize=tiny]{todonotes}
\newcommandx{\unsure}[2][1=]{\todo[linecolor=red,backgroundcolor=red!25,bordercolor=red,#1]{#2}}
\newcommandx{\change}[2][1=]{\todo[linecolor=blue,backgroundcolor=blue!25,bordercolor=blue,#1]{#2}}
\newcommandx{\info}[2][1=]{\todo[linecolor=OliveGreen,backgroundcolor=OliveGreen!25,bordercolor=OliveGreen,#1]{#2}}
\newcommandx{\improvement}[2][1=]{\todo[linecolor=Plum,backgroundcolor=Plum!25,bordercolor=Plum,#1]{#2}}

\newenvironment{rotatepage}%
{\pagebreak[4]\global\pdfpageattr\expandafter{\the\pdfpageattr/Rotate 90}}%
{\pagebreak[4]\global\pdfpageattr\expandafter{\the\pdfpageattr/Rotate 0}}%

\usepackage[plain,noend]{algorithm2e}

\makeatletter
\renewcommand{\algocf@captiontext}[2]{\quad #1\algocf@typo. \AlCapFnt{}#2} 
\def\@algocf@capt@plain{top}
\renewcommand{\algocf@makecaption}[2]{%
	\addtolength{\hsize}{\algomargin}%
	\sbox\@tempboxa{\algocf@captiontext{#1}{#2}}%
	\ifdim\wd\@tempboxa >\hsize
	\hskip .5\algomargin%
	\parbox[t]{\hsize}{\algocf@captiontext{#1}{#2}}
	\else%
	\global\@minipagefalse%
	\hbox to\hsize{\box\@tempboxa}
	\fi%
	\addtolength{\hsize}{-\algomargin}%
}
\makeatother

\def\T{{ \mathrm{\scriptscriptstyle T} }}

\newcommand{\red}{\color{black}}

\newcommand{\LD}{\langle}
\newcommand{\RD}{\rangle}
\newcommand{\bLD}{\big\langle}
\newcommand{\bRD}{\big\rangle}

\newcommand{\mQ}{\mathbb{Q}}
\newcommand{\mR}{\mathbb{R}}
\newcommand{\mE}{\mathbb{E}}

\newcommand{\mN}{\mathcal{N}}
\newcommand{\mT}{\mathcal{T}}
\newcommand{\mJ}{\mathcal{J}}

\newcommand{\mS}{\mathcal{S}}
\newcommand{\mU}{\mathcal{U}}
\newcommand{\mF}{\mathcal{F}}

\newcommand{\tB}{\bB^\star}

\newcommand{\tA}{\bA^\star}

\newcommand{\hB}{\hat{\bB}}

\newcommand{\hA}{\hat{\bA}}

\newcommand{\hmu}{\hat{\mu}}

\newcommand{\hU}{\hat{U}}

\newcommand{\hS}{\hat{S}}

\newcommand{\tV}{{V^\star}}
\newcommand{\barU}{\bar{U}}

\newcommand{\barOmega}{\bar{\Omega}}
\def\tA{A^\star}
\def\tB{B^\star}

\def\tAT{{A^{\star \T}}}
\def\tVT{{V^{\star \T}}}
\def\tZT{{Z^{\star \T}}}

\newcommand{\tmuy}{{\mu^\star_Y}}
\newcommand{\tmux}{{\mu^\star_X}}
\newcommand{\ttSigma}{\tilde{\Sigma}}
\newcommand{\ttmu}{{\mu^\star}}

\newcommand{\baralpha}{\bar{\alpha}}
\newcommand{\bars}{\bar{s}}

\newcommand{\eig}{\text{eig}}

\newcommand{\TR}{\text{tr}}

\newcommand{\SUPP}{\text{supp}}
\def\Pr{\text{pr}}

\newcommand{\rank}{\text{rank}}

\def\mQ{\mathcal{Q}}
\def\Q{\mathbb{Q}}
\def\I{\mathcal{I}}
\def\S{\mathbb{S}}

\def\P{\mathcal{P}}

\def\mL{\mathcal{L}}
\def\T{{ \mathrm{\scriptscriptstyle T} }}
\def\mbL{\bar{\mathcal{L}}}
\def\mC{\mathcal{C}}
\def\N{\mathfrak{N}}

\newcommand{\barS}{\bar{S}}
\def\b1{\pmb{1}}

\newcommand{\ttSigmax}{\tilde{\Sigma}_X}
\newcommand{\ttSigmay}{\tilde{\Sigma}_Y}
\newcommand{\tSigmax}{{\Sigma^\star_X}}
\newcommand{\tSigmay}{{\Sigma^\star_Y}}
\newcommand{\tOmegax}{{\Omega^\star_X}}
\newcommand{\tOmegay}{{\Omega^\star_Y}}
\newcommand{\tDelta}{\Delta^\star}
\newcommand{\hSigmax}{{\hat{\Sigma}_X}}
\newcommand{\hSigmay}{{\hat{\Sigma}_Y}}
\newcommand{\tS}{S^\star}
\newcommand{\ttR}{R^\star}
\newcommand{\tR}{{R^\star}}
\newcommand{\hmux}{{\hat{\mu}_X}}

\newcommand{\tLambda}{{\Lambda^\star}}
\newcommand{\tU}{U^\star}
\newcommand{\tUT}{{U^{\star \T}}}
\newcommand{\hSigma}{{\hat{\Sigma}}}
\newcommand{\tL}{{L^\star}}
\newcommand{\tLT}{{L^{\star \T}}}
\newcommand{\hDelta}{{\hat{\Delta}}}
\newcommand{\tXi}{\Xi^{\star}}

\def\hA{\hat{A}}
\def\htA{\hat{A}^{\star}}
\def\hB{\hat{B}}
\def\htB{\hat{B}^{\star}}
\def\htBT{\hat{B}^{\star \T}}
\def\htAT{\hat{A}^{\star \T}}
\newcommand{\halpha}{{\hat{\alpha}}}
\def\hs{{\hat{s}}}
\newcommand{\hr}{{\hat{r}}}
\def\hR{{\hat{R}}}
\newcommand{\tlambda}{\lambda^\star}
\newcommand{\barj}{{\bar{j}}}
\newcommand{\ttj}{{\tilde{j}}}
\newcommand{\barZ}{{\bar{Z}}}
\newcommand{\tZ}{{Z^{\star}}}
\newcommand{\CI}{\mathrel{\perp\mspace{-10mu}\perp}}

\begin{document}
\begin{bibunit}[my-plainnat]

\title{Estimating Differential Latent Variable Graphical Models \\ with Applications to Brain Connectivity}
\author[1]{Sen Na}
\author[2]{Mladen Kolar}
\author[3]{Oluwasanmi Koyejo}
\affil[1]{Department of Statistics, University of Chicago}
\affil[2]{Booth School of Business, University of Chicago}
\affil[3]{Department of Computer Science and Beckman Institute for Advanced Science and Technology\\ University of Illinois at Urbana-Champaign}
\date{}
\maketitle

\begin{abstract}

Differential graphical models are designed to represent the difference between the conditional dependence structures of two groups, thus are of particular interest for scientific investigation. Motivated by modern applications, this manuscript considers an extended setting where each group is generated by a latent variable Gaussian graphical model. Due to the existence of latent factors, the differential network is decomposed into sparse and low-rank components, both of which are symmetric indefinite matrices. We estimate these two components simultaneously using a two-stage procedure: (i) an initialization stage, which computes a simple, consistent estimator, and (ii) a convergence stage, implemented using a projected alternating gradient descent algorithm applied to a nonconvex objective, initialized using the output of the first stage. We prove that given the initialization, the estimator converges linearly with a nontrivial, minimax optimal statistical error. Experiments on synthetic and real data illustrate that the proposed nonconvex procedure outperforms existing methods.

\end{abstract}


\section{Introduction}\label{sec:1}

Gaussian graphical models \citep{Lauritzen1996Graphical} are used to capture complex relationships among observed variables in a variety of fields, ranging from computational biology \citep{Friedman2004Inferring}, genetics \citep{Lauritzen2003Graphical}, to neuroscience \citep{Smith2011Network}. Each node in a graphical model represents an observed variable and the (undirected) edge between two nodes is present if the nodes are conditionally dependent given all the other variables; thus (sparse) graphical models are highly interpretable and have been adopted for a wide variety of applications.

Of particular interest in this manuscript are applications to cognitive neuroscience, specifically functional connectivity; the study of functional interactions between brain regions, thought to be necessary for cognition~\citep{Bullmore2009Complex}. Importantly, functional connectivity is a promising biomarker for mental disorders~\citep{Castellanos2013Clinical}, where the primary object of study is the differential network, that is the differences in connectivity between healthy individuals and patients. See \citet{Bielza2014Bayesian} for a detailed review. In genetics, scientists are interested in understanding differences in gene networks between experimental conditions (that is, the case-control study), to elucidate potential mechanisms underlying genetic functions. The differential network between two groups provides important signals for detecting differences. The interested reader can find more details on estimating genetic network differences in \citet{Hudson2009differential}, \citet{DelaFuente2010differential}, and \citet{Ideker2012Differential}.

In many applications, it is clear that relationships between the observed variables are confounded by the presence of unobserved, latent factors. For example, physiological and demographic factors may have confounding effects on graphical model estimates in neuroscience and genetics \citep{Gaggiotti2009Disentangling, Willi2009Demographic, Durkee2012Prevalence}. The standard approach of estimating sparse Gaussian graphical models is of limited use here as, due to the confounding, the marginal precision matrix is not sparse. Instead of sparsity of the marginal graph, latent variable Gaussian graphical models exploit the observation that the marginal graph of the observed variables can be decomposed into a superposition of a sparse matrix and a low-rank matrix~\citep{Chandrasekaran2012Latent, Meng2014Learning}.

This manuscript addresses the estimation of differential networks with latent factors. Suppose two groups of observed variables are drawn from latent variable Gaussian graphical models and one is interested in differences in the conditional dependence structure between the two groups, which can be reduced to estimating the difference of their respective precision matrices. For this task, we develop a novel estimation procedure that does not require separate estimation for each group, which allows for robust estimation even if each group contains hub nodes. We propose a two-stage algorithm to optimize a nonconvex objective. In the first stage, we derive a simple, consistent estimator, which then serves as initialization for the next stage. In the second stage, we employ projected alternating gradient descent with a constant step size. The iterates are proven to linearly converge to a region around the ground truth, whose radius is characterized by the statistical error. Compared with {\red potential} convex approaches, our nonconvex approach would enjoy lower computation costs and hence be more time efficient. Extensive experiments validate our conceptual and theoretical claims. Our code is available at \url{https://github.com/senna1128/Differential-Network-Estimation-via-Nonconvex-Approach}.

\section{Background}\label{sec:2}

\subsection{Notations}

Throughout the paper, we use $\S^{d\times d}$, $\Q^{d\times d}$, $I_d$ to denote the set of $d\times d$ symmetric, orthogonal matrices, and {\red identity} matrices respectively. Given an integer $d$, we let $[d] = \{1, 2, \ldots, d\}$ be the index set. For any two scalars $a$ and $b$, we denote $a\lesssim b$ if $a\leq cb$ for some constant $c$. Similarly, $a \gtrsim b$ if {\red $b\leq ca$} for some constant $c$.  We write $a\asymp b$ if $a\lesssim b$ and $b\lesssim a$. We use $a\wedge b = \min(a, b)$ and $a\vee b = \max(a, b)$. For matrices $A, B\in\S^{d\times d}$, we write $A\prec B$ if $B-A$ is positive definite and $A\preceq B$ if $B-A$ is positive semidefinite. We use $\LD A, B\RD = \TR(A^\T B)$. For a matrix $A$, $\sigma_{\min}(A)$ and $\sigma_{\max}(A)$ denote the minimum and maximum singular values, respectively. For a vector $a$, $\|a\|_p$ denotes its $\ell_p$ norm, $p\geq 1$, and $\|a\|_0 = |\SUPP(a)|$ denotes the number of nonzero entries of $a$. For a matrix $A$, $\|A\|_p$ denotes the matrix induced $p$-norm, $\|A\|_F$ denotes the Frobenius norm, $\|A\|_*$ denotes the nuclear norm, and $\|A\|_{p, q} = \{\sum_{j}(\sum_{i} |A_{ij}|^p)^{q/p}\}^{1/q}$, {\red which is calculated by computing $\ell_q$ norm of the vector whose each entry corresponds to the $\ell_p$ norm of a column of $A$. For example, $\|A\|_{0, q} =  \|a\|_q$ where $j$-th entry of $a$ is $a_j = \|A_{\cdot,j}\|_0$, and similarly $\|A\|_{p, \infty} = \|a\|_{\infty}$ with $a_j = \|A_{\cdot, j}\|_p$.} Given a set $\mC\subseteq \mR^{d\times r}$, the projection operator $\P_{\mC}(\cdot)$ is defined as $\P_{\mC}(U) = \arg\min_{V\in\mC}\|V - U\|_F$.

\subsection{Preliminaries and related work}\label{sec:2.2}

A Gaussian graphical model \citep{Lauritzen1996Graphical} consists of a graph $G = (V, E)$, where $V = \{1, \ldots, d\}$ is the set of vertices and $E$ is the set of edges, and a $d$-dimensional random vector $X = (X_1, \ldots, X_d)^\T \sim N(\tmux, \tSigmax)$ that is Markov with respect to $G$. The precision matrix of $X$, $\tOmegax = (\tSigmax)^{-1}$, encodes the conditional independence relationships underlying $X$ and the graph structure $G$ where
\begin{align*}
X_i \CI X_j \mid \{X_k: k \in V \backslash \{i, j\}\} \Longleftrightarrow (i, j)\notin E\Longleftrightarrow(\tOmegax)_{i, j} = 0.
\end{align*}
See \citet{Drton2017Structure} for a recent overview of literature on structure learning of Gaussian graphical models with applications.

Latent variable Gaussian graphical models extend the applicability of Gaussian graphical models by assuming the existence of latent factors, $X_H\in\mR^r$, that confound the observed conditional independence structure of the observed variables $X_O\in\mR^d$. In particular, the observed and hidden components are assumed to be jointly normally distributed as $(X_O^\T, X_{H}^\T)^\T \sim N(\ttmu, \Sigma^\star)$ with
\begin{align}
\label{lvggm}
\ttmu =
\begin{pmatrix}
\mu_{X_O}^\star\\
\mu_{X_H}^\star
\end{pmatrix},\quad
\Sigma^{\star} =
\begin{pmatrix}
\Sigma^\star_{OO} & \Sigma^\star_{OH}\\
\Sigma^\star_{HO} & \Sigma^\star_{HH}
\end{pmatrix},
\quad
\Omega^\star = (\Sigma^\star)^{-1} = \begin{pmatrix}
\Omega_{OO}^\star & \Omega_{OH}^\star \\
\Omega_{HO}^\star & \Omega_{HH}^\star
\end{pmatrix}.
\end{align}
While the joint precision matrix $\Omega^\star$ is commonly assumed sparse, the marginal precision matrix of the observed component $X_O\sim N(\mu_{X_O}^\star, \Sigma^\star_{OO})$ is given as
\begin{align}\label{dec:1}
(\Sigma^\star_{OO})^{-1} = \Omega_{OO}^\star - \Omega_{OH}^\star(\Omega_{HH}^{\star})^{-1}\Omega_{HO}^\star,
\end{align}
and in general is not sparse. The marginal precision matrix of observed variables $X_O$ has a sparse plus low-rank structure, since the precision matrix of the conditional distribution of $X_O$ given $X_H$, $\Omega_{OO}^\star = \{\Sigma^\star_{OO} - \Sigma^\star_{OH}(\Sigma^\star_{HH})^{-1}\Sigma^\star_{HO}\}^{-1}$, is sparse and positive definite, while the second term in \eqref{dec:1} is a rank-$r$ positive semidefinite matrix, which in general is not sparse.

We study the problem of estimating the differential network, which is characterized by the difference between two precision matrices, from two groups of samples distributed according to latent variable Gaussian graphical models. More specifically, suppose that we have independent observations of $d$ variables from two groups of subjects: $X_i = (X_{i1}, \ldots, X_{id})^\T\sim N(\tmux, \tSigmax)$ for $i=1,\ldots, n_X$ from one group and $Y_i = (Y_{i1},\ldots, Y_{id})^\T\sim N(\tmuy, \tSigmay)$ for $i = 1, \ldots, n_Y$ from the other. The differential network is defined as the difference between two precision matrices, denoted as $\tDelta = \tOmegax - \tOmegay$, where $\tOmegax = (\tSigmax)^{-1}$ and  $\tOmegay = (\tSigmay)^{-1}$. We assume that the differential network can be decomposed as
\begin{align}\label{dec:3}
\tDelta = \tS + \tR,
\end{align}
where $\tS$ is sparse and $\tR$ is low-rank and they are both symmetric, but indefinite matrices. Such structure arises under the assumption that the group specific precision matrices have the sparse plus low rank structure as in \eqref{dec:1}. {\red Here, $\tR$ corresponds to the difference of two low-rank matrices, whose rank is upper bounded by the sum of their ranks, hence it's natural for $\tR$ to be low-rank}. However, imposing the sparse plus low rank structure on the differential networks puts fewer restrictions on the data generating process. {\red For example, \eqref{dec:3} also appears if one group is from the latent model while the other is from regular graphical model.}

Estimating the differential network $\tDelta$ can be na\"ively achieved by estimating group-specific precision matrices first and then taking their difference. A related approach is to learn the group-specific precision matrices by maximizing the penalized joint likelihood of samples from both groups with a penalty that encourages the estimated precision matrices to have the same support. Both of these approaches require imposing strong assumptions on the individual precision matrices and are not robust in practice \citep{Shojaie2020Differential}. For example, when hub nodes are present in a group-specific network \citep{Barabasi2004Network}, estimation of an individual precision matrix is challenging as the sparsity assumption is violated, while direct estimation of the differential network is possible without imposing overly restrictive assumptions. \citet{Zhao2014Direct} directly estimated the differential network $\tDelta$ by minimizing $\|\Delta\|_{1,1}$ subject to the constraint $\|\hSigmax\Delta\hSigmay - (\hSigmay - \hSigmax)\|_{\infty,\infty}\leq \lambda$. Under suitable conditions, the truncated and symmetrized estimator satisfies $\|\hDelta - \tDelta\|_F\lesssim \{(n_X\wedge n_Y)^{-1}\|\tDelta\|_{0,1} \log d\}^{1/2}$. \citet{Liu2014Direct} and \citet{Kim2019Two} developed procedures for estimation and inference of differential networks when $X$ and $Y$ follow a general exponential family distribution.  See \citet{Shojaie2020Differential} for a recent review.  In the presence of latent factors, the differential network is not guaranteed to be sparse and, therefore, the aforementioned methods are not applicable. We develop methodology to learn the differential network from latent variable Gaussian graphical models.

\citet{Chandrasekaran2012Latent} estimated a precision matrix under a latent variable Gaussian graphical model by minimizing the penalized negative Gaussian log-likelihood
\begin{equation}\label{convex:1}
\begin{aligned}
(\hS_X, \hR_X) = \arg\min_{S, R}\text{\ \ \ } &\TR\{\rbr{S+R}\hSigmax\} - \log\det(S+R) + \lambda_n\big(\gamma\|S\|_{1,1} + \|R\|_*\big),\\
\text{subject to\ \ } & S+R\succ 0, \text{\ \ \ } -R\succeq 0,
\end{aligned}
\end{equation}
where $\hSigmax$ is a sample covariance based on $n_X$ samples. Under suitable identifiability and regularity conditions, $\gamma^{-1}\|\hS_X - \tS_X\|_{\infty, \infty} \vee \|\hR_X - \ttR_X\|_2\lesssim \rbr{{d}/{n_X}}^{1/2}$ when $\lambda_n\asymp \rbr{{d}/{n_X}}^{1/2}$. \citet{Meng2014Learning} developed an alternating direction method of multipliers for more efficient minimization of \eqref{convex:1} and showed that $\|\hat{\Omega}_X - \tOmegax\|_F\lesssim (s\log d/n_X)^{1/2} + (rd/n_X)^{1/2}$ {\red with $s = \|\tS_X\|_{0, 1}$ being the overall sparsity of $\tS_X$}. The main drawback of minimizing \eqref{convex:1} arises from the fact that in each iteration of the algorithm, the matrix $R$ is updated without taking its {\red low-rank} structure into account. \citet{Xu2017Speeding} explicitly represented the low-rank matrix as $R = -UU^\T$ for $U\in\mR^{d\times r}$ and minimized the resulting nonconvex objective using the alternating gradient descent. Our alternating gradient descent procedure is closely related to this work, but more challenging in several aspects. First, the log-likelihood is not readily available for differential networks. We hence rely on a quasi-likelihood, which reaches its minimum at $\tDelta$. Second, the low-rank matrix $\tR$ in our setup is indefinite, so we have to estimate the positive index of inertia for $\tR$ as well. Third, in order to establish theoretical properties of our estimator, we avoid relying on the concentration of $\|\hSigmax\|_1$ that requires $n_X \asymp d^2$. By a more careful analysis, we improve the sample complexity to $n_X \asymp d\log d$.

Finally, our work is related to a growing literature on robust estimation where parameter matrices have the sparse plus low-rank structure. Example applications include robust principal component analysis \citep{Candes2011Robust,Chandrasekaran2011Rank}, robust matrix sensing \citep{Fazel2008Compressed}, and robust multi-task learning \citep{Chen2011Integrating}. \citet{Zhang2018Unified} proposed a unified framework to analyze convergence of alternating gradient descent when applied on sparse plus low-rank recovery. However, our problem is more challenging and does not satisfy conditions required by their framework. In particular, we use noisy covariance matrices to recover the difference of their true inverses via a quadratic loss. The Hessian matrix in our problem is $(\hSigmay \otimes \hSigmax + \hSigmax \otimes \hSigmay)/2$ with $\otimes$ denoting the Kronecker product, which is different compared to examples in robust estimation where the expectation of the Hessian is identity. As a result, the Condition 4.4 in \citet{Zhang2018Unified} fails to hold and hence we need a problem-oriented analysis. {\red We address three main technical challenges. First, the low-rank matrix $\tR$ is indefinite, while the existing procedures only handle positive semidefinite low-rank matrices.  We develop an estimator that consistently recovers the positive index of inertia of $\tR$. Second, the analysis of the estimator is challenging as the incoherence condition is naturally imposed on $\tU$, but in the analysis, e.g., when bounding the error in the gradient of the loss, the low-rank component $\tU$ is always multiplied by a sample covariance matrix. Finally, we use properties of the Wishart distribution to provide finer analysis, avoid any concentration of sample covariance matrices in $\|\cdot\|_1$ norm, and improve the sample complexity established in \cite{Xu2017Speeding}.}

\section{Methodology}\label{sec:3}

\subsection{Empirical loss}

We introduce the estimator of the differential network $\tDelta$ based on observations from latent variable Gaussian graphical models described in \textsection\ref{sec:2.2}. Since $\tDelta$ satisfies $\rbr{\tSigmax\tDelta\tSigmay + \tSigmay\tDelta\tSigmax}/2 - (\tSigmay - \tSigmax) = 0$, one can minimize the quadratic loss $\mL(\Delta) = \TR\cbr{\Delta \tSigmax\Delta \tSigmay / 2  -\Delta\rbr{\tSigmay - \tSigmax}}$. This loss has been used in \cite{Xu2016Semiparametric} and \cite{Yuan2017Differential} to learn sparse differential networks. Using the decomposition in \eqref{dec:3} and substituting the true covariance matrices with sample estimates, we arrive at the following empirical loss
\begin{align}\label{loss:popu:SR}
\mL_n(S, R) = \TR\cbr{(S + R) \hSigmax(S + R) \hSigmay / 2  - (S + R)(\hSigmay - \hSigmax)},
\end{align}
where $S\in\S^{d\times d}$ denotes the sparse component, $R\in\S^{d\times d}$ denotes the low-rank component with rank $r$, and $\hSigmax = n_X^{-1}\sum_{i=1}^{n_X}(X_i - \hmux)(X_i - \hmux)^\T$ with $\hmux = n_X^{-1}\sum_{i=1}^{n_X}X_i$ and $\hSigmay$ is similarly defined. The empirical loss $\mL_n(S, R)$ in \eqref{loss:popu:SR} is convex with respect to the pair $(S, R)$ and strongly convex if either of the two components is fixed.

Directly minimizing $\mL_n(S, R)$ over a suitable constraint set would be computationally challenging as in each iteration $R$ would need to be updated in $\mR^{d\times d}$, without utilizing its low-rank structure. To that end, we explicitly factorize $R$ as $R = U\Lambda U^\T$, where columns of $U\in\mR^{d\times r}$ are aligned with eigenvectors that correspond to nonzero eigenvalues, and $\Lambda\in\mR^{r\times r}$ is the diagonal sign matrix with diagonal elements being the sign of each eigenvalue. Without loss of generality, we assume $\Lambda$ has $+1$ entries on the diagonal first, followed by $-1$ entries. This factorization implicitly imposes the constraints that $\rank(R)= r$ and $R = R^\T$. Different from estimating the single latent variable Gaussian graphical model in \eqref{dec:1}, where the low-rank component is positive semidefinite and can be factorized as $R = UU^\T$, $\tR$ in our model \eqref{dec:3} is only symmetric as it corresponds to the difference of two low-rank positive semidefinite matrices. Thus, $\tR =\tU\tLambda\tUT$ and we need to estimate $\tLambda$ as well. Plugging the factorization into \eqref{loss:popu:SR}, we aim to minimize the following empirical nonconvex objective
\begin{multline}\label{loss:emp:SULam}
\mbL_n(S, U, \Lambda)
= \mL_n(S, U\Lambda U^\T)
=  \TR\big\{(S + U\Lambda U^\T)\hSigmax(S+U\Lambda U^\T)\hSigmay / 2 \\
\quad - (S+U\Lambda U^\T)(\hSigmay - \hSigmax)\big\},
\end{multline}
over a suitable constraint set that we discuss next.

We assume $\tS\in\S^{d\times d}$ has at most $s$ nonzero entries overall and each column (row) has at most a certain fraction of nonzero entries. In particular, we assume
\begin{align*}
\tS\in \mS(\alpha, s) = \big\{ S\in \S^{d\times d}: \|S\|_{0, 1}\leq s, \|S\|_{0, \infty}\leq \alpha d\big\}
\end{align*}
for some integer $s$ and fraction $\alpha\in(0, 1)$. Furthermore, to make the low-rank component separable from the sum $\tS + \tR$, we require $\tR$ to be not too sparse. One way to ensure identifiability is to impose the incoherence condition \citep{Candes2007Sparsity}, {\red which prevents the information in column (or row) spaces of $\tR$ from being concentrated in few columns. The incoherence condition guarantees that the elements of $\tR$ are roughly of the same magnitude and are not spiky. It is commonly used in the literature on low-rank matrix recovery \citep{Chen2014Coherent, Chen2015Incoherence, Yi2016Fast}.}  Specifically, suppose $\tR = \tL\tXi\tLT$ is the reduced eigenvalue decomposition, where $\tL\in \mR^{d\times r}$ satisfies $\tLT\tL = I_r$ and $\tXi = \diag(\lambda^\tR_1, \ldots, \lambda^\tR_r)$. Then, we assume $\tL$ satisfies $\beta$-incoherence condition, that is,
\begin{align*}
\tL\in \mU(\beta) = \cbr{L\in\mR^{d\times r} \mid  \|L^\T\|_{2, \infty} \leq \rbr{{\beta r}/{d}}^{1/2} }.
\end{align*}
Without loss of generality, the eigenvalues are ordered so that, for some integer $r_1 \in \{0, \ldots, r\}$, $\sign(\lambda_i^\tR) = 1$ for $1\leq i\leq r_1$ and $\sign(\lambda_i^\tR) = -1$ for $r_1+1\leq i\leq r$. Here, $r_1$, so called the positive index of inertia of $\tR$, is unique by Sylvester's law of inertia \citep[cf. Theorem 4.5.8 in][]{Horn2013Matrix}, although eigenvalue decomposition is not.

\subsection{Two-stage algorithm}

We develop a two-stage algorithm to estimate the tuple $(\tS, \tU, \tLambda)$. We start from introducing the second stage. Given a suitably chosen initial point $(S^0, U^0, \Lambda^0)$, obtained by the first stage that we introduce later, we use the projected alternating gradient descent procedure to minimize the following nonconvex optimization problem
\begin{equation}\label{pro:3}
\begin{aligned}
\min_{S,\; U} \text{\ \ } & \mbL_n(S, U, \Lambda^0) + \frac{1}{2}\|U_1^\T U_2\|_F^2,\\
\text{subject to\ \ } & S\in \mS({\red \baralpha, \bars}), \quad U\in\mU(4\beta\|U^0\|_2^2),
\end{aligned}
\end{equation}
where $U = (U_1, U_2)$ with $U_1\in\mR^{d\times \hr_1}$, $U_2\in\mR^{d\times (r-\hr_1)}$, $\hr_1$ is the number of $+1$ entries of $\Lambda^0$, used as an estimate of $r_1$, and {\red $\baralpha$, $\bars$} are user-defined tuning parameters. The quadratic penalty in \eqref{pro:3} biases the components $U_1$, $U_2$ of the matrix $U$ to be orthogonal and can also be written as $\|U^\T U - \Lambda^0U^\T U\Lambda^0\|_F^2/16$.

Before we detail steps of the algorithm, we define two truncation operators that correspond to two different sparsity structures. For any integer $s$ and $A\in\mR^{d\times d}$, the hard-truncation operator $\mJ_s(\cdot) : \mR^{d\times d} \mapsto \mR^{d\times d}$ is defined as
\begin{align*}
[\mJ_s(A)]_{i, j} = \begin{cases}
A_{i, j} & \text{if\ } |A_{i, j}| \text{\ is one of the largest\ } s \text{\ elements of\ } A,\\
0 & \text{otherwise}.
\end{cases}
\end{align*}
For any $\alpha\in(0, 1)$, the dispersed-truncation operator $\mT_{\alpha}(\cdot) : \mR^{d\times d} \mapsto \mR^{d\times d}$ is defined as
\begin{align*}
[\mT_\alpha(A)]_{i, j} = \begin{cases}
A_{i, j} & \text{if\ } |A_{i, j}| \text{\ is one of the largest\ } \alpha d \text{\ elements for both\ } A_{i, \cdot} \text{\ and\ } A_{\cdot, j},\\
0 & \text{otherwise}.
\end{cases}
\end{align*}
In the above definitions, $\mJ_s(A)$ keeps the largest $s$ entries of $A$, while $\mT_\alpha(A)$ keeps the largest $\alpha$ fraction of entries in each row and column. Therefore, the operator $\mJ_s(\cdot)$ projects iterates to the constraint set $\|S\|_{0, 1}\leq s$, while $\mT_\alpha(\cdot)$ projects to the set $\|S\|_{0, \infty}\leq \alpha d$.

We summarize the projected alternating gradient descend procedure in Algorithm \ref{alg:2}. Both the sparse and low-rank components are updated, with the other component being fixed, by the gradient descent step with a constant step size, followed by a projection step. Explicit formulas for $\nabla_S\mbL_n$ and $\nabla_U\mbL_n$ are provided in \textsection \ref{appen:main:lem} in the Supplementary Material. The sign matrix $\Lambda^0$ is not updated in the algorithm. We will show later that, under suitable conditions, the first stage estimate consistently recovers $\tLambda$, that is, $\Lambda^0 = \tLambda$. Computationally, the update of the low-rank matrix in each iteration requires only updating the factor $U$, which can be done efficiently.

\begin{algorithm}[!tp]
	\vspace*{-6pt}
	\caption{Stage II: projected alternating gradient descent for solving \eqref{pro:3}.} \label{alg:2}
	\vspace*{-15pt}
	\begin{tabbing}
		\qquad\enspace Input: Sample covariance matrices $\hSigmax$, $\hSigmay$; Initial point tuple $(S^0, U^0, \Lambda^0)$; Step sizes $\eta_1$,\\
		\qquad\enspace \phantom{Input: } $\eta_2$; Tuning parameters $\baralpha$, $\bars$, $\beta$. \\
		\qquad\enspace For $k=0$ to $k=K-1$\\
		\qquad\qquad $S^{k + 1/2} = S^k - \eta_1\nabla_S\mbL_n(S^k, U^k, \Lambda^0)$;\\
		\qquad\qquad $S^{k + 1} = \mT_{\baralpha} \cbr{ \mJ_{\bars}(S^{k + 1/2}) }$;\\
		\qquad\qquad Let $\mC^k = \mU(4\beta\|U^k\|_2^2)$;\\
		\qquad\qquad $U^{k+1/2} = U^k - \eta_2\nabla_U\mbL_n(S^k, U^k, \Lambda^0) - \frac{\eta_2}{2}U^k(U^{k \T}U^k - \Lambda^0U^{k \T} U^k\Lambda^0)$;\\
		\qquad\qquad $U^{k+1} = \P_{\mC^k}(U^{k+1/2})$;\\
		\qquad\enspace Output $S^K$, $U^K$.
	\end{tabbing}
	
\end{algorithm}

{\red The projection operator $\P_{\mU(\beta)}(\cdot)$ can be computed in a closed form as}
\begin{align*}
\red
[\P_{\mU(\beta)}(U)]_{i, \cdot} = \begin{cases}
U_{i, \cdot} & \text{if\ } \|U_{i, \cdot}\|_2\leq (\beta r/d)^{1/2},\\
(\beta r/d)^{1/2}/\|U_{i, \cdot}\|_2\cdot U_{i, \cdot} & \text{otherwise}.
\end{cases}
\end{align*}
Next, we describe how to get a \textit{good} initial point, $(S^0, U^0, \Lambda^0)$, needed for Algorithm \ref{alg:2}.  The requirements on the initial point are presented in Theorem \ref{thm:1}. Our initial point is obtained from a rough estimator of $\tDelta$. Let $\hat \Delta^0 = (\ttSigmax)^{-1} - (\ttSigmay)^{-1}$, where $\ttSigmax = {n_X}/\rbr{n_X - d - 2}\hSigmax$ (similarly for $\ttSigmay$) is the scaled sample covariance matrix. The scaled covariance matrix, so called Kaufman-Hartlap correction \citep{Paz2015Improving}, is used for the initialization step so to have $\mE(\ttSigmax^{-1}) = \tOmegax$. By rescaling the sample covariance, we are able to show that $\|(\ttSigmax)^{-1} - \tOmegax\|_{\infty, \infty}\asymp (\log d/n_X)^{1/2}$ with high probability, leading to a better sample size compared to $\|(\hSigmax)^{-1} - \tOmegax\|_{\infty, \infty}\asymp d/n_X + (\log d/n_X)^{1/2}$. We obtain $S^0$ by truncating $\hat\Delta^0$. Next, we extract $r$ eigenvectors, corresponding to the top $r$ eigenvalues in magnitude of the residual matrix $R^0 = \hat \Delta^0-S^0$.  $U^0$ and $\Lambda^0$ are further derived from the reduced matrix. See Algorithm \ref{alg:1} for details. Theorem \ref{thm:2} shows that the positive index of inertia is correctly recovered by the initial step, $\Lambda^0 = \tLambda$, and $(S^0, U^0)$ lies in a sufficiently small neighborhood of $(\tS, \tU)$.

\begin{algorithm}[!tp]
	\vspace*{-6pt}
	\caption{Stage I: initialization.} \label{alg:1}
	\vspace*{-15pt}
	\begin{tabbing}
		\qquad\enspace Input: Scaled sample covariance matrices $\ttSigmax$, $\ttSigmay$; Tuning parameters $\halpha$, $\hs$, $r$, $\beta$.\\
		\qquad\enspace Let $\hDelta^0 = (\ttSigmax)^{-1} - (\ttSigmay)^{-1}$,  $S^0 = \mT_{\halpha}\{\mJ_\hs(\hDelta^0)\}$, and $R^0 = \hDelta^0 - S^0$;\\
		\qquad\enspace Compute $R^0 = L^0\Xi^0L^{0 \T}$ the eigenvalue decomposition of $R^0$, let $\Xi^0_r\in\mR^{r\times r}$ be the \\
		\qquad\enspace \text{\ \ \ \ } {\red diagonal matrix with} largest $r$ eigenvalues in magnitude and $L^0_r\in\mR^{d\times r}$ be the \\
		\qquad\enspace \text{\ \ \ \ } corresponding eigenvectors; \\
		\qquad\enspace Let $\hr_1 = |\{i\in[r]: [\Xi^0_r]_{i, i}>0\}|$, $\Lambda^0 = \diag(I_{\hr_1}, -I_{r - \hr_1})$, and $P^0$ be the permutation \\
		\qquad\enspace \text{\ \ \ \ } matrix such that $\sign(\Xi_r^0) = P^0\Lambda^0P^{0 \T}$;\\
		\qquad\enspace Let $\barU^0 = L_r^0|\Xi^0_r|^{1/2}P^0$ where $|\Xi^0_r|$ is computed elementwise; \\
		\qquad\enspace Let $U^0 = \P_{\mC}\rbr{\barU^0}$ with $\mC = \mU(4\beta\|\barU^0\|_2^2)$;\\
		\qquad\enspace Output $S^0$, $U^0$, $\Lambda^0$.
	\end{tabbing}
\end{algorithm}

Throughout the two-stage algorithm, we only compute the (reduced) eigenvalue decomposition once in the first stage. Therefore, it is computationally efficient compared to related convex approaches, mentioned in \textsection \ref{disc} in the Supplementary Material, where in each iteration one needs to compute an eigenvalue decomposition to update $R$.

In our experiments, we set $\baralpha = \halpha$ and $\bars = \hs$ and use cross-validation to select them together with $r$ and $\beta$. Our theory requires more stringent conditions on $\baralpha, \bars$ in Algorithm~\ref{alg:2} than on $\halpha, \hs$ in Algorithm~\ref{alg:1}, where we only require $\halpha\geq \alpha$ and $\hs\geq s$. See Theorems~\ref{thm:1} and \ref{thm:2}.

\section{Theoretical analysis}\label{sec:4}

We establish the convergence rate of iterates generated by Algorithm \ref{alg:2} by first assuming that the initial point $(S^0, U^0, \Lambda^0)$ lies in a suitable neighborhood around $(\tS, \tU, \tLambda)$. Next, we prove that the output of Algorithm~\ref{alg:1} satisfies requirements on the initial point with high probability. The convergence rate of Algorithm \ref{alg:2} consists of two parts: the statistical rate and algorithmic rate. The statistical rate appears due to the approximation of population loss by the empirical loss, and it depends on the sample size, dimension, and the problem parameters including the condition numbers of covariance matrices. The algorithmic rate characterizes the linear rate of convergence of the projected gradient descent iterates to a point that is within statistical error from the true parameters.

The convergence rate is established under the following two assumptions.

\begin{assumption}[Constraint sets]\label{ass:2}
	
	Let $\tDelta = \tS + \tR$ be the differential network and $\tR = \tL\tXi\tLT$ be the reduced eigenvalue decomposition of the rank-$r$ matrix $\tR$.  There exist $\alpha$, $\beta$ and $s$ such that $\tS\in\mS(\alpha, s)$ and $\tL\in\mU(\beta)$.
\end{assumption}

\begin{assumption}\label{ass:1}
	There exist $0 < \sigma^X_d \leq \sigma^X_1 < \infty$ and $0 < \sigma^Y_d \leq \sigma^Y_1 < \infty$ such that $\sigma^X_d I_d\preceq\tSigmax\preceq \sigma^X_1 I_d$ and $\sigma^Y_d I_d\preceq \tSigmay\preceq \sigma^Y_1 I_d$.
\end{assumption}

We start by defining the distance function that will be used to measure the convergence rate of the low-rank component. From the reduced eigenvalue decomposition of $\tR$, $\tR = \tL\tXi\tLT = \tU\tLambda\tUT$ with $\tLambda = \sign(\tXi) = \diag(I_{r_1}, -I_{r-r_1})$ and $\tU = \tL(\tXi\tLambda)^{1/2}$. While $\tLambda$ is uniquely characterized by the positive index of inertia $r_1$, $\tU$ is not unique in the sense that it is possible to have $\tU\tLambda \tUT = U\tLambda U^\T$ but $U\neq \tU$. We deal with this non-uniqueness issue by using the following distance function.

\begin{definition}[Distance function]\label{def:1}
	
	Given two matrices $U_1, U_2\in\mR^{d\times r}$ and an integer $r'\in \{0, \ldots, r\}$, we define ${\red \Pi_{r'}(U_1, U_2)} = \inf_{Q\in \mQ_{r'}^{r\times r}} \|U_1 - U_2 Q\|_F$, where
	\begin{align*}
	\mQ_{r'}^{r\times r} &= \cbr{Q\in\Q^{r\times r}: Q\Lambda Q^\T = \Lambda \text{ with } \Lambda = \diag(I_{r'}, -I_{r-r'}) } \\
	& = \cbr{Q\in\Q^{r\times r}: Q = \diag(Q_1, Q_2) \text{ with } Q_1\in\Q^{r'\times r'}, Q_2\in\Q^{r-r'\times r-r'}}.
	\end{align*}
\end{definition}
In the following, we will simply use $\Pi(\cdot, \cdot)$ to represent $\Pi_{r_1}(\cdot, \cdot)$ with $r_1$ being the positive inertia of $\tR$. Based on the following lemma, we see that ${\red \Pi(U, \tU)}$ measures $\|U\tLambda U^\T - \tU\tLambda\tUT\|_F$.

\begin{lemma}[Properties of $\Pi(\cdot, \cdot)$]\label{prop:1}
	
	Suppose $\tU\in\mR^{d\times r}$ has orthogonal columns and $\tLambda = \diag(I_{r_1}, I_{r-r_1})$. Let $\sigma_1$ ($\sigma_r$) be the largest (smallest) singular value of $\tU$ and let $U\in\mR^{d\times r}$.
	\begin{enumerate}[label=(\alph*),topsep=2pt]
		\setlength\itemsep{-0em}
		\item If $\Pi(U, \tU)\leq \sigma_1$, then $\|U\tLambda U^\T - \tU\tLambda\tUT\|_F\leq 3\sigma_1\Pi(U, \tU)$.
		\item If $\|U\tLambda U^\T - \tU\tLambda\tUT\|_2\leq \sigma_r^2/2$, then $\Pi(U, \tU)\leq \{\rbr{\surd{2}-1}^{1/2}\sigma_r\}^{-1}\|U\tLambda U^\T - \tU\tLambda\tUT\|_F$.
	\end{enumerate}
	
\end{lemma}

By Lemma \ref{prop:1}, $U\tLambda U^\T = \tU\tLambda\tUT \Longleftrightarrow \Pi^2(U, \tU) =0$. Thus, once we can correctly recover $\tLambda$, that is $\hat{\Lambda} = \tLambda$, the distance function in Definition \ref{def:1} is a reasonable surrogate for $\|\hR - \tR\|_F$, since $\|\hR - \tR\|_F = \|\hU\hat{\Lambda}\hU^\T - \tU\tLambda\tUT\|_F = {\red \|\hU\tLambda\hU^\T - \tU\tLambda\tUT\|_F}\asymp \Pi(\hU, \tU)$.

Let $\sigma_1^{\tR} = \sigma_{\max}(\tR)$, $\sigma_r^{\tR} = \sigma_{\min}(\tR)$ and define the condition numbers $\kappa_X = \sigma^X_1/\sigma^X_d$, $\kappa_Y = \sigma^Y_1/\sigma^Y_d$, and $\kappa_\tR = \sigma_1^\tR/\sigma_r^\tR$. We further define the following quantities that depend only on the covariance matrices
\begin{align*}
T_1 = & \cbr{\frac{\kappa_X\kappa_Y\rbr{\|\tOmegay\|_1\|\tSigmax\|_1 + \|\tOmegax\|_1\|\tSigmay\|_1}}{\sigma_d^X\sigma_d^Y}}^2, T_2 = \rbr{\frac{1}{\sigma_d^X} + \frac{1}{\sigma_d^X}}^2, T_3 = \rbr{\frac{\|\tSigmax\|_1}{\sigma_1^X}}^2 + \rbr{\frac{\|\tSigmay\|_1}{\sigma_1^Y}}^2,\\ T_4=&\frac{(\sigma_d^X\sigma_d^Y)^2}{\kappa_X^4\kappa_Y^4\cbr{(\sigma_1^Y\|\tSigmax\|_1)^2 + (\sigma_1^X\|\tSigmay\|)^2}}, T_5 = \cbr{\|(\tOmegax)^{1/2}\|_1^2 + \|(\tOmegay)^{1/2}\|_1^2}^2, T_6 =  \rbr{\frac{\kappa_X}{\sigma_d^X} + \frac{\kappa_Y}{\sigma_d^Y}}^2.
\end{align*}
Finally, for $S\in\S^{d\times d}$ and $U\in\mR^{d\times r}$, we define the total error distance to be
\begin{align*}
TD(S, U) = {\|S - \tS\|_F^2}/{\sigma_1^\tR} + \Pi^2(U, \tU).
\end{align*}
The error for the sparse component is scaled by $\sigma_1^\tR$ in order to have the two error terms on the same scale, based on the first part of Lemma \ref{prop:1}.  With this, we have the following result on the convergence of iterates obtained by Algorithm \ref{alg:2}.

\begin{theorem}[Convergence of Algorithm \ref{alg:2}]\label{thm:1}
	
	Suppose Assumptions \ref{ass:2} and \ref{ass:1} hold. Furthermore, suppose the following conditions hold: (a) sample size
	\begin{equation}\label{cond:sample1}
	\rbr{n_X\wedge n_Y} \geq C_1\cbr{\frac{d\log d}{T_3\beta}\vee\frac{\rbr{\kappa_X\kappa_Y}^4 \rbr{T_1\cdot s\log d + T_2\cdot rd}}{(\sigma_r^\tR)^2}},
	\end{equation}
	and sparsity proportion $\alpha \leq {c_1T_4}/\rbr{\beta r \kappa_\tR}$; (b) step sizes $\eta_1 \leq c_2/\rbr{\sigma_1^X\sigma_1^Y\kappa_X\kappa_Y}$, $\eta_2 = {c_3\eta_1}/{\sigma_1^\tR}$, and tuning parameters ${\red 2(\bars/s) - 1\geq \baralpha/\alpha}\geq C_2\rbr{\kappa_X\kappa_Y}^4$; (c) initialization point $\Lambda^0 = \tLambda$, $S^0 \in\S^{d\times d}$, $U^0\in\mU(9\beta\sigma_1^\tR)$ with $TD(S^0, U^0)\leq {c_4\sigma_r^\tR}/{\rbr{\kappa_X\kappa_Y}^2}$; then the iterates $(S^k, U^k)$ of Algorithm~\ref{alg:2} satisfy $S^k\in\S^{d\times d}$ and
	\begin{align}\label{result:1}
	TD(S^k, U^k)\leq \rbr{1 - \frac{c_5}{\kappa_X^2\kappa_Y^2\kappa_\tR}}^kTD(S^0, U^0) +
	\frac{C_3\kappa_X^2\kappa_Y^2}{\sigma_r^\tR}\cdot \frac{T_1\cdot s\log d + T_2\cdot rd}{n_X\wedge n_Y},
	\end{align}
	with probability at least $1 - C_4/d^2$ for some fixed constants $\rbr{C_i}_{i = 1}^4$ sufficiently large and $\rbr{c_i}_{i=1}^5$ sufficiently small.
	
\end{theorem}

The two terms in \eqref{result:1} correspond to the algorithmic and the statistical rate of convergence, respectively. The statistical error is of the order $O\rbr{\rbr{s \log d + rd}/\rbr{n_X\wedge n_Y}}$, which matches the minimax optimal rate \citep{Chandrasekaran2012Latent}. In particular, the term $O\rbr{{s\log d}/\rbr{n_X\wedge n_Y}}$ corresponds to the statistical error of estimating $\tS$, while $O\rbr{{rd}/\rbr{n_X\wedge n_Y}}$ corresponds to the statistical error of estimating $\tR$. We stress that the condition on $\alpha$ is common in related literature. For example, \cite{Yi2016Fast} requires $\alpha\lesssim 1/\beta r (\kappa_{\tR})^2$, which is stronger than our condition in terms of the power of $\kappa_{\tR}$; \cite{Zhang2018Unified} requires $\alpha\lesssim 1/\beta r \kappa_{\tR}$, which is comparable with ours. {\red Under the condition on $\alpha$, we have $\baralpha<1$.} The sample complexity requirement in \eqref{cond:sample1} has an extra $d\log d/\beta$ term compared to typical results in robust estimation (see Corollary 4.11 and Corollary 4.13 in \cite{Zhang2018Unified} for results in robust matrix sensing and robust principal component analysis). This increased sample complexity is common in estimation of latent variable Gaussian graphical models. For example, \cite{Xu2017Speeding} requires $n_X\gtrsim d^2$ to show convergence of $\|\hSigmax\|_1$. Theorem~\ref{thm:1} improves the sample size requirement to $d\log d$. In \eqref{loss:emp:SULam}, we need to control the low-rank components $\hSigmax\tU$ (and $\hSigmay\tU$) and the large sample size guarantees that the incoherence condition can transfer from $\tU$ to $\hSigmax\tU$. Furthermore, the covariance matrices $\hSigmax$ and $\hSigmay$ work as design matrices in \eqref{loss:emp:SULam} and bring additional challenges compared to robust estimation problems. The design matrix in robust principal component analysis is identity, while in robust matrix sensing its expectation is also identity. Thus, their loss functions all satisfy Condition 4.4 in \cite{Zhang2018Unified}, which is not the case for \eqref{loss:emp:SULam}. {\red Without Condition 4.4, their proof strategy fails to show the convergence of alternating gradient descent. By direct analysis, we first establish what conditions we need on $\hSigmax\tU$ and $\hSigmay\tU$, and then show that these conditions hold under incoherence condition on $\tU$.} Finally, we observe that the algorithmic error decreases exponentially and, after $O\rbr{\log\cbr{{n_X\wedge n_Y}/\rbr{s\log d + rd}}}$ iterations, the statistical error is the dominant term.

Next, we show that the output $(S^0, U^0, \Lambda^0)$ of Algorithm \ref{alg:1} satisfies requirements on the initialization point of Algorithm \ref{alg:2} presented in condition (c) in Theorem \ref{thm:1}. The requirement that $S^0\in\S^{d\times d}$ is easy to achieve.  The following lemma suggests that $\Lambda^0 = \tLambda$ is implied by an upper bound on $\|R^0 - \tR\|_2$, which further connects to the upper bound on $TD(S^0, U^0)$ by Lemma~\ref{prop:1}.

\begin{lemma}\label{lem:1}
	
	For any $R\in\S^{d\times d}$, let $R = L\Xi L^\T$ be the eigenvalue decomposition. Let $\Xi_r \in\mR^{r\times r}$ be the diagonal matrix with $r$ largest entries of $\Xi$ in magnitude, and let $\hr_1$ be the number of positive entries of $\Xi_r$.  If $\|R - \tR\|_2\leq \sigma_r^\tR/3$, then $\hr_1 = r_1$ and $\Lambda_r = \diag(I_{\hr_1}, -I_{r - \hr_1}) = \tLambda$.
	
\end{lemma}

The next theorem shows the sample complexity under which the conditions on the initial point are satisfied and $\|R^0 - \tR\|_2 \leq \sigma_r^\tR/3$, which implies $\Lambda^0 = \tLambda$, using Lemma \ref{lem:1}.

\begin{theorem}[Initialization]\label{thm:2}
	
	Suppose Assumptions \ref{ass:2} and \ref{ass:1} hold. If $\halpha\geq \alpha$, $\hs\geq s$, the sample sizes and dimension satisfy
	\begin{align}\label{cond:sample2}
	(n_X\wedge n_Y) \geq \frac{C_1\rbr{T_5\hs\log d + T_6 d}}{(\sigma_r^\tR)^2}, \quad d \geq C_2\beta{\hs}^{1/2} r \kappa_\tR,
	\end{align}
	then $S^0\in\S^{d\times d}$, $\|R^0 - \tR\|_2\leq \sigma_r^\tR/4$, $U^0 \in \mU(9\beta\sigma_1^\tR)$, and
	\begin{align*}
	TD(S^0, U^0)\leq C_3\cbr{\frac{r \rbr{T_5\cdot \hs\log d + T_6\cdot d}}{\sigma_r^\tR(n_X\wedge n_Y)} + \frac{\hs\beta^2r^3\kappa_\tR\sigma_1^\tR}{d^2}}
	\end{align*}
	with probability $1 - C_4/d^2$ for some fixed constants $\rbr{C_1}_{i=1}^4$ sufficiently large. Furthermore, if $\hs\asymp s$,
	\begin{align}\label{cond:sample3}
	(n_X\wedge n_Y)\gtrsim \frac{r\kappa_X^2\kappa_Y^2}{(\sigma_r^\tR)^2}\big(T_5\cdot s\log d + T_6\cdot d\big), \text{\ \ \ } d\gtrsim \beta s^{1/2} r^{3/2}\kappa_\tR\kappa_X\kappa_Y,
	\end{align}
	then $TD(S^0, U^0)\lesssim{\sigma_r^\tR}/\rbr{\kappa_X\kappa_Y}^2$.
	
\end{theorem}

From Theorem~\ref{thm:2}, the requirement for the initial point is satisfied under \eqref{cond:sample3}. The sample complexity required for initialization, $O((rs\log d + rd)/(\sigma_r^\tR)^2)$, is smaller than the one for convergence, $O(d\log d/\beta + (s\log d + rd)/(\sigma_r^\tR)^2)$.  When $d\gtrsim \beta s^{1/2}r^{3/2}\kappa_\tR\kappa_X\kappa_Y$ and $\alpha\lesssim T_4/(\beta r\kappa_\tR)$, combining Theorems~\ref{thm:1} and \ref{thm:2} shows that the iterates generated by two-stage algorithm converge linearly to a point with an unavoidable minimax optimal statistical error.

{\red We briefly discuss exact recovery of the support of $\tS$ and the rank of $\tR$. Throughout the paper, we assume that the rank $r$ of $\tR$ is known. This assumption is commonly used in the literature on alternating gradient descent for low-rank matrix recovery \citep[e.g.,][]{Yi2016Fast, Xu2017Speeding, Zhang2018Unified}. The rank $r$ is used to truncate the eigenvalues of $R^0$ and to choose the number of columns of the iterates $U$ in Algorithm~\ref{alg:2}. However, we note that the rank $r$ can be exactly recovered under a suitable assumption on the signal strength, $\sigma_r^\tR$. After dropping higher order terms, Theorem \ref{thm:2} shows that $\|R^0 - \tR\|_2\lesssim \{d/(n_X\wedge n_Y)\}^{1/2}$.  Therefore, if $\sigma_r^\tR \gtrsim 2\{d/(n_X\wedge n_Y)\}^{1/2}$, one can recover $r$ by thresholding small eigenvalues of $R^0$. From the proof of Lemma~\ref{lem:6} in the Supplementary Material, $\|S^0 - \tS\|_{\infty, \infty}\leq \{\log d/(n_x\wedge n_Y)\}^{1/2}$, which allows us to recover the support of $\tS$ by  thresholding elements of $S^0$ that are smaller in magnitude than $\{\log d/(n_x\wedge n_Y)\}^{1/2}$, if the nonzero elements of $\tS$ are bigger  than $2\{\log d/(n_x\wedge n_Y)\}^{1/2}$ in magnitude. Finally, if the support set of $\tS$ is consistently estimated, then $\alpha$ and $s$ can be estimated as well.  Therefore, under suitable assumptions on the signal strength, $\alpha$, $s$, and $r$, are all consistently estimable. Similar signal strength assumptions are also needed even for convex approaches to exactly recover the sparsity and rank \citep{Chandrasekaran2012Latent, Zhao2014Direct}.
}

\section{Simulations}\label{sec:5}

\subsection{Data generation and implementation details}\label{sec:5.1}

We compare the performance of our estimator with two procedures that directly learn the differential network under the sparsity assumption, the $\ell_1$-minimization \citep{Zhao2014Direct} and $\ell_1$-penalized quadratic loss \citep{Yuan2017Differential}, and two procedures that separately learn latent variable Gaussian graphical models, sparse plus low-rank penalized Gaussian likelihood \eqref{convex:1} \citep{Chandrasekaran2012Latent} and constrained Gaussian likelihood \citep{Xu2017Speeding}. Table~\ref{tab:1} summarizes the procedures. {\red In the Supplementary Material, we provide additional simulation results, including comparison with alternative convex approaches.}

\begin{table}[!h]
	\caption{Competing methods}\label{tab:1}
	\small
	\begin{tabular}{cccc}
		Abbr. & Reference & Type & Setup \\
		M1 & \cite{Zhao2014Direct} & joint, convex & differential network is sparse\\
		M2 & \cite{Yuan2017Differential}& joint, convex & differential network is sparse\\
		M3 & \cite{Chandrasekaran2012Latent} & separate, convex & single network is sparse $+$ low-rank\\
		M4 & \cite{Xu2017Speeding} & separate, nonconvex & single network is sparse $+$ low-rank\\
		M* & present paper & joint, nonconvex & differential network is sparse $+$ low-rank\\
	\end{tabular}
\end{table}

Data are generated from the latent variable Gaussian graphical model \eqref{lvggm} described in \textsection\ref{sec:2.2}. We set $\mu_{X_O}^\star = \mu_{X_H}^\star=\mu_{Y_O}^\star=\mu_{Y_H}^\star = 0$. The blocks of $\Omega^\star$ are generated separately. For $\Omega_{OO}^\star \in\mR^{d\times d}$, we set diagonal entries to be one and, following \citet{Xia2015Testing}, off-diagonal entries to be generated according to one of the following four models.
\begin{enumerate}[label={ Model \arabic*:}, leftmargin=45pt, topsep = 2pt]
	\setlength\itemsep{-0.1em}
	\item $(\Omega_{OO}^{\star(1)})_{i, i+1} =(\Omega_{OO}^{\star(1)})_{i+1, i} = 0.6$, $(\Omega_{OO}^{\star(1)})_{i, i+2} =(\Omega_{OO}^{\star(1)})_{i+2, i} = 0.3$;
	\item $(\Omega_{OO}^{\star(2)})_{i, j} =(\Omega_{OO}^{\star(2)})_{j, i} = 0.5$ for $i = 10k-9$, $10k-6\leq j\leq 10k$, $1\leq k\leq d/10$;
	\item $(\Omega_{OO}^{\star(3)})_{i, j} =(\Omega_{OO}^{\star(3)})_{j, i} \sim 0.8\cdot \text{Bernoulli}(0.1)$ for $i+1\leq j\leq i+3$;
	\item $(\Omega_{OO}^{\star(4)})_{i, j} =(\Omega_{OO}^{\star(4)})_{j, i} \sim 0.5\cdot \text{Bernoulli}(0.5)$ for $i = 2k-1$, $2k\leq j\leq (2k+2)\wedge d$, $1\leq k\leq d/2$.
\end{enumerate}
The blocks $\Omega_{OH}^\star, \Omega_{HO}^\star$ are generated entrywise from the following mixture distribution
\begin{align*}
(\Omega_{HO}^\star)_{j, i} = (\Omega_{OH}^\star)_{i, j}\sim 0.1\cdot\delta_0 + 0.9\cdot\text{Uniform}(0.5, 1),
\quad  i=1,\ldots,d,\ j=1,\ldots,r,
\end{align*}
and $\Omega_{HH}^\star = I_r$. Combining the blocks, we get the following four models
\begin{align*}
\Omega^{\star(i)} = \begin{pmatrix}
\Omega_{OO}^{\star(i)} & \Omega^\star_{OH}\\
\Omega_{HO}^\star &\Omega_{HH}^\star
\end{pmatrix}, \text{\ \ } i = 1,2,3,4.
\end{align*}
Last, we let $\Sigma^\star_i = [D^{1/2}\{\Omega^{\star(i)} + ( \iota_i + 1)I_{d+r}\}D^{1/2}]^{-1}$, where $\iota_i = \abr{\min\cbr{\eig(\Omega^{\star(i)})}}$ and $D\in\mR^{(d+r)\times (d+r)}$ is a diagonal scaling matrix with $D_{i,i}\sim \text{Uniform}(0.5, 2.5)$. In our models, each latent variable is connected to roughly $90\%$ of observed covariates and hence the effect of latent variables is spread-out and {\red the corresponding low-rank matrix is incoherent \citep{Chandrasekaran2012Latent}.} We generate $X$ using $\Sigma^\star_1$ and denote it as the control group, while generate $Y$ using $\Sigma^\star_i$, $i = 2,3,4$, and denote it as the test $i-1$ group. Under this generation process, both $X_O$ and $Y_O$ have precision matrices with sparse plus low-rank structure.

Throughout the simulations, we set the sample size equal for both groups, $n_X = n_Y = n$.  For each combination of the tuple $(n, d, r)$, we generate a training and a validation set with sample size $n$. For each method, we choose the corresponding tuning parameters that minimize the empirical loss $\mL_n(\hS,\hR)$ on the validation set {\red (alternative loss functions are discussed in the Supplementary Material, see \textsection \ref{sec:alter:loss})}. We measure the performance by $\|\hS - \tS\|_F$ and $\|\hDelta - \Delta^\star\|_F/\surd{\sigma_{\max}(\tR)}$, where the latter is used as a surrogate for the total error distance $TD(\hS, \hU)$. Errors are computed on test sets with the same sample size based on 40 independent runs. For our method, the step sizes are set as $\eta_1 = 0.5$, $\eta_2 = \eta_1/\sigma_{\max}^2(U^0)$, where $U^0$ is the output of the initialization step; the sparsity proportion $\baralpha$ ($=\halpha$) is chosen from $\{0.01, 0.03, 0.05, 0.1, 0.3, 0.5, 0.8\}$ and $\bars$ ($=\hs$) from $\{2d, 4d, 6d, 15d, 25d, 30d\}$; the rank used in Algorithm \ref{alg:2} and \ref{alg:1} is chosen from $\{0,1,2,3,4\}$; and the incoherence parameter $\beta$ is chosen from $\{1,3\}$. For methods of \cite{Zhao2014Direct} and \cite{Yuan2017Differential}, we {\red use the loss function \eqref{loss:popu:SR} to choose among $5$ different $\lambda$ values, which denote} tuning parameters in their papers and are generated automatically by their packages. For the method of \citet{Chandrasekaran2012Latent}, we use the implementation in \citet{Ma2013Alternating}, where we greedily choose the tuning parameters $\alpha\in\{0.01, 0.05, 0.1\}$ and $\beta\in\{0.15, 0.25, 0.35\}$ (see (2.1) in \citet{Ma2013Alternating}), {\red while other parameters including the step size, augmented Lagrange multiplier, and initialization are kept as in their implementation.} For the method of \citet{Xu2017Speeding}, we select the rank and sparsity in the same way as for our method, while the other parameters are kept as in \citet{Xu2017Speeding} as well.

\afterpage{\begin{rotatepage}
		\begin{landscape}
			\begin{table}[p]
				\centering
				\caption{Simulation results for five algorithms. The estimation errors of the differential network and its sparse component are averaged over 40 independent runs, with standard error given in parentheses. The control group is generated by covariance $\Sigma_1^\star$ while the test $i$ group is generated by $\Sigma_{i+1}^\star$ for $i = 1,2,3$. Throughout the table, the smallest error under the same setup is highlighted.}\label{tab:3}
				\begin{tabular}{*{7}{c}}
					\Xhline{4\arrayrulewidth}
					\hline
					& \multicolumn{2}{c}{Control - Test 1} & \multicolumn{2}{c}{Control - Test 2} & \multicolumn{2}{c}{Control - Test 3}\\
					\Xhline{4\arrayrulewidth}
					\hline
					\\[-0.9em]
					Method & $\|\hS - S^\star\|_F$& $\frac{1}{\surd{\sigma_{\max}(R^\star)}}\|\hDelta - \Delta^\star\|_F$ & $\|\hS - S^\star\|_F$&  $\frac{1}{\surd{\sigma_{\max}(R^\star)}}\|\hDelta - \Delta^\star\|_F$ & $\|\hS - S^\star\|_F$ &  $\frac{1}{\surd{\sigma_{\max}(R^\star)}}\|\hDelta - \Delta^\star\|_F$\\
					\hline
					\Xhline{4\arrayrulewidth}
					& \multicolumn{6}{c}{$n = 1000$, $d = 100$, $r = 1$}\\
					\hline
					M* & \textbf{20.02}(0.56) & 10.35(0.45) & \textbf{18.73}(0.71) & 9.33(0.53) & \textbf{18.59}(0.59) &  \textbf{7.94}(0.20)\\
					\hline
					M1 & 26.40(0.67) & 11.18(0.27) & 27.58(0.93)  & 11.26(0.37) & 30.22(0.67) & 12.07(0.26) \\
					\hline
					M2 & 30.05(0.32) & 12.48(0.14) & 31.04(0.53) & 12.41(0.21) & 32.77(0.46) & 12.75(0.18)\\
					\hline
					M3 & 22.49(0.35) & \textbf{9.52}(0.14) & 22.54(0.44) & \textbf{9.23}(0.17) & 22.91(0.47) & 9.18(0.18) \\
					\hline
					M4 & 33.72(0.61) & 14.16(0.26) & 33.62(0.63) &  13.61(0.26) & 34.45(0.58) & 13.62(0.23)\\
					\hline
					\Xhline{4\arrayrulewidth}
					& \multicolumn{6}{c}{$n = 10000$, $d = 100$, $r = 2$}\\
					\hline
					M* & \textbf{12.55}(0.35) &  \textbf{4.87}(0.13) & \textbf{11.10}(0.38) &  \textbf{4.52}(0.14)  & \textbf{10.61}(0.38) & \textbf{4.37}(0.15) \\
					\hline
					M1 & 39.50(0.87) & 14.91(0.33) & 50.09(0.36) &  19.49(0.14) & 37.64(0.56) & 14.82(0.22)  \\
					\hline
					M2 & 27.86(0.25) & 10.41(0.09) &  32.99(0.22) & 12.77(0.09) & 29.58(0.22) & 11.56(0.09) \\
					\hline
					M3 &  30.54(0.17) &  11.51(0.06)  & 34.11(0.20) & 13.27(0.08) & 31.80(0.14) & 12.47(0.06) \\
					\hline
					M4 & 18.88(0.19) & 6.58(0.07) & 17.44(0.21) &  6.44(0.09) & 14.63(0.30) & 5.60(0.11)\\
					\hline
					\Xhline{4\arrayrulewidth}
					& \multicolumn{6}{c}{$n = 200$, $d = 50$, $r = 0$}\\
					\hline
					M* & 11.40(0.41) & 11.40(0.41) & 11.73(0.24) & 11.73(0.24) & \textbf{9.86}(0.44) & \textbf{9.86}(0.44) \\
					\hline
					M1 & \textbf{10.88}(0.33) & 10.88(0.33) &  11.92(0.27) &  11.92(0.27) & 10.64(0.25) & 10.64(0.25) \\
					\hline
					M2 & 11.37(0.54) & 11.37(0.54) & \textbf{10.81}(0.40) & \textbf{10.81}(0.40) & 10.37(0.38) & 10.37(0.38) \\
					\hline
					M3 & 11.04(0.16) & \textbf{10.85}(0.17) & 11.23(0.18) &  10.86(0.17) & 10.48(0.20) & 10.37(0.21)\\
					\hline
					M4 & 13.51(0.60) & 13.51(0.60) & 14.79(0.65) &  14.79(0.65) &  12.71(0.67) & 12.71(0.67)\\
					\Xhline{4\arrayrulewidth}
					\hline
					\multicolumn{7}{l}{
						\multirow{3}{*}{
							\parbox{47pc}{
								M*, the proposed method; M1, $\ell_1$-minimization in \cite{Zhao2014Direct}; M2, $\ell_1$-penalized quadratic loss in \cite{Yuan2017Differential};
								M3, penalized Gaussian likelihood in \cite{Chandrasekaran2012Latent}; M4, constrained Gaussian likelihood in \cite{Xu2017Speeding};
								detailed descriptions of each method are given in Table \ref{tab:1} and the choice of tuning parameters is discussed
								in \textsection\ref{sec:5.1}.
							}
						}
					}
				\end{tabular}
				
			\end{table}
		\end{landscape}
		
\end{rotatepage}}

\subsection{Results}

Simulation results are summarized in Table~\ref{tab:3}.  We see that our method outperforms other methods when $r = 2$, corresponding to the case where $\rank(\tR) = 4$ as $\tR$ is the difference of two low-rank components. When $r= 1$, \cite{Chandrasekaran2012Latent} is comparable with our method on the first two data generating models, while our method compares favorably in the third case. When $r=0$, there are no latent variable and our method is comparable to methods of \cite{Zhao2014Direct} and \cite{Yuan2017Differential} that are specifically designed for sparse differential network estimation without considering latent variables. In comparison, the approach of \cite{Chandrasekaran2012Latent} misestimates the low-rank component. Overall, the proposed nonconvex method accurately estimates both the low-rank and sparse components at a low computational cost. {\red In the Supplementary Material, we show that when the differential network has the sparse plus low-rank structure, while the group specific precision matrices do not have any structure, our method outperforms all the competitors significantly.}

Figure~\ref{ComS} and \ref{ComL} illustrate the statistical rate of convergence by plotting  $\|\hS - \tS\|_F$ versus $(d\log d/n)^{1/2}$ and  $\|\hR - \tR\|_F$ versus $(rd/n)^{1/2}$, respectively.  {\red We set $(d, r) = (50, 0), (100, 1), (150, 2)$ for each case and vary $n$ only.} Although the estimation errors for $\tS$ and $\tR$ are combined in Theorem~\ref{thm:1}, we expect a linear increasing trend in both figures since ${d\log d/n}\asymp {rd/n}$. In the Supplementary Material, we illustrate that the rank and the positive index of inertia are consistently selected by cross-validation.

\begin{figure}[t]
	\centering
	
	\subfigure[Statistical rate of convergence of estimating $\tS$. From	left to right, $(d, r) = (50, 0), (100, 1), (150, 2)$.]
	{\label{ComS}\includegraphics[width=150mm]{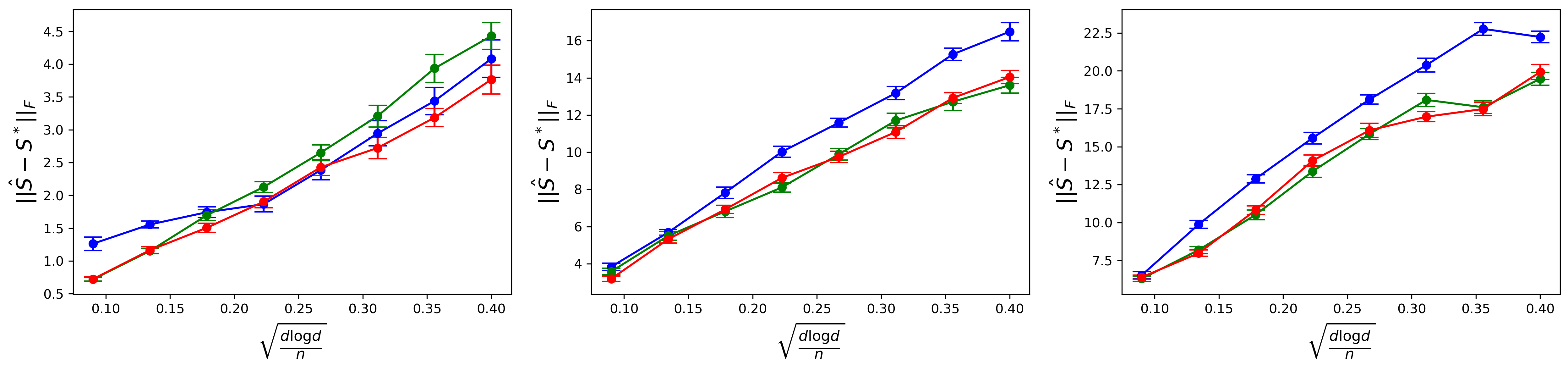}}
	\subfigure[Statistical rate of convergence of estimating $\tR$. The left panel corresponds to $(d, r) = (100, 1)$, while the right panel corresponds to $(d,r) = (150, 2)$.]
	{\label{ComL}\includegraphics[width=100mm]{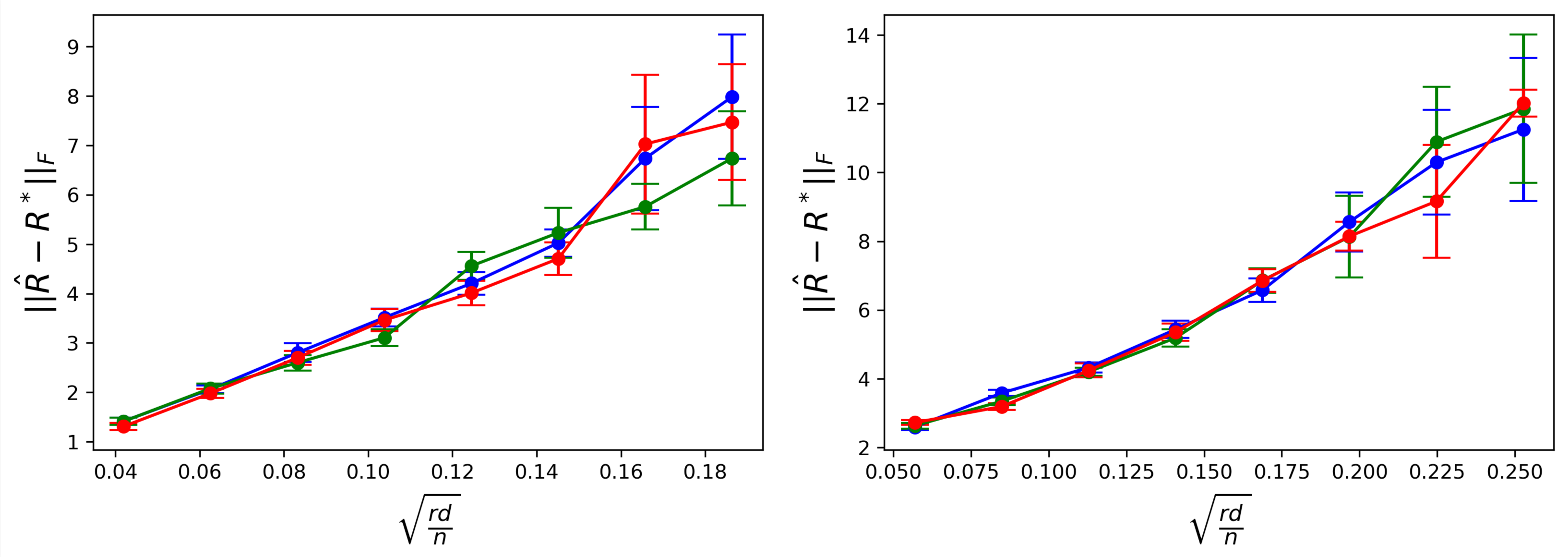}}	

	\includegraphics[width=5.5cm,height=0.35cm]{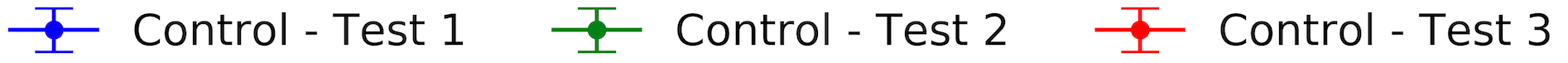}
	
	\caption{Statistical rate of convergence. All trends in figures increase linearly, which validates the results in Theorem~\ref{thm:1}.}\label{fig:1}
\end{figure}

\section{Application to fMRI functional connectivity}

We apply our method to the task of estimating differential brain functional connectivity from functional Magnetic Resonance Imaging (fMRI). In particular, we analyze the Center for Biomedical Research Excellence (COBRE) dataset, which is publicly available in \texttt{nilearn} package in Python \citep{Abraham2014Machine}. This dataset includes fMRI data from 146 subjects across two groups: 74 subjects are healthy controls and 72 subjects are diagnosed with schizophrenia. Each subject data includes resting-state fMRI time series with 150 samples. We remove time points with excessive motion as recommended by standard analyses, and apply Harvard-Oxford Atlas to automatically generate 48 regions of interests. This dataset has been carefully analyzed using the NeuroImaging Analysis Kit\footnote{\url{https://github.com/SIMEXP/niak}}.

We first estimate the differential network between the schizophrenia and control groups. We collect all fMRI series in one group across all subjects, thus assuming individuals in the same group share the same brain functional connectivity. Equivalently, we simply stack all time series from the subjects together to obtain one dataset for each group. The proposed approach and all the baseline methods remain the same as described in~\textsection\ref{sec:5}. The sparse plus low-rank decomposition is reasonable for estimating the differential network as it considers most pressing demographic confounders such as age and gender, {\red and the number of confounders is assumed to be small compared with the number of nodes in a brain, which is a conventional setup in fMRI study \citep{Greve2013survey, Geng2019Partially}.} The sparse component of the differential network is the parameter of scientific interest.

The estimated sparse component is reported in Figure~\ref{fig:2}, where each region corresponds to a vertex, each edge corresponds to an entry of the precision matrix, and the color corresponds to the magnitude of entries. Since Harvard-Oxford Atlas is a 3D parcellation atlas with lateralized labels, we show the detected connectomes in the left hemisphere only. We see that \cite{Zhao2014Direct} approach fails to recover a clear pattern; \cite{Chandrasekaran2012Latent} recovers one negative edge in \textit{Central Opercular Cortex} and one negative edge in \textit{Middle Frontal Gyrus}; our method, together with methods in \cite{Yuan2017Differential} and \cite{Xu2017Speeding}, show that the sparse network has two obvious edges, one positive and one negative, in \textit{Central Opercular Cortex} area, which is also consistent with some recent analysis that also discovered \textit{Central Opercular Cortex} is one of regions differs the most for the schizophrenia (\cite{Sheffield2015Fronto,	Geng2019Partially}). Upon closer analysis, we find that the network estimated by the proposed method has {\red smaller quadratic loss \eqref{loss:popu:SR} on the test sets. We let $\mL_n(\hDelta)$ denote the empirical test loss, where $\hDelta$ is the estimator and the covariance matrices are calculated on the test set. Our estimator has the test loss $\mL_n(\hDelta) = -3.64$, with $\|\nabla\mL_n(\hDelta)\|_{\infty,\infty} = 0.11$ and $\|\nabla\mL_n(\hDelta)\|_F = 1.56$. All three quantities are smaller than other methods, though they are not designed for minimizing \eqref{loss:popu:SR}.}

\begin{figure}[t]
	\centering     
	\setcounter{subfigure}{0}
	\subfigure[M1]{\label{Zhao}\includegraphics[width=73mm]{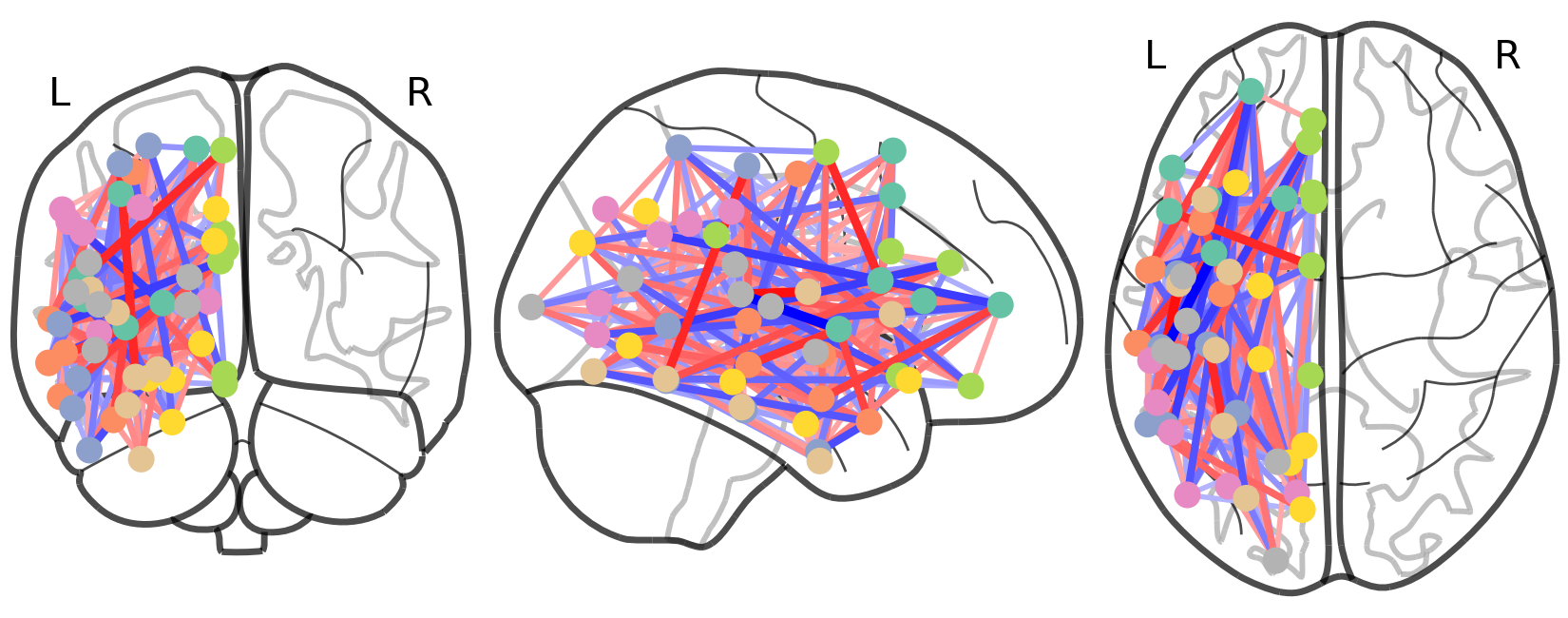}}
	\subfigure[M2]{\label{Yuan}\includegraphics[width=73mm]{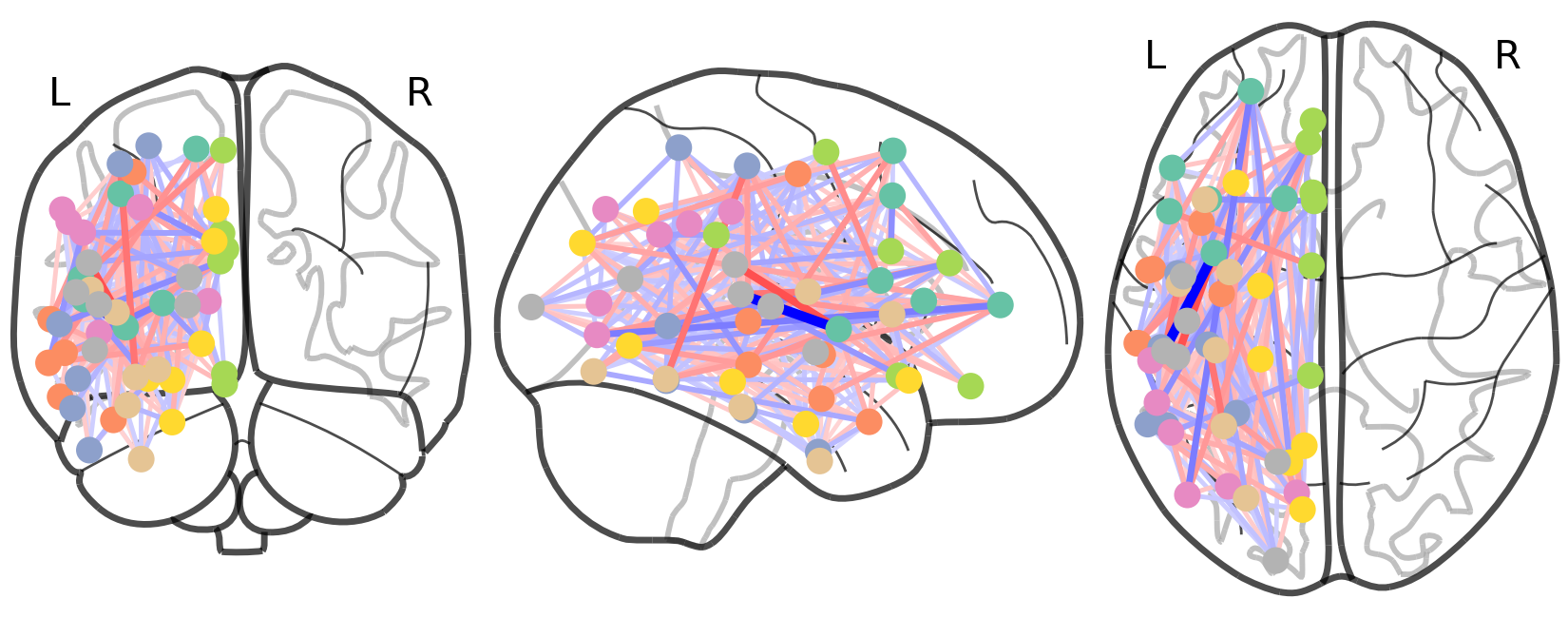}}
	
	\subfigure[M3]{\label{Venket}\includegraphics[width=73mm]{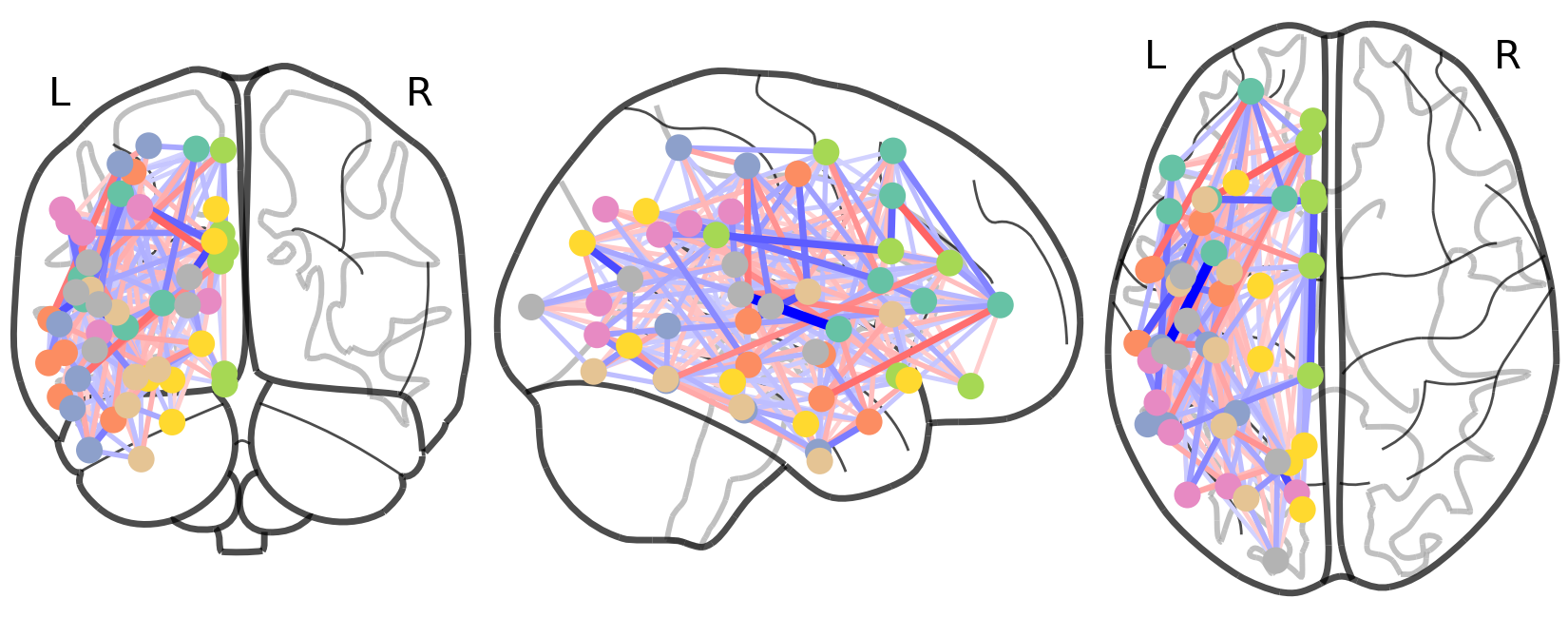}}
	\subfigure[M4]{\label{Gu}\includegraphics[width=73mm]{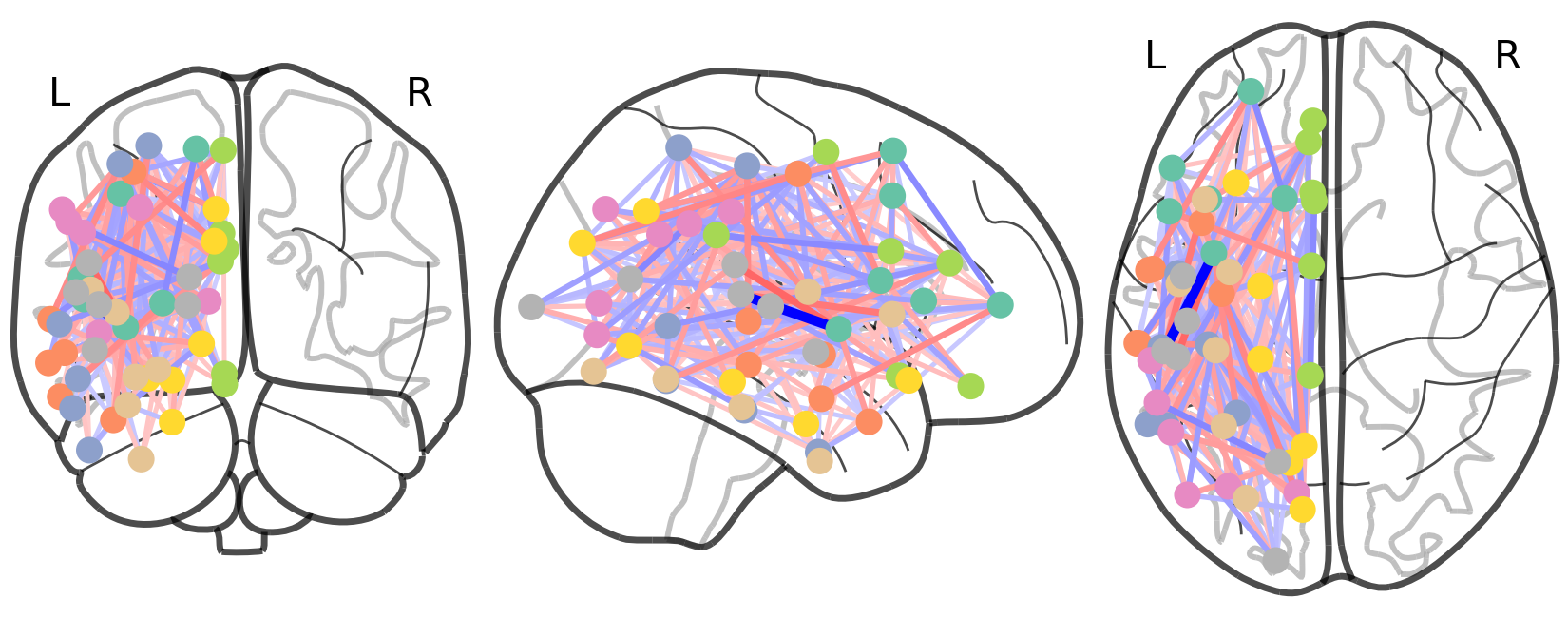}}
	
	\subfigure[M*]{\label{my}\includegraphics[width=95mm]{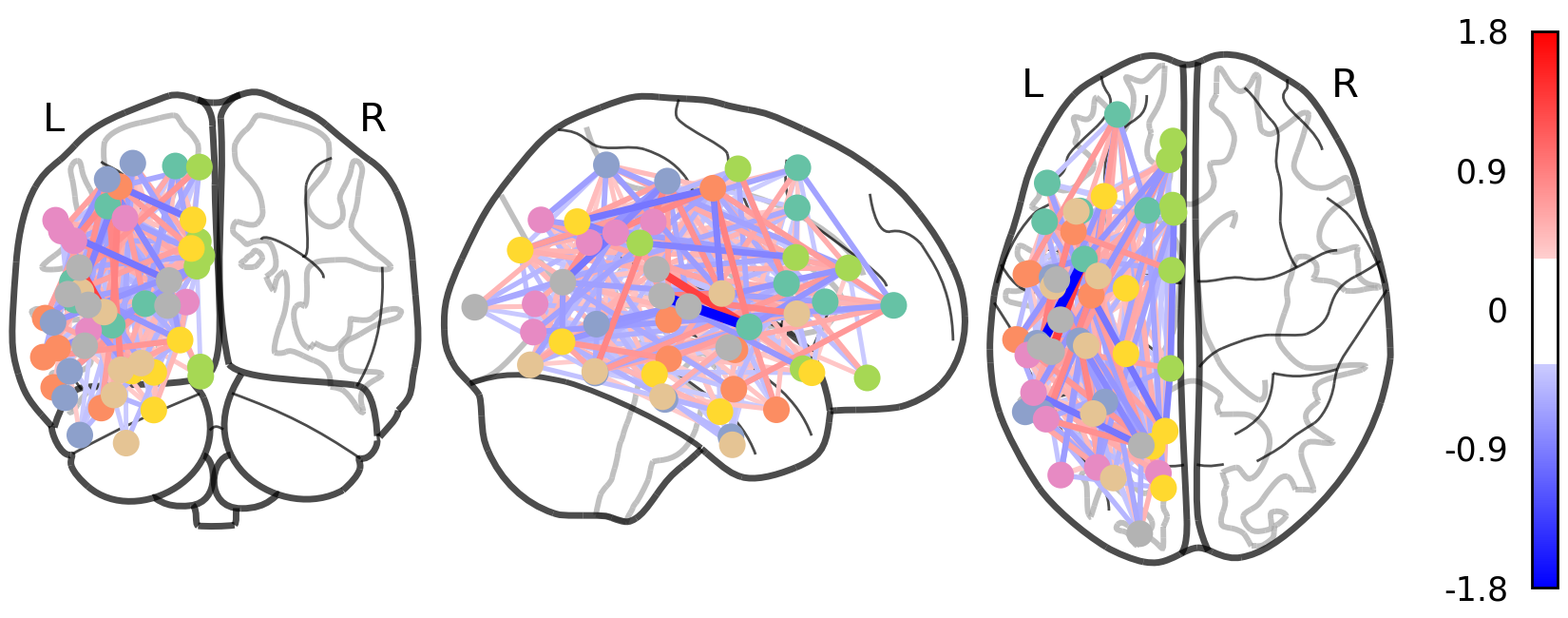}}
	
	\caption{Glass brains for the estimated sparse component of the differential network. }
	\label{fig:2}
\end{figure}

To further quantitatively validate our claims, we consider an individual-level analysis. In particular, we select 10 subjects from each group and consider the 190 possible pairs among them -- 100 out of 190 pairs are across-group while the remaining 90 pairs are within-group. We estimate the differential network for each pair and calculate $\|\hS\|_F$. Based on the group differences, one expects the sparse differential network for within-group pairs to have smaller norms than across-group pairs. Applying an unpaired two-sample $t$ test, the p-value for the proposed method is $0.09$ while greater than $0.16$ for M1 to M4. This further validates that our method also outperforms other methods at the individual level.

\section{Discussion}\label{sec:6}

We study the estimation of differential networks in the setting where the effects of the latent variables are diffused across all the observed variables leading to a low-rank component, and where the low-rank component of the difference satisfies an incoherence condition. In this setting, we are able to estimate the subspace spanned by the unobserved variables. {\red Extending our approach to identify the difference in the complete connectivity of the graph, which includes latent variables, is of additional interest. \citet{Vinyes2018Learning} studied the problem of identification and estimation of the complete connectivity of the graph in the presence of latent variables using a carefully designed convex penalty. In the limit of an infinite amount of data, under suitable assumptions, their procedure is able to identify the complete graph structure. However, finite sample properties of this procedure are not known. High-dimensional setting presents several challenges. First, additional assumptions are needed on the low rank component, $\tR$, such as a sparsity assumption on the effect of latent variables on observed variables. However, such an assumption makes identification of parameters more difficult, since we need to be able to distinguish the low-rank component from the sparse component. This is an identification problem and it is a challenge for both convex and nonconvex approaches. Second, when estimating a differential network, the matrix $\tR$ is indefinite and, as a result, development of a new penalty is required. Third, the initialization step in our algorithm requires us to compute the inverse of the sample covariance matrix, which is rank deficient in a high-dimensional setting. Therefore, a suitable and computationally efficient initialization strategy needs to be developed in a high-dimensional setting. Finally, the gradient error $\|\nabla_S\bar{\mathcal{L}}_n(\tS, U^k, \tLambda) - \nabla_S\bar{\mathcal{L}}_n(\tS, \tU, \tLambda)\|_F$,  which is a key ingredient in the proof,  is well controlled only if $\hSigmax \tU$ (and $\hSigmay \tU$) satisfies the incoherence condition. Since the incoherence condition is imposed on $\tU$, we need the sample size to satisfy $(n_X\wedge n_Y)\gtrsim d\log d$. One possible approach to developing a nonconvex estimation procedure for high-dimensional differential network estimation could be based on a thresholding step for the low-rank component \citep{Yu2018Recovery}. }

Recent work on differential networks have focused on statistical inference, including developing statistical tests for the global null $H_0 : \tDelta = 0$ \citep{Xia2015Testing, Cai2019Differential} and development of confidence intervals for elements of the differential network \citep{Kim2019Two}. The regression approach of \cite{Ren2015Asymptotic} can be used to construct asymptotically normal estimators of the elements of the differential network in the presence of latent variables. Such an approach would require both the individual precision matrices to be sparse and the correlation between latent and observed variables to be weak. How to develop an inference procedure that requires only week conditions on the differential network remains an open problem.

In our simulation and real data application, we propose to choose the tuning parameters using cross-validation. \citet{Zhao2014Direct} proposed to tune the parameters by optimizing approximate Akaike information criterion in the context of sparse differential network estimation, however, there are no theoretical guarantees associated with the chosen parameters. Extending ideas of \citet{Foygel2010Extended} in the context of sparse plus low-rank estimation and showing that Akaike or Bayesian information criterion can be used for consistent recovery is of both practical and theoretical interest, as it would allow for faster parameter tuning compared to cross-validation.

\section*{Acknowledgment}

We are grateful to the editor, the associate editor and two referees for their insightful comments, which have led to significant improvement of our paper. We thank Huili Yuan and Ruibin Xi for sharing their R codes.  This work is partially supported by the William S. Fishman Faculty Research Fund at the University of Chicago Booth School of Business. This work was completed in part with resources supported by the University of Chicago Research Computing Center.

\putbib[paper]
\end{bibunit}


\appendix
\pagebreak
\begin{bibunit}[my-plainnat]

\numberwithin{equation}{section}
\numberwithin{lemma}{section}

\section*{Supplementary material: Estimating Differential Latent Variable Graphical Models with Applications to Brain Connectivity}

We proposed a directed procedure to estimate the differential networks under the existence of latent variables. In the present setup, the differential network can be decomposed into sparse component and low-rank component. Our method minimizes a nonconvex objective function via a two-stage procedure. In the first stage, we obtain a good initialization by truncating the sample precision matrices directly, while in the second stage, we conduct projected gradient descent to obtain a sequence of iterates, which can converge linearly to a neighborhood of the ground truth. The radius of the neighborhood is characterized by the optimal statistical error rate. In this Supplementary Material, we provide extended proofs, and then show additional simulation results including the comparison with potential convex approaches.

\setcounter{page}{1}

\section{Main Lemmas}\label{appen:main:lem}

In this section, we state lemmas needed to prove Theorems \ref{thm:1} and \ref{thm:2}. Their proofs are presented in Appendix \ref{appen:pf:main:lem}. We first introduce additional notations. Partial derivatives of the empirical loss functions in \eqref{loss:popu:SR} and \eqref{loss:emp:SULam} are given as
\begin{equation}\label{a:1}
\begin{aligned}
\nabla_S\mbL_n(S, U, \Lambda) & =  \frac{1}{2}\hSigmax(S + U\Lambda U^\T)\hSigmay + \frac{1}{2}\hSigmay(S + U\Lambda U^\T)\hSigmax - (\hSigmay - \hSigmax),\\
\nabla_U\mbL_n(S, U, \Lambda) & = \hSigmax(S + U\Lambda U^\T)\hSigmay U\Lambda + \hSigmay(S + U\Lambda U^\T)\hSigmax U\Lambda - 2(\hSigmay - \hSigmax)U\Lambda,\\
\nabla_S\mL_n(S, R) & = \nabla_R\mL_n(S, R) = \frac{1}{2}\hSigmax(S + R)\hSigmay + \frac{1}{2}\hSigmay(S + R)\hSigmax - (\hSigmay - \hSigmax).
\end{aligned}
\end{equation}
Partial derivatives of the population loss are similarly obtained by replacing $\hSigmax, \hSigmay$ with $\tSigmax, \tSigmay$ in \eqref{a:1}. We further define $\gamma_1 = \baralpha/\alpha$, $\gamma_2 = \bars/s$, and the following quantities
\begin{equation}\label{a:25}
\begin{aligned}
&\Upsilon_1= (\sigma_1^X\sigma_1^Y)^2\frac{d\log d}{n_X\wedge n_Y} + \cbr{(\sigma_1^Y\|\tSigmax\|_1)^2 + \big(\sigma_1^X\|\tSigmay\|_1\big)^2}\beta, \quad \Upsilon_2 = \sigma_d^X\sigma_d^Y + 9\sigma_1^X\sigma_1^Y,\\  &\Upsilon_3= \sigma_1^X\sigma_1^Y\sigma_d^X\sigma_d^Y, \quad\quad  \Upsilon_4 = (\sigma_1^Y\|\tSigmax\|_1)^2 + \big(\sigma_1^X\|\tSigmay\|_1\big)^2, \quad\quad \Psi^2 = \frac{\sigma_r^\tR}{\kappa_X^2\kappa_Y^2},\\
& C(\gamma_1, \gamma_2) =  \cbr{1 + \rbr{\frac{2}{\gamma_1 - 1}}^{1/2}}^2\cbr{1 + \frac{2}{\rbr{\gamma_2 - 1}^{1/2}}}.
\end{aligned}
\end{equation}
For ease of presentation, we generically use $C_i$ to denote constants and their values may vary for each appearance. The following two lemmas characterize the error based on one-step iteration of Algorithm \ref{alg:2}.

\begin{lemma}[One-step iteration for sparse component]\label{lem:2}

Suppose Assumptions \ref{ass:2} and \ref{ass:1}, $n_X\gtrsim \kappa_X^2d$, $n_Y\gtrsim \kappa_Y^2d$, and following conditions hold
\begin{align*}
\Lambda^0 = \tLambda, \quad S^k \in\S^{d\times d}, \quad U^k \in \mU(9\beta\sigma_1^\tR) \cap \{U\in\mR^{d\times r}: \Pi(U, \tU)\leq \surd{\sigma_1^\tR}/2\}.
\end{align*}
If $\eta_1\leq 8/(3\Upsilon_2)$, then $S^{k+1}\in\S^{d\times d}$ and
\begin{align*}
\frac{\|S^{k+1} - \tS\|_F^2}{C(\gamma_1, \gamma_2)}
&\leq \rbr{1  - \frac{9\Upsilon_3}{4\Upsilon_2}\eta_1} \|S^k - \tS\|_F^2 + 3(1 + \gamma_2)s\rbr{\frac{\Upsilon_2\eta_1}{\Upsilon_3} + \eta_1^2}\|\nabla_S\mL_n(\tS, \tR)\|_{\infty, \infty}^2\\
& \quad + \frac{C_1\Upsilon_2\Upsilon_1}{\Upsilon_3}(1 + \gamma_1)\alpha r\sigma_1^\tR \eta_1\Pi^2(U^k, \tU) + 7(\sigma_1^X\sigma_1^Y)^2\eta_1^2\|U^k\Lambda^0U^{k \T} - \tU\tLambda \tUT\|_F^2,
\end{align*}
with probability $1 - C_2/d^2$, where $\rbr{C_i}_{i=1}^2$ are fixed constants.

\end{lemma}

\begin{lemma}[One-step iteration for low-rank component]\label{lem:3}

Suppose conditions of Lemma \ref{lem:2} hold. If $\eta_2 \leq 1/(18\Upsilon_2\sigma_1^\tR)$, then
\begin{align*}
\Pi^2(U^{k + 1}, \tU) &\leq \cbr{1  - \frac{3\sigma_r^\tR \Upsilon_3\eta_2}{2\Upsilon_2} + \frac{C_1\Upsilon_2\Upsilon_1}{\Upsilon_3}(1 + \gamma_1)\alpha r\sigma_1^\tR\eta_2}\Pi^2(U^k, \tU) - \frac{27\Upsilon_3\eta_2}{16\Upsilon_2}\|R^k - \tR\|_F^2 \\
&\quad  + 5\kappa_X\kappa_Y\Upsilon_2\eta_2\Pi^4(U^k, \tU) + \eta_2\cbr{\frac{9\Upsilon_3}{4\Upsilon_2} + C_2\sigma_1^\tR(\sigma_1^X\sigma_1^Y)^2\eta_2}\|S^k - \tS\|_F^2\\
& \quad + \eta_2r\bigg(\frac{16\Upsilon_2}{9\Upsilon_3} + \frac{4}{\Upsilon_2} + 54\sigma_1^\tR\eta_2\bigg)\|\nabla_R\mL_n(\tS, \tR)\|_2^2,
\end{align*}
with probability $1 - C_3/d^2$, where $\rbr{C_i}_{i=1}^3$ are fixed constants.

\end{lemma}

Combining the above two lemmas, we obtain the decrease of the total error in one iteration.

\begin{lemma}\label{lem:4}

Suppose Assumptions \ref{ass:2} and \ref{ass:1} hold. Furthermore, suppose the following conditions hold: (a) sample sizes and sparsity proportion
\begin{align*}
\rbr{n_X\wedge n_Y} \geq  \sbr{\frac{\rbr{\sigma_1^X\sigma_1^Y}^2}{\Upsilon_4}\frac{d\log d}{\beta} \vee\cbr{ \rbr{\kappa_X\vee \kappa_Y}^2d}}, \quad \alpha \leq \frac{\Upsilon_3^2}{C_1\Upsilon_2^2\Upsilon_4\gamma_2}\cdot \frac{1}{\beta r \kappa_\tR},
\end{align*}
(b) step sizes $\eta_1 \leq 1/(C_2\kappa_X\kappa_Y\Upsilon_2)$, $\eta_2 = {\eta_1}/\rbr{36\sigma_1^\tR}$, tuning parameters $\gamma_1 \geq 1 + {8\Upsilon_2^2}/\rbr{\Upsilon_3^2\eta_1^2}$, $\gamma_2 \geq \rbr{1+\gamma_1}/{2}$; (c) the $k$-th iterate satisfies
\begin{align*}
\Lambda^0 = \tLambda, \quad S^k \in \S^{d\times d}, \quad U^k\in \mU(9\beta\sigma_1^\tR) \cap \{U\in\mR^{d\times r}: \Pi^2(U, \tU)\leq \Psi^2/C_3\}.
\end{align*}
Then, with probability at least $1 - C_4/d^2$,
\begin{multline*}
TD(S^{k+1}, U^{k+1})\leq  \rbr{1  - \frac{\sigma_r^\tR\Upsilon_3\eta_2}{2\Upsilon_2}}TD(S^k, U^k) + \frac{1}{\kappa_X\kappa_Y\Upsilon_3\sigma_1^\tR}\cdot\bigr\{\gamma_2s\|\nabla_S\mL_n(\tS, \tR)\|_{\infty, \infty}^2\\
+ r\|\nabla_R\mL_n(\tS, \tR)\|_2^2\bigl\},
\end{multline*}
where $\rbr{C_i}_{i=1}^4$ are fixed constants.

\end{lemma}

From Lemma~\ref{lem:4}, we observe that the successive total error distance decreases with linear contraction rate $\rho= 1 - \rbr{\sigma_r^\tR\Upsilon_3\eta_2}/\rbr{2\Upsilon_2}<1$ up to a statistical error, which comes from the approximation of population loss $\mL(S, R)$. The statistical error bound is given in the next lemma.

\begin{lemma}[Statistical error bound]\label{lem:5}

The gradients of $\mL_n(S, R)$, defined in \eqref{a:1}, satisfy
\begin{align*}
&\Pr\cbr{\|\nabla_R\mL_n(\tS, \tR)\|_2\gtrsim (\kappa_X\sigma_1^Y + \kappa_Y\sigma_1^X)\rbr{\frac{d}{n_X\wedge n_Y}}^{1/2}} \lesssim \frac{1}{d^2},\\
&\Pr\cbr{\|\nabla_S\mL_n(\tS, \tR)\|_{\infty, \infty}\gtrsim \big(\|\tOmegay\|_1\|\tSigmax\|_1 + \|\tOmegax\|_1\|\tSigmay\|_1\big)\rbr{\frac{\log d}{n_X\wedge n_Y}}^{1/2}} \lesssim \frac{1}{d^2}.
\end{align*}

\end{lemma}

The proof of Theorem \ref{thm:1} combines Lemma \ref{lem:2}, \ref{lem:3}, \ref{lem:4}, \ref{lem:5} and is given in Appendix~\ref{appen:pf:main:thm}. The next lemma establishes the error bound for $S^0$ in the initialization step.

\begin{lemma}[Error bound for $S^0$]\label{lem:6}

Suppose Assumptions \ref{ass:2} and \ref{ass:1} hold. If $\halpha\geq \alpha$, $\hs\geq s$, $d\leq c(n_X\wedge n_Y)$ for $c\in(0, 1/2)$, then
\begin{align*}
\|S^0 - \tS\|_F \leq 17 \hs^{1/2}\sbr{\cbr{\|(\tOmegax)^{1/2}\|_1^2 + \|(\tOmegay)^{1/2}\|_1^2}\rbr{\frac{\log d}{n_X\wedge n_Y}}^{1/2} + \frac{\beta r\sigma_1^\tR}{d}},
\end{align*}
with probability at least $1 - 8/d^2$.

\end{lemma}

\section{Proofs of Main Lemmas}\label{appen:pf:main:lem}

\subsection{Proof of Lemma \ref{lem:2}}

We study the $k$-th iteration for updating sparse component in Algorithm \ref{alg:2}. Recall that $\baralpha = \gamma_1\alpha$ and $\bars = \gamma_2 s$. Define
\begin{align}\label{a:2}
\barS^{k + 1/2} = \mJ_{\gamma_2s}(S^{k+1/2}), \quad
\barOmega^k =  \SUPP(\tS) \cup \SUPP(S^k), \quad
\Omega^k =  \barOmega^k \cup \SUPP(\barS^{k+1/2}).
\end{align}
Since $\nabla_S\mbL_n(S^k, U^k, \Lambda^0)$ and $S^k$ are symmetric, so is $S^{k+1/2}$. From Lemma \ref{aux:lem:1} and \ref{aux:lem:2}, we have $S^{k+1}\in\S^{d\times d}$ and,
therefore, $\Omega^k, \barOmega^k\subseteq V \times V$ are two symmetric index sets. With some abuse of the notations, we use $\P_{\Omega}(\cdot)$ to denote the projection onto the $\Omega$. For a matrix $A\in\mR^{d\times d}$, $\P_{\Omega}(A)\in\mR^{d\times d}$ with elements $[\P_{\Omega}(A)]_{i, j} = A_{i, j}\cdot\pmb{1}_{\{(i, j)\in\Omega\}}$, where $\pmb{1}_{\{\cdot\}}$ is an indicator
function. From the updating rule in Algorithm~\ref{alg:2},
\begin{equation}\label{a:3}
\P_{\Omega^k}(S^{k + 1/2})
= \P_{\Omega^k}(S^k) - \eta_1\P_{\Omega^k}\cbr{\nabla_S\mbL_n(S^k, U^k, \Lambda^0)}
= S^k - \eta_1\P_{\Omega^k}\cbr{\nabla_S\mbL_n(S^k, U^k, \Lambda^0)}.
\end{equation}
Combining \eqref{a:2} and \eqref{a:3}, and noting that $\SUPP(\barS^{k+1/2})\subseteq \Omega^k$,
\begin{align*}
\barS^{k + 1/2}
= \mJ_{\gamma_2s}(S^{k+1/2})
= \mJ_{\gamma_2s}\cbr{\P_{\Omega^k}(S^{k+1/2})}
= \mJ_{\gamma_2s}\sbr{S^k - \eta_1\P_{\Omega^k}\cbr{\nabla_S\mbL_n(S^k, U^k, \Lambda^0)}}.
\end{align*}
Further, from Lemma \ref{aux:lem:1} and \ref{aux:lem:2},
\begin{align}\label{a:4}
\|S^{k+1} &- \tS\|_F^2
=  \|\mT_{\gamma_1\alpha}(\barS^{k+1/2}) - \tS\|_F^2 \leq  \cbr{1 + \rbr{\frac{2}{\gamma_1 - 1}}^{1/2}}^2\|\barS^{k + 1/2} - \tS\|_F^2 \nonumber\\
&= \cbr{1 + \rbr{\frac{2}{\gamma_1 - 1}}^{1/2}}^2\bigg\| \mJ_{\gamma_2s}\sbr{S^k - \eta_1\P_{\Omega^k}\cbr{\nabla_S\mbL_n(S^k, U^k, \Lambda^0)}}- \tS\bigg\|_F^2 \nonumber\\
& \leq \cbr{1 + \rbr{\frac{2}{\gamma_1 - 1}}^{1/2}}^2 \cbr{1 + \frac{2}{\rbr{\gamma_2 - 1}^{1/2}}} \bigg\| S^k - \eta_1\P_{\Omega^k}\cbr{\nabla_S\mbL_n(S^k, U^k, \Lambda^0)}- \tS\bigg\|_F^2 \nonumber\\
&= C(\gamma_1, \gamma_2) \rbr{\|S^k - \tS\|_F^2 - 2\eta_1\I_1 + \eta_1^2\I_2},
\end{align}
where $C(\gamma_1, \gamma_2)$ is defined in \eqref{a:25} and
\begin{align*}
\I_1 =\bLD S^k - \tS, \P_{\Omega^k}\cbr{\nabla_S\mbL_n(S^k, U^k, \Lambda^0)} \bRD, \quad
\I_2 = \big\|\P_{\Omega^k}\cbr{\nabla_S\mbL_n(S^k, U^k, \Lambda^0)} \big\|_F^2.
\end{align*}
Using Lemma \ref{lem:lem_2_I1} to lower bound $\I_1$ and Lemma \ref{lem:lem_2_I2} to upper bound $\I_2$,
\begin{align*}
\frac{\|S^{k+1} - \tS\|_F^2}{C(\gamma_1, \gamma_2)}
&\leq \rbr{1  - \frac{9\Upsilon_3}{4\Upsilon_2}\eta_1}\|S^k - \tS\|_F^2+ 3(1 + 2\gamma_2)s\rbr{\frac{4\Upsilon_2}{9\Upsilon_3}\eta_1 + \eta_1^2}\|\nabla_S\mL_n(\tS, \tR)\|_{\infty, \infty}^2\\
+ &\frac{C_1\Upsilon_2\Upsilon_1}{\Upsilon_3}(1 + \gamma_1)\alpha r\sigma_1^\tR \eta_1\Pi^2(U^k, \tU) + \frac{27(\sigma_1^X\sigma_1^Y)^2\eta_1^2}{4}\|U^k\Lambda^0U^{k \T} - \tU\tLambda \tUT\|_F^2\\
- & \eta_1\rbr{\frac{8}{\Upsilon_2} - 3\eta_1}\|\nabla_S\mbL_n(S^k, U^k, \Lambda^0) - \nabla_S\mbL_n(\tS, U^k, \Lambda^0)\|_F^2,
\end{align*}
with probability $1-C_2/d^2$ for some large enough constants $C_1, C_2>0$. With $\eta_1 \leq 8/(3\Upsilon_2)$,
\begin{align}\label{a:27}
\frac{\|S^{k+1} - \tS\|_F^2}{C(\gamma_1, \gamma_2)}
& \leq \rbr{1  - \frac{9\Upsilon_3}{4\Upsilon_2}\eta_1}\|S^k - \tS\|_F^2 +
3\rbr{1 + 2\gamma_2}s\rbr{\frac{4\Upsilon_2}{9\Upsilon_3}\eta_1 + \eta_1^2}\|\nabla_S\mL_n(\tS, \tR)\|_{\infty, \infty}^2 \nonumber\\
+ \frac{C_1\Upsilon_2\Upsilon_1}{\Upsilon_3}&(1 + \gamma_1)\alpha r\sigma_1^\tR \eta_1\Pi^2(U^k, \tU) + \frac{27(\sigma_1^X\sigma_1^Y)^2\eta_1^2}{4}\|U^k\Lambda^0U^{k \T} - \tU\tLambda \tUT\|_F^2,
\end{align}
which completes the proof.

\subsection{Proof of Lemma \ref{lem:3}}

Suppose $Q^k\in\mQ^{r\times r}_{r_1}$ satisfies $\Pi(U^k, \tU) = \|U^k - \tU Q^k\|_F$. Using the bound on $\|U^k\|_2$ in \eqref{a:17}, we know $\tU\in \mU(4\beta\|U^k\|_2^2)=\mC^k$, so does $\tU Q^k$. Let $\P_{\mC^k}(\cdot)$ be the projection operator onto $\mC^k$. Due to the non-expansion property of $\P_{\mC^k}(\cdot)$,
\begin{align}\label{a:28}
\Pi^2(U^{k+1}, \tU)
& \leq \|U^{k + 1} - \tU Q^k\|_F^2 = \|\P_{\mC^k}(U^{k+1/2}) - \P_{\mC^k}(\tU Q^k)\|_F^2 \leq  \|U^{k + 1/2} - \tU Q^k\|_F^2 \nonumber\\
&= \|U^k - \eta_2\nabla_U\mbL_n(S^k, U^k, \Lambda^0) - \frac{\eta_2}{2}U^k(U^{k \T}U^{k \T} - \Lambda^0U^{k \T}U^k\Lambda^0) - \tU Q^k\|_F^2 \nonumber\\
& \leq  \Pi^2(U^k, \tU) - 2\eta_2\I_3 + 2\eta_2^2\I_4 - \eta_2\I_5 + \frac{\eta_2^2}{2}\I_6,
\end{align}
where
\begin{align*}
\I_3 =& \LD U^k - \tU Q^k, \nabla_U\mbL_n(S^k, U^k, \Lambda^0)\RD, \quad\quad\quad\quad\text{\ \ \ }
\I_4 = \|\nabla_U\mbL_n(S^k, U^k, \Lambda^0)\|_F^2, \\
\I_5 =& \LD U^k - \tU Q^k, U^k(U^{k \T}U^k - \Lambda^0U^{k \T}U^k\Lambda^0)\RD,\quad
\I_6 = \|U^k(U^{k \T}U^k - \Lambda^0U^{k \T}U^k\Lambda^0)\|_F^2.
\end{align*}
Similar to \eqref{a:4}, we will lower bound $\I_3, \I_5$ and upper bound $\I_4, \I_6$. Using Lemma \ref{lem:lem_3_I3}, we bound the term $\I_3$. Using \eqref{a:1} and the triangle inequality,
\begin{align*}
\I_4 & = 4\|\nabla_R\mL_n(S^k, R^k)U^k \Lambda^0\|_F^2\\
& \leq 12\big\|\cbr{\nabla_R\mL_n(S^k, R^k) - \nabla_R\mL_n(S^k, \tR)}U^k\big\|_F^2 + 12\big\|\cbr{\nabla_R\mL_n(S^k, \tR) - \nabla_R\mL_n(\tS, \tR)}U^k\big\|_F^2\\
& \quad + 12\|\nabla_R\mL_n(\tS, \tR)U^k\|_F^2.
\end{align*}
Using H\"older's inequality and the bound in \eqref{a:17}, we further have
\begin{align}\label{a:34}
\I_4 & \leq  27\sigma_1^\tR\|\nabla_R\mL_n(S^k, R^k) - \nabla_R\mL_n(S^k, \tR)\|_F^2 + 27\sigma_1^\tR\|\hSigmax(S^k - \tS)\hSigmay\|_F^2
\nonumber\\
&\quad + 12r\|\nabla_R\mL_n(\tS, \tR)\|_2^2 \|U^k\|_2^2 \nonumber\\
& \leq 27\sigma_1^\tR\|\nabla_R\mL_n(S^k, R^k) - \nabla_R\mL_n(S^k, \tR)\|_F^2 + \frac{243}{4}\sigma_1^\tR(\sigma_1^X\sigma_1^Y)^2\|S^k - \tS\|_F^2 \nonumber\\
&\quad + 27r\sigma_1^\tR\|\nabla_R\mL_n(\tS, \tR)\|_2^2.
\end{align}
By Lemma \ref{aux:lem:14},
\begin{align}\label{a:35}
\I_5 \geq \frac{1}{8}\|U^{k \T}U^k - \Lambda^0U^{k \T}U^k\Lambda^0\|_F^2 - \frac{1}{2}\Pi^4(U^k, \tU),
\end{align}
and by \eqref{a:17},
\begin{align}\label{a:36}
\I_6 \leq \|U_k\|_2^2\cdot\|U^{k \T}U^k - \Lambda^0U^{k \T}U^k\Lambda^0\|_F^2
\leq \frac{9\sigma_1^\tR}{4}\|U^{k \T}U^k - \Lambda^0U^{k \T}U^k\Lambda^0\|_F^2.
\end{align}
Combining pieces in \eqref{a:28}, \eqref{a:34}, \eqref{a:35}, \eqref{a:36} and Lemma~\ref{lem:lem_3_I3}, there exist constants $C_1, C_2, C_3>0$, such that
\begin{equation}
\begin{aligned}\label{a:37}
\Pi^2&(U^{k+1}, \tU) \leq \rbr{1 + \frac{C_1(1 + \gamma_1)\alpha r\sigma_1^\tR\Upsilon_1\eta_2}{C_{32,1}}}\Pi^2(U^k, \tU) - \eta_2\cbr{\frac{9\Upsilon_3}{2\Upsilon_2} - \rbr{2r}^{\frac{1}{2}}C_{33,1}}\|R^k - \tR\|_F^2 \\
&- \frac{\eta_2 - 9\sigma_1^\tR\eta_2^2}{8}\|U^{k \T}U^k - \Lambda^0U^{k \T}U^k \Lambda^0\|_F^2 + \eta_2\cbr{\frac{1}{2}+ C_{31} + \rbr{2r}^{\frac{1}{2}}C_{33,2} + \frac{9\sigma_1^X\sigma_1^Y}{4C_{32,2}}}\Pi^4(U^k, \tU) \\
& + \eta_2\cbr{C_{32,1} + \frac{9\sigma_1^X\sigma_1^YC_{32,2}}{4} + C_2\sigma_1^\tR(\sigma_1^X\sigma_1^Y)^2\eta_2}\|S^k - \tS\|_F^2 \\
& -\eta_2\bigg(\frac{8}{\Upsilon_2} - \frac{1}{C_{31}} - 54\sigma_1^\tR\eta_2\bigg)\big\|\nabla_R\mL_n(S^k, R^k)  - \nabla_R\mL_n(S^k, \tR)\big\|_F^2 \\
&+ \cbr{\frac{\rbr{2r}^{\frac{1}{2}}(C_{33,1} + C_{33,2})}{C_{33,1}C_{33,2}}\eta_2 + 54r\sigma_1^\tR\eta_2^2}\|\nabla_R\mL_n(\tS, \tR)\|_2^2,
\end{aligned}
\end{equation}
with probability at least $1-C_3/d^2$ for any $C_{31}, C_{32,1}, C_{32,2}, C_{33,1}, C_{33,2}>0$. We let
\begin{align}\label{a:38}
C_{31} = & \frac{\Upsilon_2}{2}, \text{\ \ \ } C_{32,1} = \frac{9\Upsilon_3}{8\Upsilon_2}, \text{\ \ \ } C_{32,2} = \frac{\sigma_d^X\sigma_d^Y}{2\Upsilon_2}, \text{\ \ \ } C_{33,1} = \frac{9\Upsilon_3}{8\rbr{2r}^{1/2}\Upsilon_2}, \text{\ \ \ } C_{33,2} = \frac{\Upsilon_2}{2\rbr{2r}^{1/2}}.
\end{align}
With $\eta_2 \leq 1/(18\Upsilon_2\sigma_1^\tR)$, there exists a constant $C_4>0$ such that
\begin{align}\label{a:39}
\Pi^2(U^{k + 1}, \tU)\leq& \cbr{1 + \frac{C_4\Upsilon_2\Upsilon_1}{\Upsilon_3}(1 + \gamma_1)\alpha r\sigma_1^\tR\eta_2}\Pi^2(U^k, \tU) - \frac{27\Upsilon_3\eta_2}{8\Upsilon_2}\|R^k - \tR\|_F^2 \nonumber\\
& - \frac{\eta_2 - 9\sigma_1^\tR\eta_2^2}{8}\|U^{k \T}U^k - \Lambda^0U^{k \T}U^k \Lambda^0\|_F^2 + 5\eta_2\kappa_X\kappa_Y\Upsilon_2\Pi^4(U^k, \tU) \nonumber\\
& + \eta_2\bigg(\frac{16r\Upsilon_2}{9\Upsilon_3} + \frac{4r}{\Upsilon_2} + 54r\sigma_1^\tR\eta_2\bigg)\|\nabla_R\mL_n(\tS, \tR)\|_2^2 \nonumber\\
& + \eta_2\cbr{\frac{9\Upsilon_3}{4\Upsilon_2} + C_2\sigma_1^\tR(\sigma_1^X\sigma_1^Y)^2\eta_2}\|S^k - \tS\|_F^2.
\end{align}
Focusing on the second and the third term in above inequality, we write
\begin{align*}
\frac{27\Upsilon_3\eta_2}{8\Upsilon_2}&\|R^k - \tR\|_F^2
+ \rbr{\frac{\eta_2}{8}- \frac{9\sigma_1^\tR\eta_2^2}{8}}\|U^{k \T}U^k - \Lambda^0U^{k \T}U^k \Lambda^0\|_F^2 \\
=&\frac{27\Upsilon_3\eta_2}{16\Upsilon_2}\|R^k - \tR\|_F^2
+  \eta_2\rbr{\frac{27\Upsilon_3}{4\Upsilon_2}\frac{1}{4}\|R^k - \tR\|_F^2 + \frac{1}{16}\|U^{k \T}U^k - \Lambda^0U^{k \T}U^k \Lambda^0\|_F^2}\\
& + \frac{\eta_2}{8}\rbr{\frac{1}{2} - 9\sigma_1^\tR\eta_2 }\|U^{k \T}U^k - \Lambda^0U^{k \T}U^k \Lambda^0\|_F^2.
\end{align*}
Without loss of generality, $\Upsilon_2>1$ and ${27\Upsilon_3}/\rbr{4\Upsilon_2}<1$. Then, by Lemma \ref{aux:lem:14},
\begin{multline*}
\frac{27\Upsilon_3\eta_2}{8\Upsilon_2}\|R^k - \tR\|_F^2
+ \rbr{\frac{\eta_2}{8}- \frac{9\sigma_1^\tR\eta_2^2}{8}}\|U^{k \T}U^k - \Lambda^0U^{k \T}U^k \Lambda^0\|_F^2 \\
\geq  \frac{27\Upsilon_3\eta_2}{16\Upsilon_2}\|R^k - \tR\|_F^2 + \frac{3\eta_2\sigma_r^\tR \Upsilon_3}{2\Upsilon_2}\Pi^2(U^k, \tU).
\end{multline*}
Plugging into \eqref{a:39}, we obtain the error recursion for one-step iteration for the low-rank component
\begin{align}\label{a:40}
\Pi^2(&U^{k + 1}, \tU)\leq \cbr{1  - \frac{3\eta_2\sigma_r^\tR \Upsilon_3}{2\Upsilon_2} + \frac{C_4\Upsilon_2\Upsilon_1}{\Upsilon_3}(1 + \gamma_1)\alpha r\sigma_1^\tR\eta_2}\Pi^2(U^k, \tU)  \nonumber\\
&  + 5\eta_2\kappa_X\kappa_Y\Upsilon_2\Pi^4(U^k, \tU) + \eta_2\cbr{\frac{9\Upsilon_3}{4\Upsilon_2} + C_2\sigma_1^\tR(\sigma_1^X\sigma_1^Y)^2\eta_2}\|S^k - \tS\|_F^2\nonumber\\
& + \eta_2\bigg(\frac{16r\Upsilon_2}{9\Upsilon_3} + \frac{4r}{\Upsilon_2} + 54r\sigma_1^\tR\eta_2\bigg)\|\nabla_R\mL_n(\tS, \tR)\|_2^2- \frac{27\Upsilon_3\eta_2}{16\Upsilon_2}\|R^k - \tR\|_F^2,
\end{align}
which completes the proof.

\subsection{Proof of Lemma \ref{lem:4}}

Under the assumptions of the lemma, the conditions of Lemma \ref{lem:2} are satisfied.  By the definition of the total error distance, we combine \eqref{a:27} and \eqref{a:40} to get
\begin{multline}\label{a:41}
TD(S^{k+1}, U^{k+1})\leq M_1\frac{\|S^k - \tS\|_F^2}{\sigma_1^\tR}  + M_2\Pi^2(U^k, \tU) + M_3\|R^k - \tR\|_F^2 \\
+ M_4\|\nabla_S\mL_n(\tS, \tR)\|_{\infty, \infty}^2 + M_5\|\nabla_R\mL_n(\tS, \tR)\|_2^2,
\end{multline}
with probability at least $1 - C_4/d^2$, where
\begin{align}\label{a:42}
M_1 &= \rbr{1 - \frac{9\Upsilon_3}{4\Upsilon_2}\eta_1}C(\gamma_1, \gamma_2) + \sigma_1^\tR\eta_2\cbr{\frac{9\Upsilon_3}{4\Upsilon_2} + C_1\sigma_1^\tR(\sigma_1^X\sigma_1^Y)^2\eta_2}, \nonumber\\
M_2 &= 1 - \frac{3\sigma_r^\tR\Upsilon_3}{2\Upsilon_2}\eta_2 + \frac{C_2\Upsilon_2\Upsilon_1}{\Upsilon_3}(1 + \gamma_1)\alpha r \sigma_1^\tR\eta_2 + \frac{C_2\Upsilon_2\Upsilon_1}{\Upsilon_3}C(\gamma_1, \gamma_2)(1 + \gamma_1)\alpha r\eta_1 \nonumber\\
& \quad + 5\eta_2\kappa_X\kappa_Y\Upsilon_2\Pi^2(U^k, \tU), \nonumber\\
M_3&= \frac{27(\sigma_1^X\sigma_1^Y)^2\eta_1^2}{4\sigma_1^\tR}C(\gamma_1, \gamma_2) - \frac{27\Upsilon_3\eta_2}{16\Upsilon_2},\\
M_4 &= \frac{3(1 + 2\gamma_2)s}{\sigma_1^\tR}\big(\frac{4\Upsilon_2}{9\Upsilon_3}\eta_1 + \eta_1^2\big)C(\gamma_1, \gamma_2),\nonumber\\
M_5 &= \eta_2r\big( \frac{16\Upsilon_2}{9\Upsilon_3} + \frac{4}{\Upsilon_2} + 54\sigma_1^\tR\eta_2\big), \nonumber
\end{align}
for some constants $\rbr{C_i}_{i=1}^4$. We proceed to simplify $\rbr{M_i}_{i=1}^5$ under the assumptions. Under the conditions on $\gamma_1$ and $\gamma_2$,
\begin{align}\label{a:43}
\rbr{1 - \frac{9\Upsilon_3\eta_1}{4\Upsilon_2}}C(\gamma_1, \gamma_2)\leq 1 - \frac{\Upsilon_3\eta_1}{4\Upsilon_2}.
\end{align}
Furthermore, using the bounds on $\eta_1$ and $\eta_2$, we obtain
\begin{align}\label{a:44}
\sigma_1^\tR\eta_2\cbr{\frac{9\Upsilon_3}{4\Upsilon_2} + C_1\eta_2\sigma_1^\tR(\sigma_1^X\sigma_1^Y)^2} =  \frac{\Upsilon_3\eta_1}{16\Upsilon_2} + \frac{C_1\eta_1^2}{36^2}(\sigma_1^X\sigma_1^Y)^2 \leq \frac{\Upsilon_3\eta_1}{8\Upsilon_2}.
\end{align}
Combining the last two inequalities,
$M_1 \leq 1 - {\Upsilon_3\eta_1}/\rbr{8\Upsilon_2}$. Similarly,
\begin{align}\label{a:46}
M_2 \leq 1 - \frac{3\sigma_r^\tR\Upsilon_3\eta_2}{2\Upsilon_2} + \frac{C_2\Upsilon_2\Upsilon_1}{\Upsilon_3}(1 + \gamma_1)\alpha r\sigma_1^\tR\cbr{1 + 36C(\gamma_1, \gamma_2)}\eta_2 + 5\kappa_X\kappa_Y\Upsilon_2\eta_2\Pi^2(U^k, \tU).
\end{align}
Since $n_X\wedge n_Y \geq \beta^{-1}\{({\|\tSigmax\|_1}/{\sigma_1^X})^2 + ({\|\tSigmay\|_1}/{\sigma_1^Y})^2 \}^{-1} {d\log d},$
we have $\Upsilon_1 \leq 2\Upsilon_4 \beta$. Furthermore, if following conditions hold
\begin{align*}
\alpha \leq \frac{\Upsilon_3^2}{4C_2\Upsilon_2^2\Upsilon_4(1+\gamma_1)\cbr{1 + 36C(\gamma_1, \gamma_2)}}\cdot\frac{1}{\beta r \kappa_\tR},\quad\quad \Pi^2(U^k, \tU)\leq \frac{\Upsilon_3}{10\kappa_X\kappa_Y\Upsilon_2^2}\cdot \sigma_r^\tR,
\end{align*}
we can get
\begin{align}\label{a:50}
M_2 \leq 1 - \frac{3\sigma_r^\tR\Upsilon_3\eta_2}{2\Upsilon_2} + \frac{\sigma_r^\tR\Upsilon_3\eta_2}{2\Upsilon_2} +  \frac{\sigma_r^\tR\Upsilon_3\eta_2}{2\Upsilon_2} = 1 - \frac{\sigma_r^\tR\Upsilon_3\eta_2}{2\Upsilon_2}.
\end{align}
The above conditions on $\alpha$ and $\Pi^2(U^k, \tU)$ are implied by the ones in lemma, noting that $1+\gamma_1\leq 2\gamma_2$, $C(\gamma_1, \gamma_2)\leq 2$ due to the setup of $\eta_1$ and \eqref{a:43}, and $\Upsilon_3\sigma_r^\tR/(10\kappa_X\kappa_Y\Upsilon_2^2)\asymp\sigma_r^\tR/(\kappa_X^2\kappa_Y^2)\asymp\Psi^2$. For the term $M_3$, we note that $M_3\leq  0\Longleftrightarrow \eta_1\leq \rbr{144C(\gamma_1, \gamma_2)\kappa_X\kappa_Y\Upsilon_2}^{-1}$. Since $C(\gamma_1, \gamma_2)\leq 2$, by choosing the constant in the learning rate $\eta_1$ big enough, the right hand side condition holds so that $M_3\leq 0$. For the term $M_4$, we have
\begin{align}\label{a:52}
M_4 \leq 6(1 + 2\gamma_2)\cbr{\frac{4\Upsilon_2}{9\Upsilon_3}\eta_1 + \eta_1^2}\frac{s}{\sigma_1^\tR}
\leq \frac{\gamma_2 s}{\kappa_X\kappa_Y\Upsilon_3\sigma_1^\tR},
\end{align}
and for $M_5$,
\begin{align}\label{a:53}
M_5 = \frac{\eta_1}{36}\rbr{\frac{16\Upsilon_2}{9\Upsilon_3} + \frac{4}{\Upsilon_2} + \frac{3\eta_1}{2}}\frac{r}{\sigma_1^\tR}
\leq \frac{\Upsilon_2\eta_1r}{\Upsilon_3\sigma_1^\tR}\leq \frac{r}{\kappa_X\kappa_Y\Upsilon_3\sigma_1^\tR}.
\end{align}
Plugging all the bounds back into \eqref{a:41},
\begin{multline}
TD(S^{k+1}, U^{k+1})\leq \rbr{1 - \frac{\Upsilon_3\eta_1}{8\Upsilon_2}}\frac{\|S^k - \tS\|_F^2}{\sigma_1^\tR} + \rbr{1  - \frac{\sigma_r^\tR\Upsilon_3\eta_2}{2\Upsilon_2}}\Pi^2(U^k, \tU)\\
+\frac{1}{\kappa_X\kappa_Y\Upsilon_3}\cbr{\frac{\gamma_2s}{\sigma_1^\tR}\|\nabla_S\mL_n(\tS, \tR)\|_{\infty, \infty}^2 + \frac{r}{\sigma_1^\tR}\|\nabla_R\mL_n(\tS, \tR)\|_2^2}.
\end{multline}
The proof is completed by noting
$\rbr{\Upsilon_3\eta_1}/\rbr{8\Upsilon_2}\geq\rbr{\sigma_r^\tR\Upsilon_3\eta_2}/\rbr{2\Upsilon_2}$.

\subsection{Proof of Lemma \ref{lem:5}}

From \eqref{a:1}, we have
\begin{equation} \label{c:19}
\begin{aligned}
\nabla_R &\mL_n(\tS, \tR)\\
&= \frac{1}{2}\hSigmax\big(\tS + \tR\big)\hSigmay + \frac{1}{2}\hSigmay\big(\tS + \tR\big)\hSigmax - (\hSigmay - \hSigmax) \\
& = \frac{1}{2}\hSigmax\big(\tOmegax - \tOmegay\big)\hSigmay + \frac{1}{2}\hSigmay\big(\tOmegax - \tOmegay\big)\hSigmax - (\hSigmay - \hSigmax)\\
& = \frac{1}{2}\big(\hSigmax\tOmegax - I_d\big)\hSigmay + \frac{1}{2}\hSigmax\big(I_d - \tOmegay\hSigmay\big) + \frac{1}{2}\hSigmay\big(\tOmegax\hSigmax - I_d\big) + \frac{1}{2}\big(I_d - \hSigmay\tOmegay\big)\hSigmax \\
& = \frac{1}{2}\big(\hSigmax - \tSigmax\big)\tOmegax\hSigmay + \frac{1}{2}\hSigmax\tOmegay\big(\tSigmay - \hSigmay\big) + \frac{1}{2}\hSigmay\tOmegax\big(\hSigmax - \tSigmax\big) + \frac{1}{2}\big(\tSigmay - \hSigmay\big)\tOmegay\hSigmax.
\end{aligned}
\end{equation}
Using Lemma \ref{aux:lem:7} and \eqref{a:5}, there exists a constant $C_1 > 0$ such that
\begin{align*}
\|\nabla_R\mL_n(\tS, \tR)\|_2\leq & \|\tOmegax\|_2\|\hSigmay\|_2\|\hSigmax - \tSigmax\|_2 + \|\tOmegay\|_2\|\hSigmax\|_2\|\hSigmay - \tSigmay\|_2\\
\lesssim & (\kappa_X\sigma_1^Y + \kappa_Y\sigma_1^X)\rbr{\frac{d}{n_X\wedge n_Y}}^{1/2}
\end{align*}
with probability at least $1 - C_1/d^2$. Furthermore, since $\nabla_S\mL_n(\tS, \tR) = \nabla_R\mL_n(\tS, \tR)$, it follows from \eqref{c:19} that
\begin{align*}
\nabla_S&\mL_n(\tS, \tR)
=  \frac{1}{2}\big(\hSigmax - \tSigmax\big)\big(\tOmegax - \tOmegay\big)\big(\hSigmay - \tSigmay\big)
+ \frac{1}{2}\big(\hSigmay - \tSigmay\big)\big(\tOmegax - \tOmegay\big)\big(\hSigmax - \tSigmax\big) \\
&+ \frac{1}{2}\big(\hSigmax - \tSigmax\big)\tOmegax\tSigmay
+ \frac{1}{2}\tSigmax\tOmegay\big(\tSigmay - \hSigmay\big)
+ \frac{1}{2}\tSigmay\tOmegax\big(\hSigmax - \tSigmax\big)
+ \frac{1}{2}\big(\tSigmay - \hSigmay\big)\tOmegay\tSigmax.
\end{align*}
Therefore,
\begin{multline}\label{c:18}
\|\nabla_S\mL_n(\tS, \tR)\|_{\infty, \infty}\leq \|\hSigmay - \tSigmay\|_{\infty, \infty}\|\tOmegay\|_1\|\tSigmax\|_1 + \|\hSigmax - \tSigmax\|_{\infty, \infty}\|\tOmegax\|_1\|\tSigmay\|_1 \\
+ \|\hSigmax - \tSigmax\|_{\infty, \infty}\|\hSigmay - \tSigmay\|_{\infty, \infty}\|\tOmegax - \tOmegay\|_{1,1},
\end{multline}
where we use three inequalities:
$\|ABC\|_{\infty, \infty}\leq \|A\|_{\infty, \infty}\|C\|_{\infty, \infty}\|B\|_{1,1}$, $\|AB\|_{\infty, \infty}\leq \|A\|_{\infty, \infty}\|B\|_1$, and $\|AB\|_{\infty, \infty}\leq \|A\|_{\infty}\|B\|_{\infty, \infty}$. From Lemma 1 in \cite{Rothman2008Sparse}, there exists a constant $C_2>0$ such that (similar for $\hSigmay$)
\begin{align*}
\Pr\bigg(\|\hSigmax - \tSigmax\|_{\infty, \infty}\lesssim \rbr{\frac{\log d}{n_X}}^{1/2}\bigg)\geq 1 - \frac{C_2}{d^2}.
\end{align*}
Therefore, the last term in \eqref{c:18} only contributes a high-order term and
\begin{align*}
\Pr\bigg(\|\nabla_S\mL_n(\tS, \tR)\|_{\infty, \infty}\lesssim \big(\|\tOmegay\|_1\|\tSigmax\|_1 + \|\tOmegax\|_1\|\tSigmay\|_1\big)\rbr{\frac{\log d}{n_X\wedge n_Y}}^{1/2}\bigg)\geq 1 - \frac{C_3}{d^2},
\end{align*}
for a constant $C_3 > 0$.

\subsection{Proof of Lemma \ref{lem:6}}

Let $\barS^0 = \mJ_\hs(\hDelta^0)$. Then $S^0 = \mT_{\halpha}(\barS^0)$. Since $\SUPP(\tS - S^0)\subseteq\SUPP(\tS)\cup\SUPP(S^0)$, we consider
the following cases that depend on $(i, j)$ location.

\noindent{\bf Case 1.} If $(i, j)\in\SUPP(S^0)$, then
\begin{align}\label{d:1}
|\tS_{i, j} - S^0_{i, j}|
= |\tS_{i, j} - \hDelta_{i, j}^0|
= |\tDelta_{i, j} - \hDelta_{i, j}^0 - \ttR_{i, j}|
\leq \|\tDelta - \hDelta^0\|_{\infty, \infty} + \|\tR\|_{\infty, \infty}.
\end{align}
\noindent{\bf Case 2.} If $(i, j)\in \cbr{\SUPP(\tS)\backslash \SUPP(S^0)} \cap \SUPP(\barS^0)$, then
\begin{equation}
\begin{aligned}\label{d:2}
|\tS_{i, j} - S^0_{i, j}| = |\tS_{i, j}|
\leq  |\tDelta_{i, j}| + |\ttR_{i, j}|
\leq  |\hDelta_{i, j}^0| + \|\tDelta - \hDelta^0\|_{\infty, \infty} + \|\tR\|_{\infty, \infty}.
\end{aligned}
\end{equation}
We claim that
\begin{align}\label{d:3}
|\hDelta_{i, j}^0| = |\barS^0_{i, j}|\leq \|\barS^0 - \tS\|_{\infty, \infty}.
\end{align}
Consider otherwise. Since $\tS$ has an $\alpha$-fraction of nonzero entries per row and column, $\barS^0 - \tS$ differs from $\barS^0$ on at most $\alpha$-fraction positions per row and column. If $|\barS^0_{i, j}|> \|\barS^0 - \tS\|_{\infty, \infty}$, then $\barS^0_{i, j}$ is one of the largest $\alpha d$ entries in the $i$-th row and $j$-th column of $\barS^0$. Furthermore, it is one of the largest $\halpha d$ entries, since $\halpha\geq \alpha$. This contradicts the assumption that $(i, j)\notin\SUPP(S^0)$. Therefore, from \eqref{d:2} and \eqref{d:3},
\begin{align}\label{d:4}
|\tS_{i, j} - S^0_{i, j}|\leq \|\barS^0 - \tS\|_{\infty, \infty} + \|\tDelta - \hDelta^0\|_{\infty, \infty} + \|\tR\|_{\infty, \infty}.
\end{align}
Next, we bound $\|\barS^0 - \tS\|_{\infty, \infty}$. We have two subcases. For any $(k, l)\in\SUPP(\barS^0)$,
\begin{align}\label{d:5}
|\barS^0_{k, l} - \tS_{k, l}|
=  |\hDelta_{k, l}^0 - \tDelta_{k, l} + \ttR_{k, l}|
\leq \|\hDelta^0 - \tDelta\|_{\infty, \infty} + \|\tR\|_{\infty, \infty}.
\end{align}
For any $(k, l)\in\SUPP(\tS)\backslash\SUPP(\barS^0)$,
\begin{align}\label{d:6}
|\barS^0_{k, l} - \tS_{k, l}| = |\tS_{k, l}|\leq |\hDelta_{k, l}^0| + \|\tDelta - \hDelta^0\|_{\infty, \infty} + \|\tR\|_{\infty, \infty}.
\end{align}
In this case, we claim
\begin{align}
\label{d:6a}
|\hDelta_{k, l}^0| \leq \|\hDelta^0 - \tS\|_{\infty, \infty}.
\end{align}
Otherwise, assume $|\hDelta_{k, l}^0| > \|\hDelta^0 - \tS\|_{\infty, \infty}$. Then $\hDelta_{k, l}^0$ is one of the largest $s$ entries of $\hDelta^0$, since $\tS$ only has $s$ nonzero entries overall and $\hDelta^0 - \tS$ differs from $\hDelta^0$ on at most $s$ positions. Moreover, since $\hs\geq s$, $(k, l)\in\SUPP(\barS^0)$, which contradicts
the condition. By \eqref{d:6} and~\eqref{d:6a},
\begin{equation}
\begin{aligned}\label{d:7}
|\barS^0_{k, l} - \tS_{k, l}|
\leq \|\hDelta^0 - \tS\|_{\infty, \infty} + \|\tDelta - \hDelta^0\|_{\infty, \infty} + \|\tR\|_{\infty, \infty}
\leq 2\|\tDelta - \hDelta^0\|_{\infty, \infty} + 2\|\tR\|_{\infty, \infty}.
\end{aligned}
\end{equation}
Combining \eqref{d:5} and \eqref{d:7},
\begin{align}\label{d:8}
\|\barS^0 - \tS\|_{\infty, \infty}\leq 2\|\tDelta - \hDelta^0\|_{\infty, \infty} + 2\|\tR\|_{\infty, \infty}.
\end{align}
Finally, plugging \eqref{d:8} back into \eqref{d:4},
\begin{align}\label{d:9}
|\tS_{i, j} - S^0_{i, j}|\leq 3\|\tDelta - \hDelta^0\|_{\infty, \infty} + 3\|\tR\|_{\infty, \infty}.
\end{align}
\noindent{\bf Case 3.} If $(i, j)\in \SUPP(\tS)\backslash\cbr{\SUPP(S^0)\cup\SUPP(\barS^0)}$, then \eqref{d:7} gives us
\begin{align}\label{d:10}
|\tS_{i, j} - S^0_{i, j}| = |\tS_{i, j}|\leq 2\|\tDelta - \hDelta^0\|_{\infty, \infty} + 2\|\tR\|_{\infty, \infty},
\end{align}
since $\SUPP(\tS)\backslash\cbr{\SUPP(S^0)\cup\SUPP(\barS^0)}\subseteq\SUPP(\tS)\backslash\SUPP(\barS^0)$. Combining \eqref{d:1}, \eqref{d:9}, \eqref{d:10},
\begin{align}\label{d:11}
\|S^0 - \tS\|_{\infty, \infty} \leq 3\|\tDelta - \hDelta^0\|_{\infty, \infty} + 3\|\tR\|_{\infty, \infty}.
\end{align}
From Lemma \ref{aux:lem:12}, with probability at least $1 - 8/d^2$,
\begin{equation}\label{d:11a}
\begin{aligned}
\|\tDelta - \hDelta^0\|_{\infty, \infty}
& \leq \|\tOmegax - (\ttSigmax)^{-1}\|_{\infty, \infty} + \|\tOmegay - (\ttSigmay)^{-1}\|_{\infty, \infty}\\
& \leq 4\cbr{ \|(\tOmegax)^{1/2}\|_1^2 + \|(\tOmegay)^{1/2}\|_1^2 } \rbr{\frac{\log d}{n_X\wedge n_Y}}^{1/2}.
\end{aligned}
\end{equation}
Since $\tR = \tL\tXi\tLT$ with $\tL\in\mU(\beta)$, $\|\tR\|_{\infty, \infty}\leq \|\tXi\|_{\infty, \infty}\|\tLT\|_{2, \infty}^2\leq {\beta r\sigma_1^\tR}/{d}$, and further
\begin{align}\label{d:12}
\|S^0 - \tS\|_{\infty, \infty}
\leq 12 \cbr{\|(\tOmegax)^{1/2}\|_1^2 + \|(\tOmegay)^{1/2}\|_1^2} \rbr{\frac{\log d}{n_X\wedge n_Y}}^{1/2} + \frac{3\beta r\sigma_1^\tR}{d}.
\end{align}
With this and
\begin{align*}
\|S^0 - \tS\|_2\leq \|S^0 - \tS\|_F\leq& \rbr{2\hs}^{1/2}\|S^0 - \tS\|_{\infty, \infty},
\end{align*}
we complete the proof.

\section{Proofs of Main Theorems}\label{appen:pf:main:thm}

\subsection{Proof of Theorem \ref{thm:1}}

We show that, under assumptions of Theorem \ref{thm:1}, we can apply Lemma \ref{lem:4} by replacing $\rbr{\Upsilon_i}_{i=2}^4$ with its corresponding orders. First, we check the conditions on the step sizes. Since $\Upsilon_2\asymp \sigma_1^X\sigma_1^Y\Longrightarrow
\kappa_X\kappa_Y\Upsilon_2 \asymp
\kappa_X\kappa_Y\sigma_1^X\sigma_1^Y$, we immediately see that the conditions on the step sizes in Lemma \ref{lem:4} are satisfied. Furthermore, since $1 + {8\Upsilon_2^2}/\rbr{\Upsilon_3^2\eta_1^2} \asymp{\Upsilon_2^2}/\rbr{\Upsilon_3^2\eta_1^2} \asymp
\rbr{\kappa_X\kappa_Y}^4$, the conditions on $\gamma_1$ and $\gamma_2$ are also satisfied. For the sparsity proportion $\alpha$, since $\gamma_2\asymp \kappa_X^4\kappa_Y^4$, we have ${\Upsilon_3^2}/\rbr{\Upsilon_2^2\Upsilon_4\gamma_2} \asymp T_4$, which implies the condition on $\alpha$ in Lemma \ref{lem:4} is satisfied. Since ${\Upsilon_4}/{\rbr{\sigma_1^X\sigma_1^Y}^2} \asymp {T_3}$, condition (a) in Theorem \ref{thm:1} implies the sample complexity in Lemma \ref{lem:4}. Finally, we verify that the condition (c) in Lemma \ref{lem:4} holds for all iterations from $0$ to $k$. For $k = 0$ the condition is satisfied, since, for any constant $C_1>0$,
\begin{align*}
TD(S^0, U^0)\leq \Psi^2/C_1 \Longrightarrow U^0 \in \{U\in\mR^{d\times r}: \Pi^2(U, \tU)\leq \Psi^2/C_1\}.
\end{align*}
Therefore, we can apply Lemma \ref{lem:4} and \ref{lem:5} for $k = 0$. Let $\rho = 1 - \rbr{\sigma_r^\tR\Upsilon_3\eta_2}/\rbr{2\Upsilon_2}$. Since ${\gamma_2}\rbr{\kappa_X\kappa_Y\Upsilon_3}^{-1}\rbr{\|\tOmegay\|_1\|\tSigmax\|_1 + \|\tOmegax\|_1\|\tSigmay\|_1}^2 \asymp T_1$ and
$\rbr{\kappa_X\kappa_Y\Upsilon_3}^{-1}\rbr{\kappa_X\sigma_1^Y + \kappa_Y\sigma_1^X}^2\asymp T_2$,
\begin{align*}
TD(S^1, U^1)\leq \rho TD(S^0, U^0) + C_2\frac{T_1s\log d + T_2rd}{\sigma_1^\tR\rbr{n_X\wedge n_Y}},
\end{align*}
with probability at least $1 - C_3/d^2$ for constants $C_2, C_3$. Let $\tau' = T_1s\log d + T_2rd$. Since $1 - \rho\asymp 1/\rbr{\kappa_X^2\kappa_Y^2\kappa_\tR}$ and $(n_X\wedge n_Y) \gtrsim (\sigma_r^\tR)^{-2} \kappa_X^4\kappa_Y^4\tau'$, we can bound ${C_2\tau'}/\rbr{\sigma_1^\tR(n_X\wedge n_Y)} \leq { (1 - \rho)\Psi^2}/{C_1}$. Therefore, we further get $TD(S^1, U^1)\leq \Psi^2/C_1$. Moreover, since $U^1\in \mU(4\beta\|U^0\|_2^2)$ by the Algorithm \ref{alg:2} and $\|U^0\|_2$ satisfies \eqref{a:17}, $U^1\in\mU(9\beta\sigma_1^\tR)$. Therefore the condition (c) in Lemma \ref{lem:4} holds for $(S^1, U^1)$. Applying this reasoning iteratively and noting, for any iteration $k$,
\begin{align*}
TD(S^{k}, U^{k})\leq \rho TD(S^{k-1}, U^{k-1}) + \frac{C_2\tau'}{\sigma_1^\tR(n_X\wedge n_Y)},
\end{align*}
we have
\begin{align*}
TD(S^k, U^k)\leq \rho^kTD(S^0, U^0) + \frac{C_2\tau'}{(1-\rho)\sigma_1^\tR(n_X\wedge n_Y)},
\end{align*}
with probability at least $1 - C_3/d^2$. Plugging the order of $\rho$ completes the proof.

\subsection{Proof of Theorem \ref{thm:2}}

The proof follows in few steps. First, we use Lemma \ref{lem:6}
to upper bound $\|S^0 - \tS\|_F$. Next, we upper bound $\|R^0 - \tR\|_2$ and show that under the sample size assumption the bound guarantees $\|R^0 - \tR\|_2\leq \sigma_r^\tR/4$. Finally, we upper bound $\Pi^2(U^0, \tU)$ and $TD(S^0, U^0)$.

It directly follows from the algorithm that $S^0$ is symmetric.
By Lemma \ref{lem:6}, with probability at least $1 - C_1/d^2$,
\begin{equation}\label{d:15}
\|S^0 - \tS\|_F\leq C_2{\hs}^{1/2}\sbr{\cbr{\|(\tOmegax)^{1/2}\|_1^2 + \|(\tOmegay)^{1/2}\|_1^2}\rbr{\frac{\log d}{n_X\wedge n_Y}}^{1/2} + \frac{\beta r\sigma_1^\tR}{d}}.
\end{equation}
We bound $\|R^0 - \tR\|_2$ as
\begin{equation}\label{d:13}
\|R^0 - \tR\|_2
= \|\hDelta^0 - S^0 - \tDelta + \tS\|_2
\leq \|\tS - S^0\|_2 + \|(\ttSigmax)^{-1} - \tOmegax\|_2 + \|(\ttSigmay)^{-1} - \tOmegay\|_2.
\end{equation}
For the term $\|(\ttSigmax)^{-1} - \tOmegax\|_2$, we know
\begin{align*}
\|(\ttSigmax)^{-1} - \tOmegax\|_2 = & \|\frac{n_X-d-2}{n_X}(\hSigmax)^{-1} - (\tSigmax)^{-1}\|_2 \\
= &\frac{n_X - d -2}{n_X}\|(\hSigmax)^{-1}\big(\tSigmax - \frac{n_X}{n_X - d-2}\hSigmax\big)\tOmegax\|_2\\
\leq &\frac{n_X - d - 2}{n_X} \frac{2}{(\sigma_d^X)^2}\|\frac{n_X}{n_X - d - 2}\hSigmax - \tSigmax\|_2\\
\leq & \frac{n_X - d - 2}{n_X} \frac{2}{(\sigma_d^X)^2}\rbr{\|\hSigmax - \tSigmax\|_2 + \frac{d+2}{n_X-d-2}\|\hSigmax\|_2},
\end{align*}
where the first inequality follows from H\"older's inequality and \eqref{a:5}. From Lemma \ref{aux:lem:7}, with probability $1 - C_3/d^2$,
\begin{align}\label{d:14}
\|(\ttSigmax)^{-1} - \tOmegax\|_2 + \|(\ttSigmay)^{-1} - \tOmegay\|_2
\leq C_4 \rbr{\frac{\kappa_X}{\sigma_d^X} + \frac{\kappa_Y}{\sigma_d^Y}} \rbr{\frac{d}{n_X\wedge n_Y}}^{1/2},
\end{align}
Combining \eqref{d:15}, \eqref{d:13}, and \eqref{d:14},
\begin{align}\label{d:16}
\|R^0 - \tR\|_2\leq C_5 \rbr{{\hs}^{1/2}\varphi_1+ \varphi_2}
\end{align}
with probability at least $1 - C_6/d^2$, where
\begin{align*}
\varphi_1 = \cbr{\|(\tOmegax)^{1/2}\|_1^2 + \|(\tOmegay)^{1/2}\|_1^2}\rbr{\frac{\log d}{n_X\wedge n_Y}}^{1/2} + \frac{\beta r\sigma_1^\tR}{d}, \quad
\varphi_2 = \rbr{\frac{\kappa_X}{\sigma_d^X} + \frac{\kappa_Y}{\sigma_d^Y}}\rbr{\frac{d}{n_X\wedge n_Y}}^{1/2}.
\end{align*}
Under the assumptions of the theorem, $\|R^0 - \tR\|_2\leq \sigma_r^\tR/4$. From Lemma \ref{lem:1}, $\Lambda^0 = \tLambda$.
By Weyl's inequality, $\|\tR\|_2/2 \leq \|\barU\|_2^2 = \|R^0\|_2\leq {3}\|\tR\|_2/2$. Therefore, $U^0\in\mU(4\beta\|\barU\|_2^2)\subseteq\mU(9\beta\sigma_1^\tR)$ and $\tU \in\mU(4\beta\|\barU\|_2^2)  = \mC$. Moreover,
\begin{multline}
\label{d:17}
\|\barU^0\tLambda \barU^{0 \T} - \tU\tLambda\tUT\|_2
= \|\barU^0\Lambda^0 \barU^{0 \T} - \tU\tLambda\tUT\|_2 = \|L_r^0\Xi_r^0L_r^{0 \T} - \tR\|_2  \\
\leq \|L_r^0\Xi_r^0L_r^{0 \T} - R^0\|_2 + \|R^0 - \tR\|_2
\leq 2\|R^0 - \tR\|_2\leq \sigma_r^\tR/2,
\end{multline}
since $\|L_r^0\Xi_r^0L_r^{0 \T} - R^0\|_2 \leq \sigma_{r+1}(R^0) = \sigma_{r+1}(R^0) - \sigma_{r+1}(\tR)\leq \|R^0 - \tR\|_2$
with $\sigma_{r+1}(R^0)$ denoting the $(r+1)$-th singular value of $R^0$. Suppose $\Pi^2(\barU^0, \tU) = \|\barU^0 - \tU Q\|_F^2$, then
\begin{align*}
\Pi^2(U^0, \tU)\leq \|U^0 - \tU Q\|_F^2 = \|\P_\mC(\barU^0) - \P_\mC(\tU Q)\|_F^2\leq \|\barU^0 - \tU Q\|_F^2=\Pi^2(\barU^0, \tU).
\end{align*}
By Lemma \ref{prop:1} and \eqref{d:17},
\begin{multline}\label{d:18}
\Pi^2(\barU^0, \tU)
\leq \frac{1}{(\surd{2} - 1)\sigma_r^\tR}\|\barU^0\tLambda \barU^{0 \T} - \tU\tLambda\tUT\|_F^2 \\
\leq \frac{2r}{(\surd{2}-1)\sigma_r^\tR}\|\barU^0\tLambda \barU^{0 \T} - \tU\tLambda\tUT\|_2^2
\leq  \frac{8r}{(\surd{2}-1)\sigma_r^\tR}\|R^0 - \tR\|_2^2.
\end{multline}
Combining Lemma \ref{lem:6} with \eqref{d:18}, with probability at least $1 - C_7/d^2$,
\begin{align*}
TD(S^0, U^0)
= \frac{\|S^0 - \tS\|_F^2}{\sigma_1^\tR} + \Pi^2(U^0, \tU)
\leq C_8 \frac{r}{\sigma_r^\tR}\rbr{{\hs}^{1/2}\varphi_1 + \varphi_2}^2.
\end{align*}
Since
\begin{align*}
\frac{r}{\sigma_r^\tR}\rbr{{\hs}^{1/2}\varphi_1 + \varphi_2}^2
\asymp \frac{r\hs \varphi_1^2}{\sigma_r^\tR} + \frac{r\varphi_2^2}{\sigma_r^\tR}
\asymp \frac{r\rbr{T_5\hs\log d + T_6 d}}{\sigma_r^\tR(n_X\wedge n_Y)} + \frac{\hs\beta^2r^3\kappa_\tR\sigma_1^\tR}{d^2},
\end{align*}
the first part of the theorem follows. The second part follows, since $r\rbr{{\hs}^{1/2}\varphi_1 + \varphi_2}^2/\sigma_r^\tR \leq \Psi^2\asymp\sigma_r^\tR/\rbr{\kappa_X\kappa_Y}^2$ under the sample complexity in \eqref{cond:sample3}. This completes the proof.

\section{Complementary Lemmas}\label{com_lem}

\begin{lemma}\label{lem:lem_2_I1}

Under the conditions of Lemma \ref{lem:2}, we have
\begin{equation*}
\begin{aligned}
\I_1 & \geq \frac{9}{8} \frac{\Upsilon_3}{\Upsilon_2}\|S^k - \tS\|_F^2 + \frac{4}{\Upsilon_2}\|\nabla_S\mbL_n(S^k, U^k, \Lambda^0) - \nabla_S\mbL_n(\tS, U^k, \Lambda^0)\|_F^2 \\
& \quad - \frac{4\Upsilon_2}{9\Upsilon_3}\cbr{C_1\cdot(1 + \gamma_1)\alpha r \sigma_1^\tR \Upsilon_1 \Pi^2(U^k, \tU)  + (1 +\gamma_2)s\|\nabla_S\mL_n(\tS, \tR)\|_{\infty, \infty}^2}
\end{aligned}
\end{equation*}
with probability $1 - C_2/d^2$ for some fixed constants $C_1, C_2>0$ large enough.
\end{lemma}

\begin{proof}
	We start by bounding the sample covariance matrices $\hSigmax$ and $\hSigmay$. Let $\delta = 1/(2d^2)$ in Lemma \ref{aux:lem:7}. Then, for some constant $C_1>0$,
	\begin{align*}
	\Pr\cbr{\|\hSigmax - \tSigmax\|_2\leq C_1\|\tSigmax\|_2\rbr{\frac{d}{n_X}}^{1/2} \text{\ \ and\ \ } \|\hSigmay - \tSigmay\|_2\leq C_1\|\tSigmay\|_2\rbr{\frac{d}{n_Y}}^{1/2}}\geq 1-\frac{1}{d^2}.
	\end{align*}
	We will on the event
	\begin{align}\label{a:5}
	{\cal  E} = \cbr{ {\sigma^i_d}/{2}\leq \sigma_{\min}(\hSigma_i)\leq \sigma_{\max}(\hSigma_i)\leq {3\sigma^i_1}/{2}, \text{\ for\ } i\in\{X, Y\} },
	\end{align}
	which occurs with probability at least $1 - 1/d^2$ by Lemma \ref{aux:lem:7}, since $n_X\gtrsim \kappa_X^2d$ and $n_Y\gtrsim \kappa_Y^2d$.
	By definition of $\I_1$ and \eqref{a:1},
	\begin{align}\label{a:6}
	\I_1 &= \bLD S^k - \tS, \P_{\Omega^k}\cbr{\nabla_S\mbL_n(S^k, U^k, \Lambda^0)} \bRD= \LD S^k - \tS, \nabla_S\mbL_n(S^k, U^k, \Lambda^0) \RD \nonumber\\
	&= \LD S^k - \tS, \nabla_S\mbL_n(S^k, U^k, \Lambda^0) - \nabla_S\mbL_n(\tS, U^k, \Lambda^0)\RD + \LD S^k - \tS,  \nabla_S\mbL_n(\tS, \tU, \tLambda) \RD \nonumber\\
	& \qquad + \LD S^k - \tS, \nabla_S\mbL_n(\tS, U^k, \Lambda^0) - \nabla_S\mbL_n(\tS, \tU, \tLambda)\RD \nonumber\\
	&= \LD S^k - \tS, \hSigmax(S^k - \tS)\hSigmay\RD + \LD S^k - \tS,  \nabla_S\mbL_n(\tS, \tU, \tLambda) \RD \nonumber\\
	& \qquad  + \frac{1}{2}\LD S^k - \tS, \hSigmax(U^k\Lambda^0U^{k \T} - \tU\tLambda\tUT)\hSigmay  + \hSigmay(U^k\Lambda^0U^{k \T} - \tU\tLambda\tUT)\hSigmax \RD \nonumber\\
	&= \I_{11} + \I_{12} + \I_{13}.
	\end{align}
	We bound $\I_{11}$, $\I_{12}$, $\I_{13}$ from below separately. By Lemma \ref{aux:lem:8} and on the event ${\cal E}$ in \eqref{a:5},
	\begin{align}\label{a:7}
	\I_{11}
	& \geq \frac{9\Upsilon_3}{4\Upsilon_2}\|S^k - \tS\|_F^2 +
	\frac{1}{\Upsilon_2}\|\hSigmax(S^k - \tS)\hSigmay + \hSigmay (S^k - \tS)\hSigmax\|_F^2 \nonumber\\
	& = \frac{9\Upsilon_3}{4\Upsilon_2}\|S^k - \tS\|_F^2
	+ \frac{4}{\Upsilon_2}\|\nabla_S\mbL_n(S^k, U^k, \Lambda^0) - \nabla_S\mbL_n(\tS, U^k, \Lambda^0)\|_F^2.
	\end{align}
	Using H\"older's inequality, for any $C_{12} > 0$ to be determined
	later,
	\begin{align}\label{a:8}
	\I_{12}
	& \geq  - \|S^k - \tS\|_{1, 1}\|\nabla_S\mbL_n(\tS, \tU, \tLambda)\|_{\infty, \infty} \nonumber\\
	& \geq  - \cbr{(\gamma_2 + 1)s}^{1/2}\|S^k - \tS\|_F \|\nabla_S\mbL_n(\tS, \tU, \tLambda)\|_{\infty, \infty} \nonumber\\
	& \geq - \frac{\cbr{(\gamma_2 + 1)s}^{1/2}}{2C_{12}}\|\nabla_S\mL_n(\tS, \tR)\|_{\infty, \infty}^2 - \frac{C_{12}\cbr{(\gamma_2 + 1)s}^{1/2}}{2}\|S^k - \tS\|_F^2,
	\end{align}
	where the second inequality is due to the fact that
	$\SUPP(S^k - \tS)\subseteq\barOmega^k$ and $|\barOmega^k|\leq (\gamma_2+1)s$,
	and the third inequality uses the equation $\nabla_S\mbL_n(\tS, \tU, \tLambda) = \nabla_S\mL_n(\tS, \tR)$
	with $\barOmega^k$ defined in \eqref{a:2}. Similarly, for any $C_{13}>0$,
	\begin{align}\label{a:9}
	\I_{13}
	& \geq  - \frac{1}{2}\|S^k - \tS\|_F \big\|\P_{\barOmega^k}\cbr{\hSigmax(U^k\Lambda^0U^{k \T} - \tU\tLambda\tUT)\hSigmay} \nonumber\\
	& \quad + \P_{\barOmega^k}\cbr{\hSigmay(U^k\Lambda^0U^{k \T} - \tU\tLambda\tUT)\hSigmax}\big\|_F \nonumber\\
	& \geq  - \|S^k - \tS\|_F \big\|\P_{\barOmega^k}\cbr{\hSigmax(U^k\Lambda^0U^{k \T} - \tU\tLambda\tUT)\hSigmay}\big\|_F \nonumber\\
	& \geq  - \frac{C_{13}}{2}\|S^k - \tS\|_F^2 - \frac{1}{2C_{13}}\big\|\P_{\barOmega^k}\cbr{\hSigmax(U^k\Lambda^0U^{k \T} - \tU\tLambda\tUT)\hSigmay}\big\|_F^2,
	\end{align}
	where the second inequality is due to the fact that $\barOmega^k$ is symmetric and, hence, from Lemma \ref{aux:lem:6},
	\begin{align*}
	\P_{\barOmega^k}\cbr{\hSigmay(U^k\Lambda^0U^{k \T} - \tU\tLambda\tUT)\hSigmax} = \sbr{\P_{\barOmega^k}\cbr{\hSigmax(U^k\Lambda^0U^{k \T} - \tU\tLambda\tUT)\hSigmay}}^\T.
	\end{align*}
	We further upper bound the second term on the right hand side of \eqref{a:9}. Let $Q^k\in\mQ^{r\times r}_{r_1}$ be the optimal rotation of $\tU$ such that $\Pi^2(U^k, \tU) = \|U^k - \tU Q^k\|_F^2$ and define
	\begin{gather*}
	\hA^k = \hSigmax U^k, \quad \htA = \hSigmax \tU Q^k, \quad \tA = \tSigmax\tU Q^k,\\
	\hB^k = \hSigmay U^k, \quad \htB = \hSigmay \tU Q^k, \quad \tB = \tSigmay\tU Q^k.
	\end{gather*}
	Since $\Lambda^0 = \tLambda$,
	\begin{align}\label{a:11}
	\big\|& \P_{\barOmega^k}  \cbr{\hSigmax(U^k\Lambda^0U^{k \T} - \tU\tLambda\tUT)\hSigmay}\big\|_F^2 \nonumber\\
	& = \big\|\P_{\barOmega^k}\cbr{\hSigmax(U^k\Lambda^0U^{k \T} - \tU Q^k\tLambda Q^{k \T}\tUT)\hSigmay}\big\|_F^2 \nonumber\\
	& = \|\P_{\barOmega^k}\big(\hA^k\tLambda\hB^{k \T} - \htA\tLambda\htBT\big)\|_F^2 \nonumber\\
	& = \sum_{(i, j)\in \barOmega^k} |(\hA^k_{i, \cdot})^\T\tLambda\hB^k_{j, \cdot} - (\htA_{i, \cdot})^\T\tLambda\htB_{j, \cdot}|^2 \nonumber\\
	& = \sum_{(i, j)\in \barOmega^k} \big|(\hA^k_{i, \cdot} - \htA_{i, \cdot})^\T\tLambda(\hB^k_{j, \cdot} - \htB_{j, \cdot}) + (\htA_{i, \cdot})^\T\tLambda(\hB^k_{j, \cdot} - \htB_{j, \cdot}) + (\htB_{j, \cdot})^\T\tLambda(\hA^k_{i, \cdot} - \htA_{i, \cdot})\big|^2 \nonumber\\
	& \leq  3\sum_{(i, j)\in \barOmega^k}\rbr{\|\hA^k_{i, \cdot} - \htA_{i, \cdot}\|_2^2 \|\hB^k_{j, \cdot} - \htB_{j, \cdot}\|_2^2 + \|\htA_{i, \cdot}\|_2^2\|\hB^k_{j, \cdot} - \htB_{j, \cdot}\|_2^2 + \|\htB_{j, \cdot}\|_2^2\|\hA^k_{i, \cdot} - \htA_{i, \cdot}\|_2^2} \nonumber\\
	& \leq  \frac{3}{2}\max_{i\in[d]}\|\hA^k_{i, \cdot} - \htA_{i, \cdot}\|_2^2\sum_{j\in[d]}\sum_{i\in\barOmega^k_{\cdot, j}}\|\hB^k_{j, \cdot} - \htB_{j, \cdot}\|_2^2 + \frac{3}{2}\max_{j\in[d]}\|\hB^k_{j, \cdot} - \htB_{j, \cdot}\|_2^2\sum_{i\in[d]}\sum_{j\in\barOmega^k_{i, \cdot}}\|\hA^k_{i, \cdot} - \htA_{i, \cdot}\|_2^2 \nonumber\\
	& \qquad + 3\max_{i\in[d]}\|\htA_{i, \cdot}\|_2^2\sum_{j\in[d]}\sum_{i\in\barOmega^k_{\cdot, j}}\|\hB^k_{j, \cdot} - \htB_{j, \cdot}\|_2^2 + 3\max_{j\in[d]}\|\htB_{j, \cdot}\|_2^2\sum_{i\in[d]}\sum_{j\in\barOmega^k_{i, \cdot}}\|\hA^k_{i, \cdot} - \htA_{i, \cdot}\|_2^2,
	\end{align}
	where $\barOmega^k_{i, \cdot} = \{j\mid (i, j)\in\barOmega^k\}$ and
	$\barOmega^k_{\cdot, j} = \{i\mid (i, j)\in\barOmega^k\}$.
	For any $i, j\in[d]$,
	$|\barOmega^k_{i, \cdot}| \vee |\barOmega^k_{\cdot, j}|\leq (1 + \gamma_1)\alpha d$  and, therefore,
	\begin{align}\label{a:12}
	\|\P_{\barOmega^k}&\cbr{\hSigmax(U^k\Lambda^0U^{k \T} - \tU\tLambda\tUT)\hSigmay}\|_F^2 \nonumber\\
	& \leq \frac{3(1+\gamma_1)\alpha d}{2}\|(\hA^{k} - \htA)^\T\|_{2, \infty}^2\|\hB^k - \htB\|_F^2 + \frac{3(1 + \gamma_1)\alpha d}{2}\|(\hB^k - \htB)^\T\|_{2, \infty}^2\|\hA^k - \htA\|_F^2 \nonumber\\
	& \qquad  + 3(1 + \gamma_1)\alpha d\|\htAT\|_{2, \infty}^2\|\hB^k - \htB\|_F^2 + 3(1+\gamma_1)\alpha d\|\htBT\|_{2, \infty}^2\|\hA^k - \htA\|_F^2 \nonumber\\
	& = 3(1+\gamma_1)\alpha d\|\hB^k - \htB\|_F^2\cbr{\frac{1}{2}\|(\hA^{k} - \htA)^\T\|_{2, \infty}^2 + \|\htAT\|_{2, \infty}^2} \nonumber\\
	& \qquad + 3(1+\gamma_1)\alpha d\|\hA^k - \htA\|_F^2\cbr{\frac{1}{2}\|(\hB^{k} - \htB)^\T\|_{2, \infty}^2 + \|\htBT\|_{2, \infty}^2}.
	\end{align}
	We bound each term involving $\hA$ above as
	\begin{align}\label{a:13}
	\|\hA^k - \htA\|_F^2
	= \|\hSigmax\big(U^k - \tU Q^k\big)\|_F^2
	\leq \frac{9(\sigma_1^X)^2}{4} \|U^k - \tU Q^k\|_F^2
	= \frac{9(\sigma_1^X)^2}{4} \Pi^2(U^k, \tU)
	\end{align}
	and
	\begin{equation}
	\begin{aligned}\label{a:14}
	\|\htAT\|_{2, \infty}
	\leq & \|\htAT - \tAT\|_{2, \infty}^2 + \|\tAT\|_{2, \infty}  = \big\|\cbr{(\hSigmax - \tSigmax)\tU Q^k}^\T\big\|_{2, \infty} + \|(\tSigmax\tU Q^k)^\T\|_{2, \infty} \\
	\leq & \big\|\cbr{(\hSigmax - \tSigmax)\tU Q^k}^\T\big\|_{2, \infty}   + \|\tSigmax\|_1\|\tUT\|_{2, \infty} \\
	 \leq &  \big\|\cbr{(\hSigmax - \tSigmax)\tU Q^k}^\T\big\|_{2, \infty}  + \|\tSigmax\|_1\rbr{\frac{\beta r\sigma_1^\tR}{d}}^{1/2},
	\end{aligned}
	\end{equation}
	where the second inequality comes from Lemma \ref{aux:lem:10} and the last inequality is due to the incoherence condition in Assumption \ref{ass:2}. From Lemma \ref{aux:lem:11}, with probability at least $1 - {2}/{d^2}$,
	\begin{align*}
	\|\big((\hSigmax - \tSigmax)\tU Q^k\big)^\T\|_{2, \infty} \leq 11 \rbr{\|\tUT\tSigmax\tU\|_2\sigma_1^X}^{1/2}\rbr{\frac{r\log d}{n_X}}^{1/2}.
	\end{align*}
	Since $\|\tU\|_2^2 \leq {\sigma_1^\tR}$, with probability at least $1 - {2}/{d^2}$,
	\begin{align}\label{a:16}
	\|\htAT\|_{2, \infty} \leq 11\rbr{r\sigma_1^\tR }^{1/2}
	\cbr{\sigma_1^X\rbr{\frac{\log d}{n_X}}^{1/2} + \|\tSigmax\|_1\rbr{\frac{\beta }{d}}^{1/2}}.
	\end{align}
	We bound $\|(\hA^k - \htA)^\T\|_{2, \infty}$ analogously. We have
	\begin{align}\label{a:18}
	\|(\hA^k - \htA)^\T\|_{2, \infty}
	&= \big\|\cbr{\hSigmax(U^k - \tU Q^k)}^\T\big\|_{2, \infty} \nonumber\\
	& \leq \big\|\cbr{(\hSigmax - \tSigmax)(U^k - \tU Q^k)}^\T\big\|_{2, \infty} + \big\|\cbr{\tSigmax(U^k - \tU Q^k)}^\T\big\|_{2, \infty} \nonumber\\
	& \leq \big\|\cbr{(\hSigmax - \tSigmax)(U^k - \tU Q^k)}^\T\big\|_{2, \infty} + \|\tSigmax\|_1\big(\|U^{k \T}\|_{2, \infty} + \|\tUT\|_{2, \infty}\big) \nonumber\\
	& \leq \big\|\cbr{(\hSigmax - \tSigmax)(U^k - \tU Q^k)}^\T\big\|_{2, \infty} + 4\|\tSigmax\|_1\rbr{\frac{\beta r\sigma_1^\tR}{d}}^{1/2},
	\end{align}
	where the second inequality is due to Lemma \ref{aux:lem:10}, the last inequality is due to the incoherence condition and assumption that $U^k\in\mU(9\beta\sigma_1^\tR)$. For the first term in \eqref{a:18}, from Lemma \ref{aux:lem:11}, with probability at least $1-2/d^2$,
	\begin{align*}
	\|\big((\hSigmax - \tSigmax)(U^k - \tU Q^k)\big)^\T\|_{2, \infty}
	& \leq 11\cbr{\|(U^k - \tU Q^k)^\T\tSigmax(U^k - \tU Q^k)\|_2\sigma_1^X}^{1/2}\rbr{\frac{r\log d}{n_X}}^{1/2}.
	\end{align*}
	Since $\Pi(U^k, \tU)\leq \surd{\sigma_1^\tR}/2$ and  $\Pi(U^k, \tU) \geq \|U^k - \tU Q^k\|_2$,
	\begin{align}\label{a:17}
	\frac{\surd{\sigma_1^\tR}}{2}\leq \surd{\sigma_1^\tR} - \Pi(U^k, \tU) \leq \|U^k\|_2\leq  \surd{\sigma_1^\tR} + \Pi(U^k, \tU)\leq \frac{3\surd{\sigma_1^\tR}}{2}.
	\end{align}
	Thus,
	\begin{align*}
	\|(U^k - \tU Q^k)^\T\tSigmax(U^k - \tU Q^k)\|_2\leq 2.5^2\sigma_1^\tR\sigma_1^X,
	\end{align*}
	and further
	\begin{align*}
	\|\big((\hSigmax - \tSigmax)(U^k - \tU Q^k)\big)^\T\|_{2, \infty}\leq 28\sigma_1^X\rbr{\frac{r\sigma_1^\tR\log d}{n_X}}^{1/2}.
	\end{align*}
	Combining with \eqref{a:18}, with probability at least $1 - 2/d^2$,
	\begin{align}\label{a:20}
	\|(\hA^k - \htA)^\T\|_{2, \infty}\leq 28\rbr{r\sigma_1^\tR}^{1/2}\cbr{\sigma_1^X\rbr{\frac{\log d}{n_X}}^{1/2} + \|\tSigmax\|_1\rbr{\frac{\beta}{d}}^{1/2}}.
	\end{align}
	Similar to \eqref{a:13}, \eqref{a:16}, \eqref{a:20}, we can bound $\|\hB^k - \htB\|_F^2$, $\|\htBT\|_{2, \infty}$, and $\|(\hB^k - \htB)^\T\|_{2,\infty}$. Putting everything together and combining with \eqref{a:12}, we know for some constant $C_2>0$, with probability at least $1 - {8}/{d^2}$,
	\begin{equation}\label{a:21}
	\|\P_{\barOmega^k}\cbr{\hSigmax(U^k\Lambda^0U^{k \T} - \tU\tLambda\tUT)\hSigmay}\|_F^2 \leq C_2 (1 + \gamma_1)\alpha r \sigma_1^\tR \Upsilon_1 \Pi^2(U^k, \tU),
	\end{equation}
	where $\Upsilon_1$ is defined \eqref{a:25}. Together with \eqref{a:9}, we have
	\begin{align}\label{a:19}
	\I_{13}\geq - \frac{C_{13}}{2}\|S^k - \tS\|_F^2  - \frac{C_2}{2C_{13}}(1 + \gamma_1)\alpha r \sigma_1^\tR \Upsilon_1 \Pi^2(U^k, \tU).
	\end{align}
	Setting $C_{12}$ and $C_{13}$ in \eqref{a:8} and \eqref{a:19} as
	\begin{align*}
	C_{12} = \frac{9}{8\cbr{(\gamma_2 + 1)s}^{1/2}} \frac{\Upsilon_3}{\Upsilon_2}, \quad
	C_{13} =  \frac{9}{8} \frac{\Upsilon_3}{\Upsilon_2},
	\end{align*}
	and combining with \eqref{a:6}, \eqref{a:7}, \eqref{a:8}, \eqref{a:19}, we have
	\begin{align}\label{a:22}
	\I_1
	& \geq \frac{9}{8} \frac{\Upsilon_3}{\Upsilon_2}\|S^k - \tS\|_F^2
	+ \frac{4}{\Upsilon_2}\|\nabla_S\mbL_n(S^k, U^k, \Lambda^0) - \nabla_S\mbL_n(\tS, U^k, \Lambda^0)\|_F^2 \\
	& \quad - \frac{4\Upsilon_2}{9\Upsilon_3}\cbr{C_2\cdot(1 + \gamma_1)\alpha r \sigma_1^\tR \Upsilon_1 \Pi^2(U^k, \tU)  + (1 +\gamma_2)s\|\nabla_S\mL_n(\tS, \tR)\|_{\infty, \infty}^2}
	\end{align}
	with probability at least $1 - C_3/d^2$ for some constant $C_2, C_3$ large enough. This completes the proof.
\end{proof}

\begin{lemma}\label{lem:lem_2_I2}
Under the conditions of Lemma~\ref{lem:2}, we have
\begin{align*}
\I_2&\leq 3\|\nabla_S\mbL_n(S^k, U^k, \Lambda^0) - \nabla_S\mbL_n(\tS, U^k, \Lambda^0)\|_F^2 + 3(1 + 2\gamma_2)s\|\nabla_S\mL_n(\tS, \tR)\|_{\infty, \infty}^2\nonumber\\
&\quad + \frac{27(\sigma_1^X\sigma_1^Y)^2}{4}\|U^k\Lambda^0U^{k \T} - \tU\tLambda \tUT\|_F^2
\end{align*}
with probability at least $1-1/d^2$.

\end{lemma}

\begin{proof}
Using the fact that $\nabla_S\mbL_n(\tS, \tU, \tLambda) = \nabla_S\mL_n(\tS, \tR)$ and $|\Omega^k|\leq (2\gamma_2 + 1)s$, by definition of $\I_2$,
\begin{equation}
\begin{aligned}\label{a:23}
\I_2 &= \big\|\P_{\Omega^k}\cbr{\nabla_S\mbL_n(S^k, U^k, \Lambda^0)} \big\|_F^2 \\
& \leq 3 \big\|\P_{\Omega^k}\cbr{\nabla_S\mbL_n(S^k, U^k, \Lambda^0) - \nabla_S\mbL_n(\tS, U^k, \Lambda^0)}\big\|_F^2 + 3\|\P_{\Omega^k}\cbr{\nabla_S\mL_n(\tS, \tR)}\|_F^2 \\
& \quad + 3\big\|\P_{\Omega^k}\cbr{\nabla_S\mbL_n(\tS, U^k, \Lambda^0) - \nabla_S\mbL_n(\tS, \tU, \tLambda)}\big\|_F^2 \\
& \leq 3\|\nabla_S\mbL_n(S^k, U^k, \Lambda^0) - \nabla_S\mbL_n(\tS, U^k, \Lambda^0)\|_F^2 + 3(1 + 2\gamma_2)s\|\nabla_S\mL_n(\tS, \tR)\|_{\infty, \infty}^2 \\
& \quad + \frac{3}{4}\big\|\P_{\Omega^k}\cbr{\hSigmax(U^k\Lambda^0U^{k \T} - \tU\tLambda \tUT)\hSigmay + \hSigmay(U^k\Lambda^0U^{k \T} - \tU\tLambda \tUT)\hSigmax}\big\|_F^2 \\
& \leq 3\|\nabla_S\mbL_n(S^k, U^k, \Lambda^0) - \nabla_S\mbL_n(\tS, U^k, \Lambda^0)\|_F^2 + 3(1 + 2\gamma_2)s\|\nabla_S\mL_n(\tS, \tR)\|_{\infty, \infty}^2 \\
& \quad + 3\|\hSigmax(U^k\Lambda^0U^{k \T} - \tU\tLambda \tUT)\hSigmay\|_F^2,
\end{aligned}
\end{equation}
where the second inequality is due to the fact that $|\Omega^k|\leq ( 1 + 2\gamma_2)s$. For the last term above, using H\"older's inequality and event $\cal E$ in \eqref{a:5},
\begin{equation}
\begin{aligned}\label{a:24}
\|\big(\hSigmax(U^k\Lambda^0U^{k \T} - \tU\tLambda \tUT)\hSigmay\big)\|_F^2 \leq \frac{9(\sigma_1^X\sigma_1^Y)^2}{4}\|U^k\Lambda^0U^{k \T} - \tU\tLambda \tUT\|_F^2,
\end{aligned}
\end{equation}
with probability at least $1-1/d^2$. The proof follows by combining the last two displays.
\end{proof}

\begin{lemma}\label{lem:lem_3_I3}

Under the conditions of Lemma \ref{lem:3}, for any $C_{31}, C_{32,1}, C_{32,2}, C_{33,1}, C_{33,2}>0$
\begin{align*}
\I_3\geq& \cbr{\frac{9\Upsilon_3}{4\Upsilon_2} - \frac{\rbr{2r}^{1/2}C_{33,1}}{2}}\|R^k - \tR\|_F^2 - \rbr{\frac{C_{32,1}}{2} + \frac{9\sigma_1^X\sigma_1^YC_{32,2}}{8}}\|S^k - \tS\|_F^2 \\
& + \rbr{\frac{4}{\Upsilon_2} - \frac{1}{2C_{31}}}\|\nabla_R\mL_n(S^k, R^k) - \nabla_R\mL_n(S^k, \tR)\|_F^2 \nonumber\\
&  - \frac{\rbr{2r}^{1/2}(C_{33,1} + C_{33,2})}{2C_{33,1}C_{33,2}}\|\nabla_R\mL_n(\tS, \tR)\|_2^2 - \frac{C_1}{2C_{32,1}}(1 + \gamma_1)\alpha r\sigma_1^\tR\Upsilon_1\Pi^2(U^k, \tU)\\
& - \rbr{\frac{C_{31} + \rbr{2r}^{1/2}C_{33,2}}{2} + \frac{9\sigma_1^X\sigma_1^Y}{8C_{32,2}}}\Pi^4(U^k, \tU),
\end{align*}
with probability at least $1-C_2/d^2$ for some fixed constants $C_1, C_2>0$ large enough.

\end{lemma}

\begin{proof}

	Let $R^k = U^k\Lambda^0U^{k \T} = U^k\tLambda U^{k \T}$ and $\Theta^k = U^k - \tU Q^k$. Using formulas in \eqref{a:1},
	\begin{align}\label{a:29}
	\I_3 &= \LD U^k - \tU Q^k, \nabla_U\mbL_n(S^k, U^k, \Lambda^0)\RD = \LD U^k - \tU Q^k, 2\nabla_R\mL_n(S^k, R^k)U^k \Lambda^0\RD \nonumber\\
	& = 2\LD (U^k - \tU Q^k)\Lambda^0 U^{k \T}, \nabla_R\mL_n(S^k, R^k)\RD  = \LD R^k - \tR + \Theta^k\tLambda\Theta^{k \T},  \nabla_R\mL_n(S^k, R^k)\RD \nonumber\\
	& = \I_{31} + \I_{32} + \I_{33},
	\end{align}
	where
	\begin{align*}
	\I_{31} & = \LD R^k - \tR + \Theta^k\tLambda\Theta^{k \T},  \nabla_R\mL_n(S^k, R^k) - \nabla_R\mL_n(S^k, \tR)\RD,\\
	\I_{32} &= \LD R^k - \tR + \Theta^k\tLambda\Theta^{k \T},  \nabla_R\mL_n(S^k, \tR) - \nabla_R\mL_n(\tS, \tR)\RD,\\
	\I_{33} &= \LD R^k - \tR + \Theta^k\tLambda\Theta^{k \T},   \nabla_R\mL_n(\tS, \tR)\RD.
	\end{align*}
	We bound the three terms separately. First, using \eqref{a:1}
	\begin{align*}
	\I_{31}
	&= \LD R^k - \tR + \Theta^k\tLambda\Theta^{k \T},  \frac{1}{2}\hSigmax(R^k - \tR)\hSigmay + \frac{1}{2}\hSigmay(R^k - \tR)\hSigmax\RD\\
	&= \LD R^k - \tR, \hSigmax(R^k - \tR)\hSigmay\RD + \LD \Theta^k\tLambda\Theta^{k \T}, \frac{1}{2}\hSigmax(R^k - \tR)\hSigmay + \frac{1}{2}\hSigmay(R^k - \tR)\hSigmax\RD.
	\end{align*}
	By Lemma \ref{aux:lem:8},
	\begin{align*}
	\LD R^k - \tR, \hSigmax(R^k - &\tR)\hSigmay\RD\\
	& \geq \frac{9\Upsilon_3}{4\Upsilon_2}\|R^k - \tR\|_F^2 + \frac{1}{\Upsilon_2} \|\hSigmax(R^k - \tR)\hSigmay + \hSigmay (R^k - \tR)\hSigmax\|_F^2\\
	&= \frac{9\Upsilon_3}{4\Upsilon_2}\|R^k - \tR\|_F^2 + \frac{4}{\Upsilon_2}\|\nabla_R\mL_n(S^k, R^k) - \nabla_R\mL_n(S^k, \tR)\|_F^2.
	\end{align*}
	For any $C_{31} > 0$ to be determined later,
	\begin{align*}
	\LD \Theta^k\tLambda\Theta^{k \T}, \frac{1}{2}\hSigmax(R^k - \tR)\hSigmay &+ \frac{1}{2}\hSigmay(R^k - \tR)\hSigmax\RD\\
	\geq & -\|\Theta^k\tLambda \Theta^{k \T}\|_F\cdot\|\nabla_R\mL_n(S^k, R^k) - \nabla_R\mL_n(S^k, \tR)\|_F\\
	\geq &-\frac{C_{31}}{2}\|\Theta^k\|_F^4 - \frac{1}{2C_{31}}\|\nabla_R\mL_n(S^k, R^k) - \nabla_R\mL_n(S^k, \tR)\|_F^2.
	\end{align*}
	Combining the last two inequalities,
	\begin{equation}\label{a:30}
	\I_{31} \geq
	\frac{9\Upsilon_3}{4\Upsilon_2}\|R^k - \tR\|_F^2 - \frac{C_{31}}{2}\Pi^4(U^k, \tU)
	+ \rbr{\frac{4}{\Upsilon_2} - \frac{1}{2C_{31}}} \|\nabla_R\mL_n(S^k, R^k) - \nabla_R\mL_n(S^k, \tR)\|_F^2.
	\end{equation}
	Next, we bound the term $\I_{32}$ as (notation of $\barOmega^k$ is in \eqref{a:2})
	\begin{align*}
	\I_{32}
	&= \LD R^k - \tR + \Theta^k\tLambda\Theta^{k \T},  \frac{1}{2}\hSigmax(S^k - \tS)\hSigmay + \frac{1}{2}\hSigmay(S^k - \tS)\hSigmax\RD\\
	&= \frac{1}{2}\LD\hSigmax\big(R^k - \tR + \Theta^k\tLambda\Theta^{k \T}\big)\hSigmay + \hSigmay\big(R^k - \tR + \Theta^k\tLambda\Theta^{k \T}\big)\hSigmax, S^k - \tS\RD\\
	&= \frac{1}{2}\bLD \P_{\barOmega^k}\cbr{\hSigmax(R^k - \tR + \Theta^k\tLambda\Theta^{k \T})\hSigmay} + \P_{\barOmega^k}\cbr{\hSigmay(R^k - \tR + \Theta^k\tLambda\Theta^{k \T})\hSigmax}, S^k - \tS\bRD\\
	& \geq  -\|S^k - \tS\|_F\big\|\P_{\barOmega^k}\cbr{\hSigmax(R^k - \tR + \Theta^k\tLambda\Theta^{k \T})\hSigmay}\big\|_F\\
	& \geq  -\|S^k - \tS\|_F\big\|\P_{\barOmega^k}\cbr{\hSigmax(R^k - \tR )\hSigmay}\big\|_F -\|S^k - \tS\|_F\|\P_{\barOmega^k}\big(\hSigmax\Theta^k\tLambda\Theta^{k \T}\hSigmay\big)\|_F\\
	& \geq  -\|S^k - \tS\|_F\big\|\P_{\barOmega^k}\cbr{\hSigmax(R^k - \tR )\hSigmay}\big\|_F - \frac{9\sigma_1^X\sigma_1^Y}{4}\|S^k - \tS\|_F\|\Theta^k\|_F^2,
	\end{align*}
	where the first inequality is due to Lemma \ref{aux:lem:6}. Using the same derivation as in \eqref{a:9} to obtain \eqref{a:21}, for any $C_{32,1}>0$,
	\begin{multline*}
	-\|S^k - \tS\|_F\|\P_{\barOmega^k}\big(\hSigmax(R^k - \tR )\hSigmay\big)\|_F \\ \geq -\frac{C_{32,1}}{2}\|S^k - \tS\|_F^2 - \frac{C_1}{2C_{32,1}}(1 + \gamma_1)\alpha r\sigma_1^\tR\Upsilon_1\Pi^2(U^k, \tU).
	\end{multline*}
	with probability at least $1 - 8/d^2$ for some constant $C_1$. Also, for any $C_{32,2} > 0$,
	\begin{align*}
	- \frac{9\sigma_1^X\sigma_1^Y}{4}\|S^k - \tS\|_F\|\Theta^k\|_F^2\geq -\frac{9\sigma_1^X\sigma_1^YC_{32,2}}{8}\|S^k - \tS\|_F^2 - \frac{9\sigma_1^X\sigma_1^Y}{8C_{32,2}}\Pi^4(U^k, \tU).
	\end{align*}
	Therefore, with probability at least $1 - 8/d^2$,
	\begin{equation}\label{a:31}
	\begin{aligned}
	\I_{32}&\geq -\rbr{\frac{C_{32,1}}{2} + \frac{9\sigma_1^X\sigma_1^YC_{32,2}}{8}}\|S^k - \tS\|_F^2- \frac{C_1}{2C_{32,1}}(1 + \gamma_1)\alpha r\sigma_1^\tR\Upsilon_1\Pi^2(U^k, \tU) \\
	&\quad - \frac{9\sigma_1^X\sigma_1^Y}{8C_{32,2}}\Pi^4(U^k, \tU).
	\end{aligned}
	\end{equation}
	For the term $\I_{33}$, for any $C_{33,1}, C_{33,2}>0$,
	\begin{align}\label{a:32}
	&\I_{33}
	= \LD R^k - \tR + \Theta^k\tLambda\Theta^{k \T},   \nabla_R\mL_n(\tS, \tR)\RD\nonumber\\
	\geq & -\|\nabla_R\mL_n(\tS, \tR)\|_2\big(\|R^k - \tR\|_* + \|\Theta^k\tLambda\Theta^{k \T}\|_*\big)\nonumber\\
	\geq &-\rbr{2r}^{1/2}\|\nabla_R\mL_n(\tS, \tR)\|_2\big(\|R^k - \tR\|_F + \|\Theta^k\tLambda\Theta^{k \T}\|_F\big) \nonumber\\
	\geq& -\rbr{2r}^{1/2}\rbr{\frac{C_{33,1}+C_{33,2}}{2C_{33,1}C_{33,2}} \|\nabla_R\mL_n(\tS, \tR)\|_2^2 + \frac{C_{33,1}}{2}\|R^k - \tR\|_F^2 + \frac{C_{33,2}}{2}\Pi^4(U^k, \tU)}.
	\end{align}
	Combining \eqref{a:29}, \eqref{a:30}, \eqref{a:31}, and \eqref{a:32}, we complete the proof.
\end{proof}

\section{Proofs of Other Lemmas}\label{pf_com_lem}

\subsection{Proof of Lemma \ref{prop:1}}

Since $\tU$ has orthogonal columns, $\sigma_1^2$ ($\sigma_r^2$) is the largest (smallest) singular value of $\tU\tLambda\tUT$. Then, for any $Q\in\mQ^{r\times r}_{r_1}$,
\begin{align*}
\|U\tLambda U^\T &- \tU\tLambda\tUT\|_F^2 \\
&=\|U\tLambda U^\T - \tU Q\tLambda Q^\T\tUT\|_F^2\\
&= \|(U - \tU Q)\tLambda (U - \tU Q)^\T + (U - \tU Q)\tLambda Q^\T \tUT + \tU Q\tLambda (U - \tU Q)^\T\|_F^2\\
&\leq 3\|U - \tU Q\|_2^2\|U - \tU Q\|_F^2 + 6\sigma_1^2\|U - \tU Q\|_F^2\\
&\leq 3\|U - \tU Q\|_F^4 + 6\sigma_1^2\|U - \tU Q\|_F^2.
\end{align*}
We minimize the right hand side over $Q$ to obtain
\begin{align*}
\inf_{Q\in \mQ_{r_1}^{r\times r}}\big\{3\|U - \tU Q\|_F^4 + 6\sigma_1^2\|U - &\tU Q\|_F^2\big\}\\
&= 3\inf_{Q\in \mQ_{r_1}^{r\times r}}\|U - \tU Q\|_F^4 + 6\sigma_1^2\inf_{Q\in \mQ_{r_1}^{r\times r}}\|U - \tU Q\|_F^2\\
&= 3\Pi^4(U, \tU) + 6\sigma_1^2\Pi^2(U, \tU)\\
& \leq 9\sigma_1^2\Pi^2(U, \tU),
\end{align*}
which completes the proof of the first part. For the second part of the result, based on Lemma \ref{aux:lem:9},
\begin{equation}
\begin{aligned}\label{b:2}
\inf_{Q\in\Q^{r\times r}}\|U - \tU Q\|_F^2 + \|U\tLambda - \tU\tLambda Q\|_F^2
&= \inf_{Q\in\Q^{r\times r}} -2\TR(U^T\tU Q) - 2\TR(\tLambda U^T\tU\tLambda Q) \\
&= - 2\sup_{Q\in\Q^{r\times r}} \TR\big((U^\T\tU + \tLambda U^\T\tU\tLambda) Q\big)\\
& \leq \frac{2}{(\surd{2}-1)\sigma_r^2}\|U\tLambda U^\T - \tU\tLambda\tUT\|_F^2.
\end{aligned}
\end{equation}
Let $U = (U_1, U_2)$ with $U_1\in\mR^{d\times r_1}$ and $U_2 = \mR^{d\times (r-r_1)}$, and analogously $\tU = (\tU_1, \tU_2)$. Furthermore, suppose the following singular value decompositions $U_1^\T \tU_1 = A_1\Sigma_1B_1^\T$, $U_2^\T \tU_2 = A_2\Sigma_2B_2^\T$. Then,
\begin{align*}
U^\T\tU + \tLambda U^\T\tU\tLambda = 2\begin{pmatrix}
U_1^\T\tU_1 &\\
& U_2^\T\tU_2
\end{pmatrix} = 2 \begin{pmatrix}
A_1 & \\
& A_2
\end{pmatrix}\begin{pmatrix}
\Sigma_1 & \\
& \Sigma_2
\end{pmatrix}\begin{pmatrix}
B_1^\T & \\
& B_2^\T
\end{pmatrix}.
\end{align*}
The optimal $Q$ that achieves the supremum in \eqref{b:2} is
\begin{align*}
Q = \begin{pmatrix}
B_1A_1^\T & \\
& B_2A_2^\T
\end{pmatrix}.
\end{align*}
Since $Q\in\mQ^{r\times r}_{r_1}$,
\begin{equation}
\begin{aligned}\label{b:3}
\inf_{Q\in\Q^{r\times r}}\|U - \tU Q\|_F^2 + \|U\tLambda - \tU\tLambda Q\|_F^2
&=\inf_{Q\in \mQ_{r_1}^{r\times r}}\|U - \tU Q\|_F^2 + \|U\tLambda - \tU\tLambda Q\|_F^2 \\
&= \inf_{Q\in \mQ_{r_1}^{r\times r}}\|U - \tU Q\|_F^2 + \|U - \tU\underbrace{\tLambda Q\tLambda}_{Q}\|_F^2 \\
&= 2\inf_{Q\in \mQ_{r_1}^{r\times r}}\|U - \tU Q\|_F^2\\
&= 2\Pi^2(U, \tU).
\end{aligned}
\end{equation}
Combining \eqref{b:3} with \eqref{b:2} completes the proof.

\subsection{Proof of Lemma \ref{lem:1}}

By Sylvester's law of inertia \citep[cf. Theorem 4.5.8][]{Horn2013Matrix}, $\tR$ has $r_1$ positive eigenvalues and $r-r_1$ negative eigenvalues, denoted as $\tlambda_1\geq \cdots\geq \tlambda_{r_1}>0>\tlambda_{r_1+1}\geq \cdots\tlambda_r$. Therefore $\sigma_r^\tR = \min\{\tlambda_{r_1}, |\tlambda_{r_1+1}|\}$. Similarly, we denote the eigenvalue of $R$ as $\lambda_1\geq \cdots\geq \lambda_{\hat r_1} \geq  \epsilon_1 \geq \cdots \geq \epsilon_{d-r} \geq \lambda_{\hat r_1+1} \geq \cdots \lambda_r$, where $\{\lambda_i\}_{i = 1}^{r}$ denote the $r$ eigenvalues with the largest magnitude, $\hr_1$ of which are positive. In particular, $|\epsilon_i| \leq |\lambda_j|$,
$i = 1,\ldots, d-r$, $j = 1, \ldots, r$.

Suppose $\hr_1<r_1$. Then $\epsilon_1$ will correspond to
$\tlambda_\barj$ for some $\barj\in \{ \hr_1+1, \ldots, r_1 \}$.
Based on Weyl's inequality \citep[cf. Theorem 4.3.1][]{Horn2013Matrix},
\begin{align}\label{b:6}
\epsilon_1
\geq \tlambda_\barj - |\epsilon_1 - \tlambda_\barj|
\geq \tlambda_\barj - \|R - \tR\|_2
\geq \tlambda_{r_1} -  \|R - \tR\|_2
\geq 2\sigma_r^\tR/3.
\end{align}
On the other hand, $\lambda_\ttj$ for some $\ttj\in \{ \hr_1+1, \ldots, r\}$
corresponds to the zero eigenvalue  of $R^\star$.
Then, by Weyl's inequality,
\begin{align}\label{b:7}
|\lambda_\ttj|\leq \|R - \tR\|_2\leq \sigma_r^\tR/3.
\end{align}
Combining \eqref{b:6} and \eqref{b:7}, $\epsilon_1 > |\lambda_\ttj|$,
which is a contradiction. Using the same technique,
assuming that $\hr_1>r_1$ allows us to
show that $\epsilon_{d-r}$ will be one of the $r$ largest eigenvalues of $R$, leading to a contradiction. Therefore $\hr_1 = r_1$.

\section{Auxiliary Results}\label{appen:aux:lem}

\begin{lemma}[Lemma B.3 in \cite{Zhang2018Unified}]\label{aux:lem:1}

	Let $\tS\in \S^{d\times d}$ satisfy $\|\tS\|_{0, \infty}\leq \alpha d$. Then for any $S\in\S^{d\times d}$ and $\gamma>1$, we have $\mT_{\gamma\alpha}(S)\in\S^{d\times d}$ and
	\begin{align*}
	\|\mT_{\gamma\alpha}(S) - \tS\|_F^2\leq \cbr{1 + \rbr{\frac{2}{\gamma - 1}}^{1/2}}^2\|S - \tS\|_F^2.
	\end{align*}

\end{lemma}

\begin{lemma}[Lemma 3.3 in \cite{Li2016Nonconvexa}]\label{aux:lem:2}

	Let $\tS\in \S^{d\times d}$ satisfy $\|\tS\|_{0, 1}\leq s$. Then for any  $S\in\S^{d\times d}$ and $\gamma>1$, we have $\mJ_{\gamma s}(S)\in\S^{d\times d}$ and
	\begin{align*}
	\|\mJ_{\gamma s}(S) - \tS\|_F^2\leq \cbr{1 + \frac{2}{\rbr{\gamma - 1}^{1/2}}}\|S - \tS\|_F^2.
	\end{align*}

\end{lemma}

\begin{lemma}[Proposition 1.1 in \cite{Hsu2012tail}]\label{aux:lem:4}

	Let $X\sim \mN(0, \Sigma)$. Then, for any $\delta>0$,
	\begin{align*}
	\Pr\rbr{\|X\|_2^2 \leq \|\Sigma\|_2\sbr{d + 2\cbr{d\log\rbr{1/\delta}}^{1/2} + 2\log\rbr{1/\delta}}} \geq 1-\delta.
	\end{align*}

\end{lemma}

\begin{lemma}[Lemma 5.14 in \cite{Tu2016Low}]\label{aux:lem:9}

	Let $M_1, M_2\in\mR^{d_1\times d_2}$ be rank-$r$ matrices with the reduced singular value decomposition $M_i = U_i\Sigma_iV_i^T$, $i = 1, 2$. Let $X_i = U_i\Sigma_i^{1/2}$ and $Y_i = V_i\Sigma_i^{1/2}$. If $\|M_2 - M_1\|_2\leq \sigma_r(M_1)/2$, then
	\begin{align*}
	\inf_{Q\in\Q^{r\times r}}\|X_2 - X_1 Q\|_F^2 + \|Y_2 - Y_1 Q\|_F^2\leq \frac{2}{\surd{2}-1}\frac{\|M_2 - M_1\|_F^2}{\sigma_r(M_1)},
	\end{align*}
	where $\sigma_r(M_1)$ denotes the $r$-th singular value of $M_1$.

\end{lemma}

\begin{lemma}[Lemma C.1 in \cite{Wang2017Unified}]\label{aux:lem:13}

	Let $U\in\mR^{d_1\times r}$, $V\in\mR^{d_2\times r}$,
	\begin{align*}
	Z = \begin{pmatrix}U \\ V \end{pmatrix}, \quad
	\barZ = \begin{pmatrix}U \\ -V \end{pmatrix}.
	\end{align*}
	For any $\tZ\in\mR^{(d_1 + d_2)\times r}$, we define $Q = \arg\inf_{Q\in\Q^{r\times r}}\|Z - \tZ Q\|_F$. Then
	\begin{align*}
	\LD \barZ\barZ^\T Z, Z - \tZ Q \RD \geq \frac{1}{4}\|\barZ^\T Z\|_F^2 - \frac{1}{4}\|Z - \tZ Q\|_F^4.
	\end{align*}

\end{lemma}

\begin{lemma}[Lemma 5.4 in \cite{Tu2016Low}]\label{aux:lem:15}

	Let $U, V\in\mR^{d\times r}$ and let $\sigma_r$ be the $r$-th singular value of $V$. Then
	\begin{align*}
	\inf_{Q\in\Q^{r\times r}}\|U - VQ\|_F^2\leq \frac{1}{2(\surd{2}-1)\sigma_r^2}\|UU^\T - VV^\T\|_F^2.
	\end{align*}

\end{lemma}

\begin{lemma}\label{aux:lem:6}

	Suppose $\Omega\subseteq[d]\times[d]$ is a symmetric index set. For any matrix $A\in\mR^{d\times d}$, we have $\big(\P_{\Omega}(A)\big)^\T = \P_{\Omega}(A^\T)$.

\end{lemma}

\begin{proof}
	For any $(i, j)\in [d]\times[d]$, we have
	\begin{align*}
	[\P_{\Omega}(A^\T)]_{i, j} = \begin{cases}
	[A^\T]_{i, j} & \text{if\ } (i, j)\in\Omega\\
	0 & \text{otherwise}
	\end{cases} = \begin{cases}
	A_{j, i} & \text{if\ } (i, j)\in\Omega\\
	0 & \text{otherwise}
	\end{cases} = \begin{cases}
	A_{j, i} & \text{if\ } (j, i)\in\Omega\\
	0 & \text{otherwise}
	\end{cases} = [\P_{\Omega}(A)]_{j, i}.
	\end{align*}
	This completes the proof.
\end{proof}

\begin{lemma}\label{aux:lem:8}

	Suppose $S_1, S_2\in\S^{d\times d}$ are two symmetric matrices and $\Sigma_1, \Sigma_2\in\mR^{d\times d}$ satisfy $0\prec \sigma_{i, d}I_d\preceq\Sigma_i\preceq \sigma_{i, 1}I_d$ for $i = 1,2$. Then
	\begin{multline*}
	\LD \Sigma_1(S_1 - S_2)\Sigma_2, S_1 - S_2\RD \\
	\geq \frac{\sigma_{1, 1}\sigma_{2, 1}\sigma_{1, d}\sigma_{2, d}}{\sigma_{1, 1}\sigma_{2, 1} + \sigma_{1, d}\sigma_{2, d}}\|S_1 - S_2\|_F^2 + \frac{1}{4\big(\sigma_{1, 1}\sigma_{2, 1} + \sigma_{1, d}\sigma_{2, d}\big)}\|\Sigma_1(S_1 - S_2)\Sigma_2 + \Sigma_2(S_1 - S_2)\Sigma_1\|_F^2.
	\end{multline*}
\end{lemma}

\begin{proof}
	For a symmetric matrix $S \in \S^{d\times d}$, let $\mF(S) = \frac{1}{2}\TR(S\Sigma_1S\Sigma_2)$. Then, for any $S_1, S_2 \in \S^{d\times d}$,
	\begin{align*}
	\mF(S_1) - \mF(S_2) - \LD \nabla\mF(S_2), S_1 - S_2\RD
	= & \frac{1}{2}\TR\sbr{\rbr{S_1-S_2}\Sigma_1\rbr{S_1 - S_2}\Sigma_2}.
	\end{align*}
	Therefore,
	\begin{align*}
	\frac{\sigma_{1, d}\sigma_{2, d}}{2}\|S_1 - S_2\|_F^2 \leq \mF(S_1) - \mF(S_2) - \LD \nabla\mF(S_2), S_1 - S_2\RD \leq \frac{\sigma_{1, 1}\sigma_{2, 1}}{2}\|S_1 - S_2\|_F^2,
	\end{align*}
	implying that $\mF(\cdot)$ is a $\sigma_{1, 1}\sigma_{2, 1}$-smooth and $\sigma_{1, d}\sigma_{2, d}$-strongly convex function. Furthermore,
	\begin{align*}
	\LD&\Sigma_1(S_1 - S_2)\Sigma_2, S_1 - S_2\RD \\
	& = \LD \nabla\mF(S_1) - \nabla\mF(S_2), S_1 - S_2\RD\\
	& \geq \frac{\sigma_{1, 1}\sigma_{2, 1}\sigma_{1, d}\sigma_{2, d}}{\sigma_{1, 1}\sigma_{2, 1} + \sigma_{1, d}\sigma_{2, d}}\|S_1 - S_2\|_F^2 + \frac{1}{\sigma_{1, 1}\sigma_{2, 1} + \sigma_{1, d}\sigma_{2, d}}\|\nabla\mF(S_1) - \nabla\mF(S_2)\|_F^2\\
	& = \frac{\sigma_{1, 1}\sigma_{2, 1}\sigma_{1, d}\sigma_{2, d}}{\sigma_{1, 1}\sigma_{2, 1} + \sigma_{1, d}\sigma_{2, d}}\|S_1 - S_2\|_F^2 + \frac{1}{4\big(\sigma_{1, 1}\sigma_{2, 1} + \sigma_{1, d}\sigma_{2, d}\big)}\|\Sigma_1(S_1 - S_2)\Sigma_2 + \Sigma_2(S_1 - S_2)\Sigma_1\|_F^2,
	\end{align*}
	where the inequality is due to Lemma 3.5 in \cite{Bubeck2015Convex}.
\end{proof}

\begin{lemma}\label{aux:lem:10}

	Let $A\in\mR^{d\times d}$, $U\in\mR^{d\times r}$, and $Q\in\Q^{r\times r}$. Then $\|(AUQ)^\T\|_{2, \infty}\leq \|A^\T\|_1\|U^\T\|_{2, \infty}$.

\end{lemma}

\begin{proof}
	Note that $\|(AUQ)^\T\|_{2, \infty} = \|Q^\T U^\T A^\T\|_{2, \infty} = \|U^\T A^\T\|_{2, \infty}$. Suppose $U^\T = \big(U_1, \ldots, U_d\big)$, where $U_i\in\mR^r$ is the $i$-th row of $U$. Then
	\begin{align*}
	\|U^\T A^\T\|_{2, \infty}= &\max_{j\in[d]}\big\|\sum_{i=1}^d A_{j, i}U_i\big\|_2\leq \max_{j\in[d]} \sum_{i=1}^d |A_{j, i}| \|U_i\|_2
	\leq \max_{j\in[d]} \big(\max_{i\in[d]}\|U_i\|_2\big)\sum_{i=1}^d |A_{j, i}|\\
	=& \|U^\T\|_{2, \infty} \max_{j\in[d]} \|A_{j, \cdot}\|_1 = \|U^\T\|_{2, \infty} \|A^\T\|_{1, \infty}= \|U^\T\|_{2, \infty}\|A^\T\|_1.
	\end{align*}
\end{proof}

\begin{lemma}\label{aux:lem:14}

	Let $\tU\in\mR^{d\times r}$ have orthogonal columns with the $r$-th singular value being $\sigma_r$. For any $U\in\mR^{d\times r}$ and $r_1\in\{0,\ldots, r\}$, let $Q = \arg\inf_{Q\in\mQ^{r \times r}_{r_1}}\|U - \tU Q\|$ and $\Lambda = \diag(I_{r_1}, -I_{r-r_1})$. Then the following two inequalities hold
	\begin{align*}
	\LD U - \tU Q, U\big(U^\T U - \Lambda U^\T U\Lambda\big)\RD
	& \geq \frac{1}{8}\|U^\T U - \Lambda U^\T U\Lambda\|_F^2 - \frac{1}{2}\Pi^4(U, \tU), \\
	\|U^\T U  - \Lambda U^\T U\Lambda\|_F^2
	& \geq  8(\surd{2} - 1)\sigma_r^2\Pi^2(U, \tU) - 4\|U\Lambda U - \tU\Lambda\tUT\|_F^2.
	\end{align*}
	Furthermore, we have
	\begin{align*}
	\LD U - \tU Q, U(U^\T U - \Lambda U^\T U\Lambda)\RD
	&\geq (\surd{2} - 1)\sigma_r^2\Pi^2(U, \tU) - \frac{1}{2}\|U\Lambda U - \tU\Lambda\tUT\|_F^2 - \frac{1}{2}\Pi^4(U, \tU).
	\end{align*}

\end{lemma}

\begin{proof}
	We only prove two inequalities. The last argument comes from two inequalities immediately. Let $V = U\Lambda$, $\tV = \tU\Lambda$, and
	\begin{align*}
	Z = \begin{pmatrix}U \\ V \end{pmatrix}, \quad
	\barZ = \begin{pmatrix}U \\ -V \end{pmatrix}, \quad
	\tZ = \begin{pmatrix}\tU \\ \tV \end{pmatrix}, \quad
	\barZ^\star = \begin{pmatrix}\tU \\ -\tV \end{pmatrix}.
	\end{align*}
	The $r$-th singular value of $\tZ$ is $\surd{2}\sigma_r$. From \eqref{b:2} and \eqref{b:3}, we know $Q = \arg\inf_{Q\in\Q^{r\times r}}\|Z - \tZ Q\|_F$. By Lemma \ref{aux:lem:13},
	\begin{align}\label{c:11}
	\LD \barZ\barZ^\T Z, Z - \tZ Q\RD \geq \frac{1}{4}\|\barZ^\T Z\|_F^2 - \frac{1}{4}\|Z - \tZ Q\|_F^4.
	\end{align}
	We can rewrite the left hand side of \eqref{c:11} as
	\begin{equation}\label{c:12}
	\begin{aligned}
	\LD\barZ\barZ^\T Z, Z - \tZ Q\RD
	&= \TR\big((U^\T U - V^\T V)\barZ^\T(Z - \tZ Q)\big) \\
	&= \TR\rbr{\big(U^\T U - V^\T V\big)\cbr{U^\T(U - \tU Q) - V^\T(V - \tV Q)}} \\
	&= \TR\big((U^\T U - V^\T V)U^\T(U - \tU Q)\big) - \TR\big((U^\T U - V^\T V)V^\T(V - \tV Q)\big).
	\end{aligned}
	\end{equation}
	Since $\Lambda Q\Lambda = Q$ and by the definition of $V$, the second term in above equation can be written as
	\begin{equation}
	\begin{aligned}
	- \TR\big((U^\T U - V^\T V)V^\T(V - \tV Q)\big)
	&= - \TR\big((U^\T U - V^\T V)\Lambda U^\T(U - \tU \Lambda Q\Lambda)\Lambda\big)\\
	&= - \TR\big((U^\T U - V^\T V)\Lambda U^\T(U - \tU Q)\Lambda\big)\\
	&= - \TR\big(\Lambda(U^\T U - V^\T V)\Lambda U^\T(U - \tU Q)\big)\\
	&= \TR\big((U^\T U - V^\T V)U^\T(U - \tU Q)\big),
	\end{aligned}
	\end{equation}
	which is the same as the first term. Using the definition of $V$ and plugging  the above display into \eqref{c:12}, we have
	\begin{align}\label{c:13}
	\LD U - \tU Q, U(U^\T U - \Lambda U^\T U\Lambda)\RD = \frac{1}{2}\LD\barZ\barZ^\T Z, Z - \tZ Q\RD.
	\end{align}
	For the right hand side of \eqref{c:11}, by the definition of $Z$, $\barZ$, and \eqref{b:3}, we have following relations
	\begin{align}\label{c:17}
	\|\barZ^\T Z\|_F^2 = \|U^\T U - \Lambda UU^\T \Lambda\|_F^2, \quad \quad \|Z - \tZ Q\|_F^2 = 2\Pi^2(U, \tU).
	\end{align}
	Combine \eqref{c:17} with \eqref{c:11}, \eqref{c:13} and we prove the first inequality in the argument. Moreover, since $\tU$ has orthogonal columns, $\barZ^{\star \T}\tZ = \tUT\tU - \tVT\tV = \tUT\tU - \Lambda\tUT\tU\Lambda = 0$. Therefore,
	\begin{align}\label{c:15}
	\|\barZ^\T Z\|_F^2
	&= \LD ZZ^\T, \barZ\barZ^\T\RD =  \LD ZZ^\T, \barZ\barZ^\T\RD + \LD \tZ\tZT, \barZ^{\star}\barZ^{\star \T}\RD \nonumber\\
	&= \LD ZZ^\T - \tZ\tZT, \barZ\barZ^\T - \barZ^\star\barZ^{\star \T}\RD + \LD ZZ^\T, \barZ^\star\barZ^{\star \T}\RD  + \LD \barZ\barZ^\star, \tZ\tZT\RD \nonumber\\
	&\geq \LD  ZZ^\T - \tZ\tZT, \barZ\barZ^\T - \barZ^\star\barZ^{\star \T}\RD \nonumber\\
	&= \|UU^\T - \tU\tUT\|_F^2 - 2\TR\big((UV^\T - \tU\tVT)(VU^\T - \tV\tUT)\big) + \|VV^\T - \tV\tVT\|_F^2 \nonumber\\
	&= 2\|UU^\T - \tU\tUT\|_F^2 - 2\|U\Lambda U^\T - \tU\Lambda\tUT\|_F^2 \nonumber\\
	&= \|ZZ^\T - \tZ\tZT\|_F^2 - 4\|U\Lambda U^\T - \tU\Lambda\tUT\|_F^2.
	\end{align}
	By Lemma \ref{aux:lem:15} and \eqref{b:3},
	\begin{align}\label{c:16}
	\|ZZ^\T - \tZ\tZT\|_F^2\geq 2(\surd{2} - 1)(\surd{2}\sigma_r)^2\inf_{Q\in\Q^{r\times r}}\|Z - \tZ Q\|_F^2
	= 8(\surd{2} - 1)\sigma_r^2\Pi^2(U, \tU).
	\end{align}
	Combining with relations in \eqref{c:17} and \eqref{c:15}, we prove the second inequality.
\end{proof}

\begin{lemma}[Concentration of sample covariance]\label{aux:lem:7}

	Let $X_1, \ldots, X_n$ be independent realizations of $X$, which is $d$-dimensional random vector distributed as $\mN(\mu, \Sigma)$. Let
	\begin{align}\label{cov:1}
	\hSigma = \frac{1}{n}\sum_{i=1}^n(X_i - \hmu)(X_i - \hmu)^\T,\quad \hmu = \frac{1}{n}\sum_{i=1}^nX_i,
	\end{align}
	be the sample covariance and mean, respectively. Then there exists a constant $C_1 > 0$ such that
	\begin{align*}
	\Pr\rbr{\|\hSigma - \Sigma\|_2
		\leq C_1\|\Sigma\|_2
		\sbr{\cbr{\frac{d + \log(1/\delta)}{n}}^{1/2} \bigvee \frac{d + \log(1/\delta)}{n}}}\geq 1-\delta.
	\end{align*}
	Moreover, suppose $\sigma_d\leq \sigma_{\min}(\Sigma)\leq \sigma_{\max}(\Sigma)\leq \sigma_1$, then if $n\gtrsim \kappa^2d$ with $\kappa = \sigma_1/\sigma_d$ being the condition number,
	\begin{align*}
	\Pr\rbr{\frac{\sigma_d}{2}\leq\sigma_{\min}(\hSigma)\leq\sigma_{\max}(\hSigma)\leq \frac{3\sigma_1}{2}}\geq 1 - \frac{1}{d^2}.
	\end{align*}

\end{lemma}

\begin{proof}
	Since $\hSigma =  n^{-1}\sum_{i=1}^n(X_i - \mu)(X_i - \mu)^\T - (\hmu - \mu)(\hmu - \mu)^\T$, we have
	\begin{align}\label{c:1}
	\|\hSigma - \Sigma\|_2 \leq \bigg\|\frac{1}{n}\sum_{i=1}^n(X_i - \mu)(X_i - \mu)^\T - \Sigma\bigg\|_2 + \|\hmu - \mu\|_2^2.
	\end{align}
	Using equation (6.12) in \citet[][]{Wainwright2019High}, there exists a constant $C > 0$, so that
	\begin{align}\label{c:2}
	\left\| \frac{1}{n}\sum_{i=1}^n(X_i - \mu)(X_i - \mu)^\T - \Sigma \right\|_2
	\leq C\|\Sigma\|_2
	\sbr{\cbr{\frac{d + \log(1/\delta)}{n}}^{1/2} \bigvee \frac{d + \log(1/\delta)}{n}}
	\end{align}
	with probability at least $1- \delta$. For the second term, we have $\hmu - \mu\sim \mN(0, \Sigma/n)$. By Lemma~\ref{aux:lem:4},
	\begin{align}\label{c:3}
	\|\hmu - \mu\|_2^2 \leq \frac{\|\Sigma\|_2}{n}\sbr{d + 2\cbr{d\log(1/\delta)}^{1/2} + 2\log(1/\delta)}
	\end{align}
	with probability at least $1-\delta$. The proof of the first part follows by combining the last two displays. Moreover, let $\delta = 1/d^2$. We see if $n\gtrsim \kappa^2d$ then $\|\hSigma - \Sigma\|_2\leq \sigma_d/2$ with probability at least $1-1/d^2$. By Weyl's inequality \citep[cf. Theorem 4.3.1][]{Horn2013Matrix}, we can further get
	\begin{align*}
	\sigma_{\min}(\hSigma)\geq \sigma_{\min}(\Sigma) - \|\hSigma - \Sigma\|_2\geq \frac{\sigma_d}{2}.
	\end{align*}
	Similarly, the upper bound satisfies
	\begin{align*}
	\sigma_{\max}(\hSigma)\leq \sigma_{\max}(\Sigma) + \|\hSigma - \Sigma\|_2\leq \frac{3\sigma_1}{2}.
	\end{align*}
	This completes the second part of proof.
\end{proof}

\begin{lemma}\label{aux:lem:12}

	Let $X_1, \ldots, X_n$ be independent copies of $X\sim\mN(\mu, \Sigma) \in \mR^d$ and let $\ttSigma = n/(n-d-2)\hSigma$ with $\hSigma$ defined in \eqref{cov:1} be the scaled sample covariance. Suppose $d \leq cn$ for $c\in(0, 1/2)$,
	\begin{align*}
	\Pr\cbr{\|\ttSigma^{-1} - \Sigma^{-1}\|_{\infty, \infty}>\|\Sigma^{-1/2}\|_1^2\rbr{\frac{8\log d}{n}}^{1/2}}\leq \frac{4}{d^2}.
	\end{align*}
\end{lemma}

\begin{proof}
	We only prove the result for $\mu = 0$. The same technique can be applied for a general $\mu$.  Let $Y_i = \Sigma^{-1/2}X_i\sim \mN(0, I_d)$ for $i=1,\ldots,n$. Then
	\begin{align}
	\label{c:7}
	\|\ttSigma^{-1} - \Sigma^{-1}\|_{\infty, \infty}
	&= \|\Sigma^{-1/2}\big(\Sigma^{1/2}\ttSigma^{-1}\Sigma^{1/2} - I_d\big)\Sigma^{-1/2}\|_{\infty, \infty} \nonumber\\
	&\leq \|\Sigma^{-1/2}\|_1^2 \|(\frac{1}{n-d-2}\sum_{i=1}^n Y_iY_i^\T)^{-1} - I_d\|_{\infty, \infty}.
	\end{align}
	Let $S_{n, d} = \frac{1}{n-d-2}\sum_{i=1}^nY_iY_i^\T$, $\Delta_{n, d} = S_{n, d} - I_d$, and $\omega_{j, k}^{(n, d)} =  \rbr{S_{n, d}^{-1}}_{j, k}$ for $j, k = 1,\ldots,d$. We first consider the case when $d = 2$. For a large enough sample size $n$, we have $\|\Delta_{n, 2}\|_1 \leq {1}/{2}$ and
	\begin{align*}
	(I_2 + \Delta_{n, 2})^{-1} = \sum_{k=0}^{\infty}(-1)^k\Delta_{n, 2}^k = I_2 - \Delta_{n, 2} + \Delta_{n, 2}(I_2 + \Delta_{n, 2})^{-1} \Delta_{n, 2}.
	\end{align*}
	Then
	\begin{align}\label{c:9}
	\|S_{n,2}^{-1} - I_2\|_{\infty, \infty} = &\|(I_2 + \Delta_{n, 2})^{-1} - I_2\|_{\infty, \infty}  \nonumber\\
	\leq &\|\Delta_{n, 2}\|_{\infty, \infty} + \|\Delta_{n, 2}(I_2 + \Delta_{n, 2})^{-1} \Delta_{n, 2}\|_{\infty, \infty}.
	\end{align}
	Since $\|\Delta_{n, 2}\|_1 \leq {1}/{2}$, we have
	\begin{multline*}
	\|\Delta_{n, 2}(I_2 + \Delta_{n, 2})^{-1} \Delta_{n, 2}\|_{\infty, \infty}
	\leq  \|\Delta_{n, 2}\|_{\infty, \infty} \|\Delta_{n, 2}\|_1\sum_{k=0}^\infty\|\Delta_{n, 2}\|_1^k\\
	\leq  \frac{\|\Delta_{n, 2}\|_{\infty, \infty} \|\Delta_{n, 2}\|_1}{1 - \|\Delta_{n, 2}\|_1}
	\leq  \|\Delta_{n, 2}\|_{\infty, \infty}.
	\end{multline*}
	Combining with \eqref{c:9}, we have $\|S_{n,2}^{-1} - I_2\|_{\infty, \infty}\leq 2\|\Delta_{n, 2}\|_{\infty, \infty}$. By Lemma 1 in \cite{Rothman2008Sparse}, there exists a constant $\kappa>0$, such that
	\begin{align}\label{c:10}
	\Pr\bigg(\|S_{n,2}^{-1} - I_2\|_{\infty, \infty}>t\bigg)\leq \Pr\bigg(\|\Delta_{n, 2}\|_{\infty, \infty}>\frac{t}{2}\bigg)\leq 4\exp(-nt^2), \qquad  |t|\leq \kappa.
	\end{align}
	Next, we consider a general $d$. We divide $S_{n, d}$ into $2 \times 2$ block matrix as
	\begin{align*}
	S_{n, d} = \begin{pmatrix}
	S_{n, d}^{1, 1} & S_{n, d}^{1, 2}\\
	S_{n, d}^{2, 1} & S_{n, d}^{2, 2}
	\end{pmatrix}
	\end{align*}
	where $S_{n, d}^{1, 1}\in\mR^{2\times 2}$ and $S_{n, d}^{2, 2}\in\mR^{(d-2)\times (d-2)}$. Let $S_{n, d}^{11\cdot2} = S_{n, d}^{1, 1} - S_{n, d}^{1, 2}(S_{n, d}^{2, 2})^{-1}S_{n, d}^{2, 1}$. Due to the structure of $S_{n, d}^{-1}$, it suffices to show concentration of two representative entries: $(1,1)$-entry $\omega_{1,1}^{n, d}$ and $(1,2)$-entry $\omega_{1,2}^{n, d}$. By the block matrix inversion formula \citep[cf. Theorem 2.1 in][]{Lu2002Inverses}, $\omega_{1, 1}^{n, d} = [(S_{n, d}^{11\cdot 2})^{-1}]_{1, 1}$ and $\omega_{1, 2}^{n, d} = [(S_{n, d}^{11\cdot 2})^{-1}]_{1, 2}$. By Proposition 8.7 in \cite{Eaton2007Multivariate}, we have that $S_{n, d}^{11\cdot 2}$ is equal in distribution to $(n-d-2)^{-1}\sum_{i=1}^{n-d+2}Z_iZ_i^\T$ where $Z_i$, $i=1,\ldots,n-d+2$, are independently drawn from $\mN(0, I_2)$. In particular, we have that $S_{n, d}^{11\cdot 2}$ is equal in distribution to $(n-d+2)/(n-d-2)\cdot S_{n-d+2, 2}$. Therefore,
	\begin{align*}
	|\omega_{1, 1}^{n, d} - 1| \vee |\omega_{1, 2}^{n, d} |\leq \|\frac{n-d-2}{n-d+2}S_{n-d+2, 2}^{-1} - I_2\|_{\infty, \infty}\leq \|S_{n-d+2, 2}^{-1} - I_2\|_{\infty, \infty} + \frac{4}{n-d}.
	\end{align*}
	Further, we have
	\begin{align*}
	\Pr\rbr{ |\omega_{1, 1}^{n, d} - 1| > t } \vee
	\Pr\rbr{ |\omega_{1, 2}^{n, d} | > t }\leq &\Pr\rbr{\|S_{n-d+2, 2}^{-1} - I_2\|_{\infty, \infty}>t - \frac{4}{n-d}}\\
	\leq & 4\exp\cbr{- (n-d+2)\rbr{t - \frac{4}{n-d}}^2},
	\end{align*}
	for $|t - 4/(n-d)|\leq \kappa$, using \eqref{c:10}. Setting $t = \rbr{{8\log d}/{n}}^{1/2}$, taking union bound over all entries, ignoring smaller order term $4/(n-d)$, and combining with \eqref{c:7}, we finally complete the proof.
\end{proof}

\begin{lemma}\label{aux:lem:11}

	Let $X_1, \ldots, X_n$ be independent copies of $X\sim\mN(\mu, \Sigma) \in \mR^d$ and let $\hSigma$ be the sample covariance defined in \eqref{cov:1}. For any $U\in\mR^{d\times r}$ and $Q\in\Q^{r\times r}$,
	\begin{align*}
	\Pr\cbr{\|\big((\hSigma - \Sigma)UQ\big)^\T\|_{2, \infty}>11\rbr{\|U^\T \Sigma U\|_2\|\Sigma\|_2}^{1/2}\rbr{\frac{r\log d}{n}}^{1/2}}\leq \frac{2}{d^2}.
	\end{align*}

\end{lemma}

\begin{proof}
	Let $Y_i = X_i - \mu$, $i=1,\ldots,n$. Then
	\begin{align}\label{c:4}
	\|\big((\hSigma - \Sigma)UQ\big)^\T\|_{2, \infty} =& \|\big((\hSigma - \Sigma)U\big)^\T\|_{2, \infty} \nonumber\\
	\leq  &\|U^\T\big(\frac{1}{n}\sum_{i=1}^nY_iY_i^\T - \Sigma\big)\|_{2, \infty} + \|U^\T(\hmu - \mu)(\hmu - \mu)^\T\|_{2, \infty}.
	\end{align}
	For the first term in \eqref{c:4}, we have
	\begin{align*}
	\|U^\T\big(\frac{1}{n}\sum_{i=1}^n Y_iY_i^\T - \Sigma\big)\|_{2, \infty}
	= \max_{j\in[d]}\sup_{\substack{v\in\mR^r, \\ \|v\|_2\leq 1}} \frac{1}{n}\sum_{i=1}^n v^\T U^\T Y_i Y_i^\T e_j - \mE\rbr{ v^\T U^\T YY^\T e_j},
	\end{align*}
	where $e_j\in\mR^d$ denotes the $j$-th canonical basis of $\mR^d$. Let $\N$ be a $1/2$-net of $\{v\in\mR^r: \|v\|_2\leq 1\}$. Then $|\N| \leq 6^r$ and
	\begin{multline*}
	\max_{j\in[d]}\sup_{\substack{v\in\mR^r, \\ \|v\|_2\leq 1}} \frac{1}{n}\sum_{i=1}^n v^\T U^\T Y_i Y_i^\T e_j - \mE\rbr{ v^\T U^\T YY^\T e_j } \\
	\leq 2\max_{j\in[d]}\sup_{v\in\N} \frac{1}{n}\sum_{i=1}^n v^\T U^\T Y_i Y_i^\T e_j - \mE\rbr{ v^\T U^\T YY^\T e_j }
	\end{multline*}
	using (4.10) and Lemma 4.4.1 of \citet{Vershynin2018High}. Note that $v^\T U^\T Y_i Y_i^\T e_j $ is a sub-Exponential random	variable with the Orlitz $\psi_1$-norm \citep[c.f. Definition 2.7.5][]{Vershynin2018High} bounded as $\|v^\T U^\T Y_i Y_i^\T e_j\|_{\psi_1}\leq \|U^\T \Sigma U\|_2^{1/2}\|\Sigma\|_2^{1/2}$.  Using Bernstein's inequality \citep[cf. Theorem 2.8.1 in][]{Vershynin2018High} with $t = 5\|U^\T \Sigma U\|_2^{1/2}\|\Sigma\|_2^{1/2}\rbr{{r\log d}/{n}}^{1/2}$ and the union bound over $j\in[d]$ and $v\in\N$, we have
	\begin{align}\label{c:5}
	\Pr\cbr{\|U^\T\big(\frac{1}{n}\sum_{i=1}^n Y_iY_i^\T - \Sigma\big)\|_{2, \infty}>5\|U^\T \Sigma U\|_2^{1/2}\|\Sigma\|_2^{1/2}\rbr{{r\log d}/{n}}^{1/2} } \leq \frac{1}{d^2}.
	\end{align}
	For the second term in \eqref{c:4}, we proceed similarly. With $w = \hmu - \mu$, for any $t > 0$,
	\begin{align*}
	\Pr \rbr{\|U^\T ww^\T\|_{2, \infty}>t}\leq  d6^r\Pr \rbr{v^\T U^\T ww^Te_j>\frac{t}{2}}\leq  d6^r\exp\rbr{-\frac{nt}{2 \|U^\T \Sigma U\|_2^{1/2}\|\Sigma\|_2^{1/2}}}.
	\end{align*}
	Setting $t = 6\|U^\T \Sigma U\|_2^{1/2}\|\Sigma\|_2^{1/2}{r\log d}/{n}$, we get
	\begin{align}\label{c:6}
	\Pr\cbr{\|U^\T(\hmu - \mu)(\hmu - \mu)^\T\|_{2, \infty}>6\|U^\T \Sigma U\|_2^{1/2}\|\Sigma\|_2^{1/2}{r\log d}/{n} } \leq \frac{1}{d^2}.
	\end{align}
	Combining \eqref{c:5} and \eqref{c:6} with \eqref{c:4} completes the proof.
\end{proof}

\section{Convex approaches}\label{disc}

We propose two convex relaxation approaches for estimating $\tDelta$ defined in \eqref{dec:3}. First, given the empirical loss in \eqref{loss:popu:SR}, we solve the following regularized convex problem
\begin{align}\label{pro:1}
\min_{S, R}\text{\ \ } \mL_n(S, R) + \lambda_1\|S\|_{1,1} + \lambda_2 \|R\|_{*},
\end{align}
where the sparse and low-rank constraints are removed, and the $\ell_1$ and nuclear norm penalties are used to encourage sparse and low-rank solutions. The tuning parameters $\lambda_1$ and $\lambda_2$ are user specified and control the sparsity of $S$ and the rank of $R$, respectively. Our second convex proposal is to solve the following optimization problem
\begin{equation}\label{pro:2}
\begin{aligned}
\min_{S, R}\text{\ \ } &\|S\|_{1,1} + \lambda\|R\|_*,\\
\text{s.t.} \text{\ \ } & \big\|\hSigmax(S+R)\hSigmay - (\hSigmay - \hSigmax)\big\|_{\infty, \infty}\leq \eta/\lambda,\\
& \big\|\hSigmax(S+R)\hSigmay - (\hSigmay - \hSigmax)\big\|_2\leq \eta.
\end{aligned}
\end{equation}
The constraints in \eqref{pro:2} are obtained by restricting the gradient $\nabla_\Delta \mL_n(\Delta)$ by the dual norms of regularizers in \eqref{pro:1}. Both of them are typically solved by alternating direction method of multipliers, which suffers from high computational cost since each iteration requires an eigenvalue decomposition of $R$ to compute the proximal update corresponding the nuclear norm penalty. \cite{Ma2013Alternating} and \cite{Wang2013Large} proposed accelerated algorithms targeting the above mentioned drawback.  In comparison, nonconvex procedures are widely used to speed up estimation problems involving low-rank matrices \citep{Tu2016Low, Park2018Finding, Chi2018Nonconvex, Yu2018Provable}.

We will compare the estimator in \eqref{pro:1} with the proposed nonconvex approach in the next section. We detail an alternating direction method of multipliers for solving \eqref{pro:1}, which can be rewritten as
\begin{equation}\label{pro:4}
\begin{aligned}
\min_{S, R, \Delta}\ \  &\frac{1}{2}\TR(\Delta\hSigmax\Delta\hSigmay) - \TR\{\Delta(\hSigmay - \hSigmax)\} + \lambda_1\|S\|_{1,1} + \lambda_2 \|R\|_{*},\\
\text{subject to}\ \ &\Delta = S + R,
\end{aligned}
\end{equation}
with the augmented Lagrange function
\begin{multline*}
\hat{\mL}_{\nu}(S, R, \Delta; \Phi) = \frac{1}{2}\TR(\Delta\hSigmax\Delta\hSigmay) - \TR\{\Delta(\hSigmay - \hSigmax)\} + \lambda_1\|S\|_{1,1} + \lambda_2 \|R\|_{*} \\
  - \LD \Phi, \Delta - S - R\RD + \frac{1}{2\nu}\|\Delta - S - R\|_F^2.
\end{multline*}
A natural alternating direction method of multipliers has the following iterations
\begin{align}
\label{eq:admm:iter}
\begin{cases*}
\Delta^{k+1} = \arg\min \hat{\mL}_{\nu}(S^k, R^k, \Delta; \Phi^k),\\
S^{k+1} = \arg\min \hat{\mL}_{\nu}(S, R^k, \Delta^{k+1}; \Phi^k),\\
R^{k+1} = \arg\min \hat{\mL}_{\nu}(S^{k+1}, R, \Delta^{k+1}; \Phi^k),\\
\Phi^{k+1} = \Phi^k - (\Delta^{k+1} - S^{k+1} - R^{k+1})/\nu.
\end{cases*}
\end{align}
For the first iteration, $\Delta^{k+1}$ is obtained by solving the following linear system
\begin{align*}
\cbr{\frac{1}{2}(\hSigmax\otimes\hSigmay + \hSigmay\otimes\hSigmax) + \frac{1}{\nu}I_{d^2}} \text{vec}(\Delta) = \text{vec}\cbr{\frac{1}{\nu}(S^k + R^k) +\Phi^k + \hSigmay - \hSigmax},
\end{align*}
where $\text{vec}(\Delta)\in\mR^{d^2}$ is the vector obtained by vectorizing $\Delta$. Solving the above system, however, requires $O(d^6)$ computations per iteration and is prohibitive even for a small dimensional problem. In comparison, the proposed nonconvex approach requires only $O(d^2r)$ computations per iteration. We further remark on the computational cost of the alternating direction method of multipliers used to estimate a single group Gaussian graphical model with latent variables \citep{Chandrasekaran2012Latent, Ma2013Alternating}. There, each iteration of $\Omega = S + R$ (which denotes the precision matrix of a single group instead) requires two eigenvalue decompositions, which are also computationally expensive, albeit cheaper than solving the linear system above. The other iterates in \eqref{eq:admm:iter} can be obtained explicitly as
\begin{align*}
\begin{cases*}
S^{k+1} = \arg\min \frac{1}{2\nu}\|S + R^k + \nu\Phi^k - \Delta^{k+1}\|_F^2 + \lambda_1\|S\|_{1,1} = \text{Soft}(\Delta^{k+1} - R^k - \nu\Phi^k, \nu\lambda_1),\\
R^{k+1} = \arg\min \frac{1}{2\nu}\|R + S^{k+1} + \nu\Phi^k - \Delta^{k+1}\|_F^2 + \lambda_2\|R\|_* = \overline{\text{Soft}}(\Delta^{k+1} - S^{k+1} - \nu\Phi^k, \nu\lambda_2),\\
\Phi^{k+1} = \Phi^k - (\Delta^{k+1} - S^{k+1} - R^{k+1})/\nu,
\end{cases*}
\end{align*}
where, suppose $U = Q\Lambda Q^\T$ is the eigenvalue decomposition of $U$ and $\Lambda$ is the diagonal matrix,
\begin{equation*}
\cbr{\text{Soft}(U, \xi)}_{i, j} =
\begin{cases}
U_{i, j} - \xi & \text{if } U_{i, j}>\xi, \\
U_{i, j} + \xi & \text{if } U_{i, j}<\xi, \\
0              & \text{otherwise},
\end{cases}
\quad\quad
\overline{\text{Soft}}(U, \xi) = Q\text{Soft}(\Lambda, \xi)Q^\T.
\end{equation*}
From the above iteration regime, we note that an eigenvalue decomposition for $d$-dimensional matrix is required in each iteration.

\section{Additional experiments}\label{add:simu}

\subsection{Recovery of rank and positive index of inertia}

We provide the empirical evidence of the recovery of the rank, $r$, and the positive index of inertia, $r_1$, via cross-validation. For $(d,r) = (50, 1), (100, 1)$ and varying $n$, we check whether the rank chosen by cross-validation is consistent with the true rank and how often we can consistently estimate the positive index of inertia. From the result shown in Figure \ref{ComR}, we see that $r$ and $r_1$ are consistently selected in both cases when $\rbr{d\log d/n}^{1/2}\leq 0.25$.

\begin{figure}[htp!]
\centering     
\includegraphics[width=120mm]{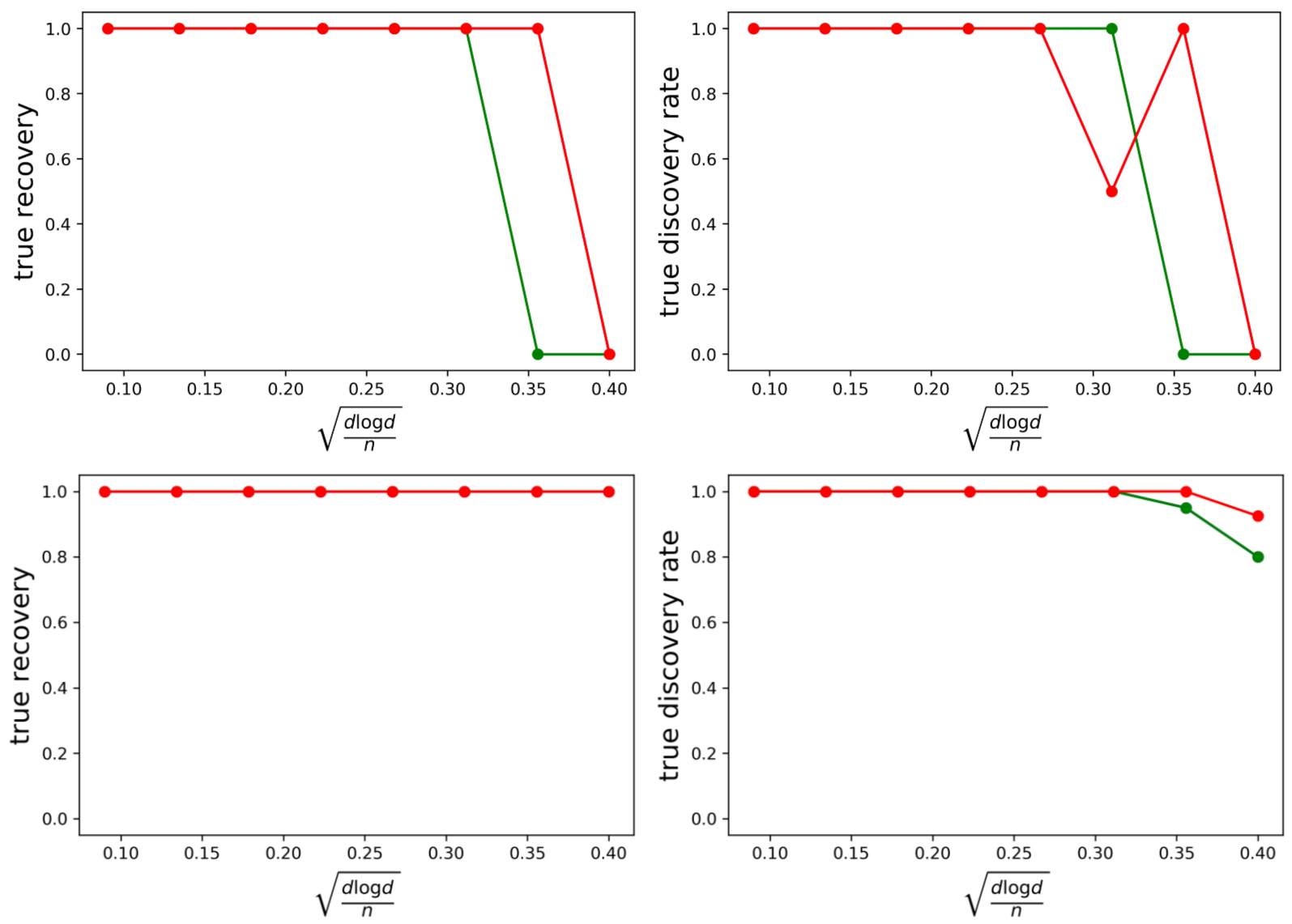}

\includegraphics[width=5.5cm,height=0.3cm]{art1/simu/Errbarleg}

\caption{Recovery of rank (left) and positive index of inertia (right) of $\tR$. The above two panels correspond to $(d,r) = (50,1)$, while the bottom two panels correspond to $(d,r) = (100,1)$. The left figures take value either $0$ or $1$, where $1$ represents the rank chosen by cross-validation is consistent with the true rank. The right figures are the proportion over 40 independent runs that have correct estimate of the positive index of inertia. In both figures, the blue line is covered by the green line. Both of them are correctly recovered when sample size is large enough. }
\label{ComR}

\label{fig:3}
\end{figure}

\subsection{Comparison of convex and nonconvex estimators}

We compare our nonconvex approach with the alternating direction method of multipliers described in the previous section on both the estimation precision and the time required to compute the estimator. We let $n = 10000$, $d = 100$, and $r = 2$ and generate the covariance matrices as in \textsection\ref{sec:5}. In particular, the control group is generated from $\Sigma_1^\star$, while the test $i$ group is generated from $\Sigma_{i+1}^\star$ for $i = 1,2,3$.  For the alternating direction method of multipliers, we initialize the iterates and $\nu$ as in the implementation in \cite{Ma2013Alternating}, while $\lambda_1\in\{0.01, 0.05, 0.1, 0.15\}$ and $\lambda_2\in\{0.15, 0.25, 0.35,0.45\}$ are chosen to minimize the objective function in \eqref{pro:4}  on the validation data. The results are summarized in Table \ref{tab:4}.  Comparing with Table \ref{tab:3}, we see that the convex approach for joint estimation outperforms than the convex approach for separate estimation \citep{Chandrasekaran2012Latent}. However, we also see that the proposed nonconvex approach is preferred both from estimation accuracy and computational efficiency perspectives.

\begin{table}[t]
\centering
\caption{Simulation results for the convex and nonconvex comparison. The estimation errors of the differential network and its sparse component are averaged over 40 independent runs, with standard error given in parentheses. Data are generated following the steps in \textsection\ref{sec:5} with $n = 10000$, $d = 100$, $r=2$. The smallest error under the same setup is highlighted.}\label{tab:4}
\begin{tabular}{*{5}{c}}
\Xhline{4\arrayrulewidth}
\hline
& & Control - Test 1 & Control - Test 2 & Control - Test 3\\
\hline
\multirow{2}{*}{$\|\hS - S^\star\|_F$} & M*&\textbf{12.55}(0.35) & \textbf{11.10}(0.38) & \textbf{10.61}(0.38)\\
& convex & 29.29(0.27) & 32.58(0.29) & 28.31(0.23)\\
\cline{2-5}
\multirow{2}{*}{$\frac{1}{\surd{\sigma_{\max}(R^\star)}}\|\hDelta - \Delta^\star\|_F$} & M* & \textbf{4.87}(0.13) & \textbf{4.52}(0.14) & \textbf{4.37}(0.15)\\
& convex & 11.31(0.10) & 12.84(0.11) & 11.33(0.09)\\
\cline{2-5}
\multirow{2}{*}{time in seconds} & M* & \textbf{3.25}(0.65) & \textbf{3.91}(0.51) & \textbf{6.01}(0.43)\\
& convex & 109.36(5.27) & 107.95(6.70) & 96.25(3.25)\\
\Xhline{4\arrayrulewidth}
\hline
\end{tabular}
\end{table}

\subsection{Loss functions in cross-validation}\label{sec:alter:loss}

We discuss some alternative loss functions that can be used to tuning parameters for other methods. Recall that there are no competing methods designed to directly estimate the differential network from observations comping from Gaussian  graphical models with latent variables. Therefore, throughout the implementation we use the loss function in \eqref{loss:popu:SR} to tune parameters for other methods. This is a reasonable approach since the loss function in \eqref{loss:popu:SR} is simply a quadratic function of the differential network, which in the limit of infinite samples is minimized at the true differential network $\tDelta$. Due to the quadratic curvature, the smaller the loss value the estimator has, the closer it is to the true network.

We provide a comprehensive evaluation of alternative approaches and consider different loss functions to tune parameters for each method. We consider two settings. In the first, we estimate sparse differential networks with $(n, d, r) = (200, 50, 0)$, while in the second the differential network has latent variables with $(n, d, r) = (10000, 100, 2)$. Data generating process follows the description in \textsection\ref{sec:5}. Results are averaged over 40 independent simulation runs.

The method in \cite{Zhao2014Direct} has a tuning parameter $\lambda$, which controls the sparsity of the differential network. See (2) in \cite{Zhao2014Direct}. They proposed two alternative loss functions, $2nL_\infty(\Delta) + \|\Delta\|_{0,1}$ and $2nL_F(\Delta) + \|\Delta\|_{0,1}$, where
\begin{equation}\label{loss:alter}
\begin{aligned}
L_\infty(\Delta) =& \|\hSigmax\Delta\hSigmay - (\hSigmay-\hSigmax)\|_{\infty, \infty},\\
L_F(\Delta) =& \|\hSigmax\Delta\hSigmay - (\hSigmay-\hSigmax)\|_F,
\end{aligned}
\end{equation}
to chose the tuning parameter $\lambda$. These two losses are motivated by the Akaike information criterion. Table \ref{tab:5} summarizes estimation results when different loss functions are used to tune the parameter $\lambda$ in \cite{Zhao2014Direct}. When $r=0$, we report $\|\hS - \tS\|_F$ only, as it is the same as $\|\hDelta - \tDelta\|_F$. We see that using our loss function to tune $\lambda$ results in a better estimate of the differential network when $\tDelta$ is sparse. Furthermore, results are comparable to using the loss $2nL_F(\Delta) + \|\Delta\|_{0,1}$.

\begin{table}[t]
\centering
\caption{Estimation of the differential network using the method of \cite{Zhao2014Direct} with tuning parameter $\lambda$ chosen to minimize different loss functions on the  validation data.}\label{tab:5}
\begin{tabular}{*{5}{c}}
\Xhline{4\arrayrulewidth}
\hline
& \multicolumn{4}{c}{$n = 200$, $d = 50$, $r = 0$}\\
\Xhline{4\arrayrulewidth}
\hline
& & Control - Test 1 & Control - Test 2 & Control - Test 3\\
\hline
\multirow{3}{*}{$\|\hS - S^\star\|_F$} & L1 & \textbf{10.88} & \textbf{11.92} & \textbf{10.64}\\
& L2 & 14.72 & 15.12 & 13.90\\
& L3 & 14.47 & 18.20 & 16.04\\
\Xhline{4\arrayrulewidth}
\hline
& \multicolumn{4}{c}{$n = 10000$, $d = 100$, $r = 2$}\\
\Xhline{4\arrayrulewidth}
\hline
\multirow{3}{*}{$\|\hS - S^\star\|_F$} & L1 & \textbf{39.50} & 50.09 & \textbf{37.64} \\
& L2 & 44.94 & 57.83 & 41.05\\
& L3 & 39.54 & \textbf{49.88} & 38.62\\
\hline
\hline
\multirow{3}{*}{$\frac{1}{\surd{\sigma_{\max}(R^\star)}}\|\hDelta - \Delta^\star\|_F$} & L1 & \textbf{14.91} & 19.49 & \textbf{14.82}\\
& L2 & 17.11 & 22.58 & 16.30\\
& L3 & 14.92 & \textbf{19.40} & 15.81\\
\Xhline{4\arrayrulewidth}
\hline
\multicolumn{5}{l}{
  \multirow{2}{*}{
	\parbox{33pc}{L1, the loss in \eqref{loss:popu:SR};
	L2, the loss $2nL_\infty(\Delta) + \|\Delta\|_{0,1}$;
	L3, the loss $2nL_F(\Delta) + \|\Delta\|_{0,1}$;
	$L_\infty$ and $L_F$ are defined in \eqref{loss:alter}.}
	}
	}
\end{tabular}
\end{table}

The method in \cite{Yuan2017Differential} also has a tuning parameter $\lambda$, which controls the sparsity of the differential network. See (2) in \cite{Yuan2017Differential}. In addition to the loss functions used above, we also consider the following two loss functions proposed in \cite{Yuan2017Differential} based on Bayesian information criterion, $2n\bar{L}_\infty(\Delta) + \log(2n)\|\Delta\|_{0,1}$ and $2n\bar{L}_F(\Delta) + \log(2n)\|\Delta\|_{0,1}$, where
\begin{equation}\label{loss:alter:1}
\begin{aligned}
\bar L_\infty(\Delta) =& \|(\hSigmax\Delta\hSigmay + \hSigmay\Delta\hSigmax)/2 - (\hSigmay - \hSigmax)\|_{\infty, \infty},\\
\bar L_F(\Delta) =& \|(\hSigmax\Delta\hSigmay + \hSigmay\Delta\hSigmax)/2 - (\hSigmay - \hSigmax)\|_F.
\end{aligned}
\end{equation}
Table \ref{tab:6} summarizes estimation results when different loss functions are used to tune the parameter $\lambda$ in \cite{Yuan2017Differential}. We see that when $(n, d, r) = (200, 50, 0)$ and the loss function \eqref{loss:popu:SR} is used to choose the tuning parameter the estimation error is smallest. When $(n,d,r) = (10000, 100, 2)$, the loss function \eqref{loss:popu:SR} and $2nL_F(\Delta) + \|\Delta\|_{0,1}$ give comparable results. Therefore, using our loss function to tune parameters in \cite{Yuan2017Differential} is reasonable.

\begin{table}[t]
\centering \caption{Estimation of the differential network using the method of \cite{Yuan2017Differential} with tuning parameter $\lambda$ chosen to minimize different loss functions on the  validation data.}\label{tab:6}
\begin{tabular}{*{5}{c}}
\Xhline{4\arrayrulewidth}
\hline
& \multicolumn{4}{c}{$n = 200$, $d = 50$, $r = 0$}\\
\Xhline{4\arrayrulewidth}
\hline
& & Control - Test 1 & Control - Test 2 & Control - Test 3\\
\hline
\multirow{5}{*}{$\|\hS - S^\star\|_F$} & L1 & \textbf{11.37} & \textbf{10.81} & \textbf{10.37}\\
& L2 & 14.75 & 11.80 & 14.24\\
& L3 & 11.47 & 11.80 & 13.97\\
& L4 & 14.78 & 15.25 & 13.60\\
& L5 & 14.57 & 11.89 & 13.99\\
\Xhline{4\arrayrulewidth}
\hline
& \multicolumn{4}{c}{$n = 10000$, $d = 100$, $r = 2$}\\
\Xhline{4\arrayrulewidth}
\hline
\multirow{5}{*}{$\|\hS - S^\star\|_F$} & L1 & 27.86 & 32.99 & \textbf{29.58}\\
& L2 & 33.28 & 39.15 & 35.14 \\
& L3 & \textbf{27.80} & \textbf{32.56} & 29.76\\
& L4 & 41.44 & 39.23 & 35.79\\
& L5 & 33.21 & 38.79 & 35.60\\
\hline
\hline
\multirow{5}{*}{$\frac{1}{\surd{\sigma_{\max}(R^\star)}}\|\hDelta - \Delta^\star\|_F$} & L1 & 10.41 & 12.77 & \textbf{11.56}\\
& L2 & 12.72 & 15.35 & 13.94\\
& L3 & \textbf{10.38} & \textbf{12.58} & 11.65\\
& L4 & 15.76 & 15.39 & 14.20\\
& L5 & 12.69 & 15.22 & 14.13\\
\Xhline{4\arrayrulewidth}
\hline
\multicolumn{5}{l}{
  \multirow{3}{*}{
	\parbox{33pc}{L1, the loss in \eqref{loss:popu:SR};
	L2, the loss $2nL_\infty(\Delta) + \|\Delta\|_{0,1}$;
	L3, the loss $2nL_F(\Delta) + \|\Delta\|_{0,1}$;
	L4, the loss $2n\bar{L}_\infty(\Delta) + \log(2n)\|\Delta\|_{0,1}$;
	L5, the loss $2n\bar{L}_F(\Delta) + \log(2n)\|\Delta\|_{0,1}$;
	$L_\infty$ and $L_F$ are defined in \eqref{loss:alter}, while
	$\bar L_\infty$ and $\bar L_F$ are defined in \eqref{loss:alter:1}.}
	}
	}\\[1em]
\end{tabular}
\end{table}

The method used to estimate single latent variable Gaussian graphical model \cite{Chandrasekaran2012Latent, Ma2013Alternating} requires selecting the tuning parameters $\alpha$ and $\beta$. See (2.1) in \cite{Ma2013Alternating}. We consider three alternative loss functions: $L_\infty(\Delta)$, $L_F(\Delta)$, and the penalized Gaussian likelihood, which was  used in \cite{Chandrasekaran2012Latent, Ma2013Alternating} and is given as
\begin{align}\label{Loss:pen:Gau}
L_p(\Omega) = \TR(\Omega\hSigma) - \log\det(\Omega) + \alpha \|S\|_{1,1} + \beta \|R\|_*,
\end{align}
where $\Omega = S+R$ is the precision matrix of a single group and $\hSigma$ is either $\hSigmax$ or $\hSigmay$. For the penalized Gaussian likelihood, it is designed for estimating a single group, hence we have to tune parameters separately for each group when using such loss function. The differential network is then obtained by computing the difference between the estimated precision matrices. The results are summarized in Table \ref{tab:7} and show that all the loss functions are comparable, with no loss function dominating others.

\begin{table}[t]
\centering
\caption{
Estimation of the differential network using the method of \cite{Ma2013Alternating} with tuning parameter $\alpha$ and $\beta$ chosen to minimize different loss functions on the  validation data.}\label{tab:7}
\begin{tabular}{*{5}{c}}
\Xhline{4\arrayrulewidth}
\hline
& \multicolumn{4}{c}{$n = 200$, $d = 50$, $r = 0$}\\	\Xhline{4\arrayrulewidth}
\hline
& & Control - Test 1 & Control - Test 2 & Control - Test 3\\
\hline
\multirow{4}{*}{$\|\hS - S^\star\|_F$} & L1 & \textbf{11.04} & \textbf{11.23} & \textbf{10.48}\\
& L2 & 11.90 & 11.75 & 12.71\\
& L3 & 12.99 & 12.20 & 12.41\\
& L4 & 11.23 & 11.37 & 10.76\\
\hline\hline
\multirow{4}{*}{$\frac{1}{\surd{\sigma_{\max}(R^\star)}}\|\hDelta - \Delta^\star\|_F$} & L1 & \textbf{10.85} & 10.86 & \textbf{10.37}\\
& L2 & 11.57 & \textbf{10.70} & 12.31\\
& L3 &  12.77 & 12.00 & 12.05\\
& L4 & 11.24 & 11.37 & 10.76\\
\Xhline{4\arrayrulewidth}
\hline
& \multicolumn{4}{c}{$n = 10000$, $d = 100$, $r = 2$}\\
\Xhline{4\arrayrulewidth}
\hline
\multirow{5}{*}{$\|\hS - S^\star\|_F$} & L1 & 30.54 & 34.11 & 31.80\\
& L2 &  \textbf{30.27} & 34.32 & \textbf{31.69}\\
& L3 & 30.50 & \textbf{34.08} & 31.81\\
& L4 & 30.52 & 34.09 & 31.74\\
\hline\hline
\multirow{4}{*}{$\frac{1}{\surd{\sigma_{\max}(R^\star)}}\|\hDelta - \Delta^\star\|_F$} & L1 & 11.51 & 13.27 & \textbf{12.47}\\
& L2 & 11.63 & 13.44 & 12.64\\
& L3 & \textbf{11.49} & \textbf{13.25} & 12.48\\
& L4 & 11.50 & 13.26 & 12.49\\
\Xhline{4\arrayrulewidth}
\hline
\multicolumn{5}{l}{
  \multirow{2}{*}{
	\parbox{33pc}{L1, the loss in \eqref{loss:popu:SR};
	L2, the loss $L_\infty(\Delta)$;
	L3, the loss $L_F(\Delta)$;
	L4, the penalized Gaussian likelihood defined in \eqref{Loss:pen:Gau}.}
	}
	}
\end{tabular}
\end{table}

\cite{Xu2017Speeding} estimates a single latent variable Gaussian graphical model using a nonconvex estimator. We tune parameters $\hat{s}$ and $r$,  while other tuning parameters are kept as in their implementation. We consider three loss functions: $L_\infty(\Delta)$, $L_F(\delta)$, and the Gaussian likelihood. Estimation results are summarized in Table \ref{tab:8}. When $r= 0$, we report only $\|\hS - \tS\|_F$, as the correct rank is selected using any of the loss functions. We observe that all loss functions lead to similar choice of tuning parameters and results.

\begin{table}[htp]
\centering
\caption{Estimation of the differential network using the method of \cite{Xu2017Speeding} with tuning parameter $\hat{s}$ and $r$ chosen to minimize different loss functions on the  validation data.}\label{tab:8}
\begin{tabular}{*{5}{c}}
\Xhline{4\arrayrulewidth}
\hline
& \multicolumn{4}{c}{$n = 200$, $d = 50$, $r = 0$}\\	\Xhline{4\arrayrulewidth}
\hline
& & Control - Test 1 & Control - Test 2 & Control - Test 3\\
\hline
\multirow{4}{*}{$\|\hS - S^\star\|_F$} & L1 & \textbf{13.51} & \textbf{14.79} & \textbf{12.71}\\
& L2 & 13.59 & 19.23 & 15.38 \\
& L3 & 13.51 & 15.12 & 12.79\\
& L4 & 13.51 & 15.11 & 12.78\\
\Xhline{4\arrayrulewidth}
\hline
& \multicolumn{4}{c}{$n = 10000$, $d = 100$, $r = 2$}\\
\Xhline{4\arrayrulewidth}
\hline
\multirow{5}{*}{$\|\hS - S^\star\|_F$} & L1 & 18.88 & 17.44 & \textbf{14.63}\\
& L2 & 16.62 & \textbf{15.27} & 19.92\\
& L3 & \textbf{16.55} & 17.41 & 14.69\\
& L4 & 18.96 & 17.41 & 14.71\\
\hline\hline
\multirow{4}{*}{$\frac{1}{\surd{\sigma_{\max}(R^\star)}}\|\hDelta - \Delta^\star\|_F$} & L1 & 6.58 & 6.44 & \textbf{5.60}\\
& L2 & 6.03 &  \textbf{5.78} & 7.29\\
& L3 & \textbf{5.99} & 6.45 & 5.69\\
& L4 & 6.61 & 6.45 & 5.64\\
\Xhline{4\arrayrulewidth}
\hline
\multicolumn{5}{l}{
  L1, the loss in \eqref{loss:popu:SR};
	L2, the loss $L_\infty(\Delta)$;
	L3, the loss $L_F(\Delta)$;
	L4, the Gaussian likelihood.
	}
\end{tabular}
\end{table}

Based on the additional simulation results in this section, we find that using the loss function \eqref{loss:popu:SR} to select the tuning parameters of all the procedures leads to comparable estimates to those that would be obtained by alternative loss functions. In particular, our approach to tuning parameter selection does not favor our estimation procedure.

\subsection{Algorithmic convergence and initialization}

We study how the initialization in Algorithm~\ref{alg:1} affects the convergence rate. The data generating procedure is as in \textsection\ref{sec:5} with $(n,d,r) = (150, 100, 1)$ and $(150,100,2)$. We compare with random initialization, where $S^0 = 0_{d\times d}$, $(U^0)_{i, j}\sim \mN(0, 3)$ for $1\leq i\leq d$ and $1\leq j\leq r$, and $L^0 = U^0\tLambda U^{0\T}$. In particular, we assume the rank and the positive index of inertia of $\tR$ are both correctly specified.

Figure \ref{fig:4} shows the error in estimation, $2\log\{\|\hDelta - \tDelta\|_F/\surd{\sigma_{\max}(R^\star)}\}$, against the iteration number. We plot the first 300 iterations for random initialization, as, typically, around 3000 iterations are needed for the algorithm to converge when randomly initialized. Furthermore, we observe that the iterates diverge in the first few steps for random initialization. We resolve this problem by adjusting the learning rates to $\eta_1/k$ and $\eta_2/k$, where $k$ is the iteration number, whenever we observe an increase in the objective value. From the results, we observe that our method converges smoothly to a stationary point with a nearly linear rate of convergence, as suggests by Theorem \ref{thm:1}. For random initialization, even after reducing the step size to fix the divergence problem, we observe that the accuracy of the estimator is not comparable to our estimator, as summarized in Table \ref{tab:12}. This illustrates the necessity for having a good initialization in the proposed two-stage algorithm.

\begin{figure}[thp!]
\centering     

\subfigure[Algorithmic convergence when estimating $\tDelta$ with $(n,d,r) = (150,100,1)$.]
{\label{cov1}\includegraphics[width=130mm]{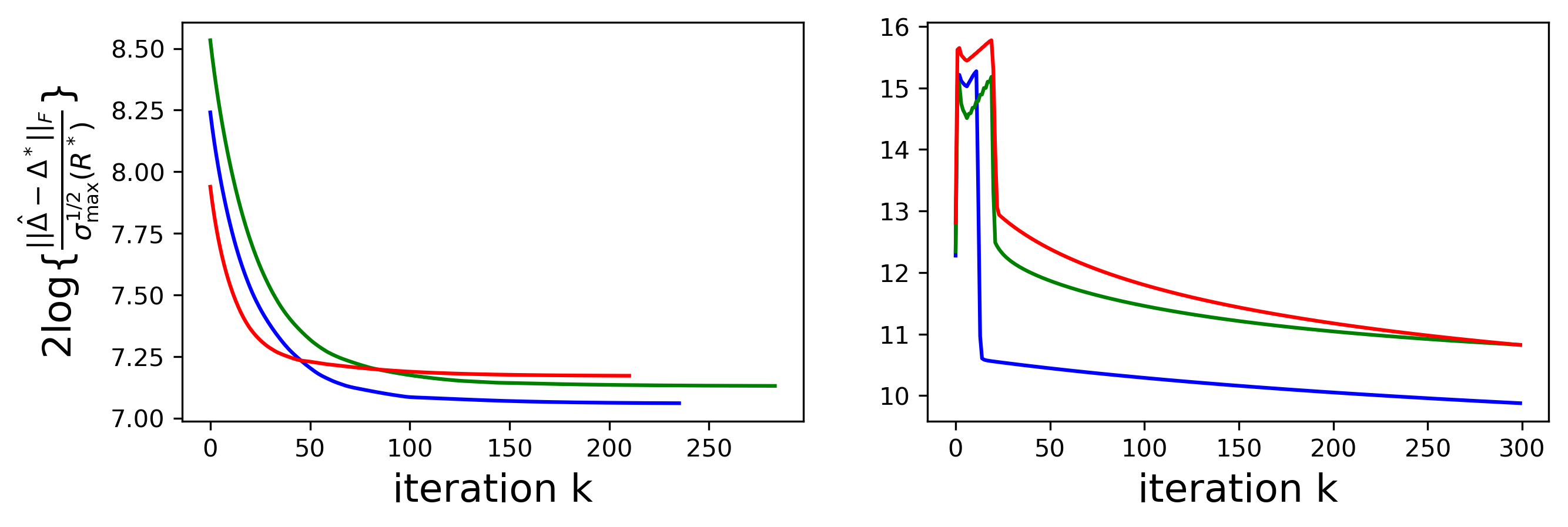}}

\subfigure[Algorithmic convergence when estimating $\tDelta$ with $(n,d,r) = (150,100,2)$.]
{\label{cov2}\includegraphics[width=130mm]{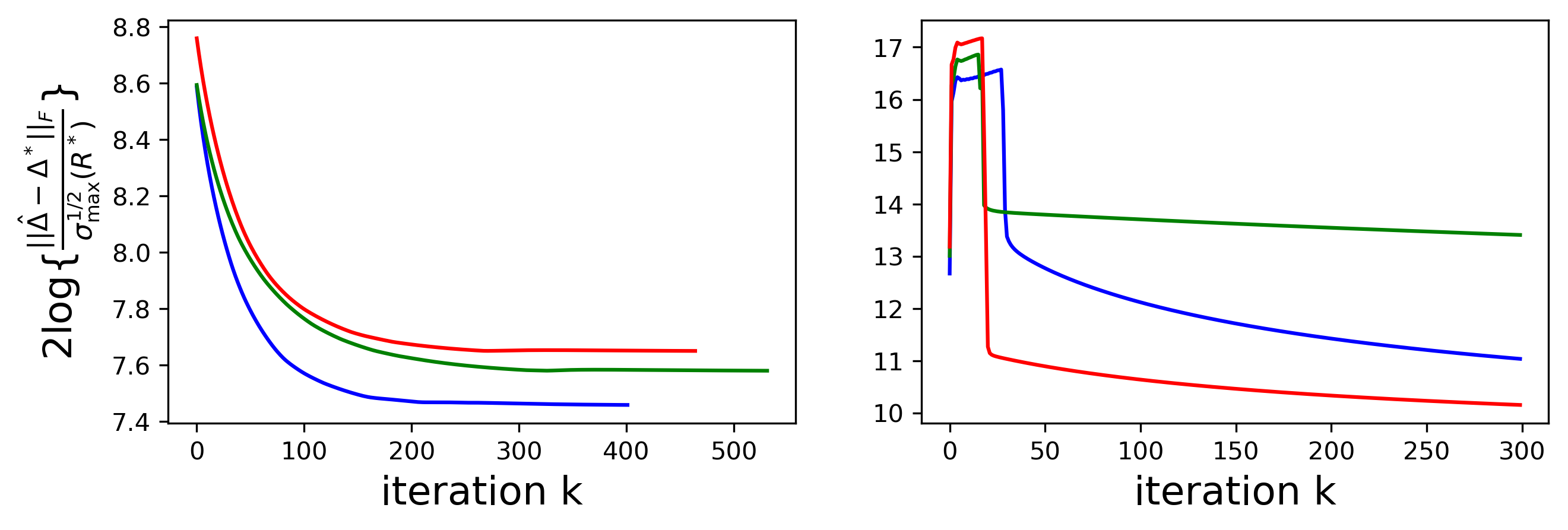}}

\includegraphics[width=5.5cm,height=0.3cm]{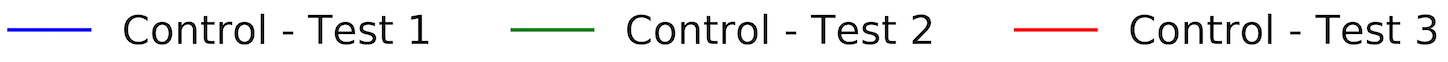}

\caption{Algorithmic convergence. The left plots illustrate convergence of Algorithm~\ref{alg:2} when initialized with Algorithm \ref{alg:1}, while the right plots correspond to random initialization.}

\label{fig:4}
\end{figure}

\begin{table}[thp!]
	\centering
	\caption{Comparison of random initialization with Algorithm~\ref{alg:2}. The error in estimation
		is measured by $\|\hDelta - \Delta^\star\|_F/\surd{\sigma_{\max}(R^\star)}$.}\label{tab:12}
	\begin{tabular}{ccccc}
		Setup & Initialization & Control - Test 1  & Control - Test 2 & Control - Test 3 \\
		\hline
		\multirow{2}{*}{$(150, 100, 1)$} & Algorithm \ref{alg:1} & \textbf{34.14} & \textbf{35.36} & \textbf{36.09}\\
		& Random & 79.09 & 169.29 & 114.25\\
		\hline
		\multirow{2}{*}{$(150, 100, 2)$} & Algorithm \ref{alg:1} & \textbf{41.65} & \textbf{44.25} & \textbf{45.83}\\
		& Random & 114.81 & 381.35 & 109.12\\
		\hline
	\end{tabular}
\end{table}

\subsection{Advantages of direct estimation}\label{simu:advantage}

We run additional simulations to better illustrate benefits of our direct estimator of the differential network. Here we consider more general structures on single group precision matrices. Compared to direct estimators in \cite{Zhao2014Direct} and \cite{Yuan2017Differential}, we allow the differential network to have a low-rank component. Compared to separate estimation procedures in \cite{Chandrasekaran2012Latent} and \cite{Xu2017Speeding}, we do not rely on group specific precision matrices to have structure.

We consider two additional data generating settings: (i) only the differential network has sparse plus low-rank structure, while each group has no structure; (ii) one group has a sparse precision matrix, while the other has  a sparse plus low-rank precision matrix. For the setting (i), we first set the precision matrices $\Omega_i^\star$, $i=1,\ldots,4$, as described in \textsection\ref{sec:5}. We then modify the top left $d\times d$ submatrix of each precision matrix by adding a symmetric positive definite matrix  $P = d\cdot I_d + (P_1 + P_1^\T)/2$, where $(P_1)_{i, j}$ is generated independently from $\text{Uniform}(0, 1)$. Let  $\bar{\Omega}_i^\star$ denote the modified precision matrices and $\bar{\Sigma}_i^\star = (\bar{\Omega}_i^\star)^{-1}$. The control group $X$ is generated based on $\bar{\Sigma}_1^\star$, while $\bar{\Sigma}_i^\star$ are used to generate test $i-1$ groups for $i = 2,3,4$. Note that under this data generating process, the differential network still has the same sparse plus low rank structure as in \textsection\ref{sec:5}. However, the marginal precision matrices of the observed variables do not have the sparse plus low-rank structure, since they are  blurred by a positive definite matrix $P$. For the setting (ii), the precision matrix for the control group is sparse, while the precision matrix for the test group is generated as in \textsection\ref{sec:5}. The sparse precision is generated as follows: we let $\Omega_{OH}^\star$ and $\Omega_{HO}^\star$ be zero matrices, and $\Omega_{HH}^\star$ be the identity matrix. The diagonal entries of $\Omega_{OO}$ are all equal to $1$, and $(\Omega_{OO})_{i, d+1-i} = 1$ for $\leq i\leq d$, $(\Omega_{OO})_{i, d-i} = (\Omega_{OO})_{i+1, d+1-i} = 0.6$ for $1\leq i \leq d-1$, and $(\Omega_{OO})_{i, d-1-i} = (\Omega_{OO})_{i+2, d+1-i} = 0.3$ for $1\leq i \leq d-2$. Then $\Omega_{OO}^\star = \Omega_{OO} + (\iota + 1) I_d$ where $\iota = |\min\{\eig(\Omega_{OO})\}|$. Combining all blocks together, the covariance matrix for the control group is $(\Omega_{OO}^\star)^{-1}$.

Using the two models described above, we generate data with $(n,d,r) = (5000, 30, 0)$, $(10000, 30, 1)$, and $(20000,30,2)$. The implementation details for the proposed and competing methods are described in \textsection\ref{sec:5}. We report the average error over 40 independent runs. Results are summarized in Tables \ref{tab:9}-\ref{tab:11}.  Unsurprisingly, our method outperforms other methods in almost all cases. Specifically, we observe that when each network has no structure, but the differential network is decomposable as sparse plus low-rank matrix, our method performs strictly better than all other methods. This illustrates that our direct estimation procedure does not rely on the structure of a single group and results in a wider applicability.

\newpage
\begin{rotatepage}
\begin{landscape}
\begin{table}[p]
\centering
\caption{Simulation results for five algorithms corresponding to the data generating process in Appendix \ref{simu:advantage}.}\label{tab:9}
\begin{tabular}{*{7}{c}}
\Xhline{4\arrayrulewidth}
\hline
& \multicolumn{6}{c}{$n = 5000$, $d = 30$, $r = 0$}\\
\Xhline{4\arrayrulewidth}
\hline
& \multicolumn{2}{c}{Control - Test 1} & \multicolumn{2}{c}{Control - Test 2} & \multicolumn{2}{c}{Control - Test 3}\\

\hline
\\[-0.9em]
Method & $\|\hS - S^\star\|_F$& $\frac{1}{\surd{\sigma_{\max}(R^\star)}}\|\hDelta - \Delta^\star\|_F$ & $\|\hS - S^\star\|_F$&  $\frac{1}{\surd{\sigma_{\max}(R^\star)}}\|\hDelta - \Delta^\star\|_F$ & $\|\hS - S^\star\|_F$ &  $\frac{1}{\surd{\sigma_{\max}(R^\star)}}\|\hDelta - \Delta^\star\|_F$\\
\Xhline{4\arrayrulewidth}

\hline
& \multicolumn{6}{c}{Case (i)}\\
\Xhline{1\arrayrulewidth}

\hline
M* & \textbf{10.23} & \textbf{10.23} &\textbf{ 9.42} & \textbf{9.42} & \textbf{8.12} & \textbf{8.12}\\
\hline
M1 &  14.01 & 14.01 & 12.03 & 12.03 & 10.32 & 10.32\\
\hline
M2 & 13.40 & 13.40 & 16.26 & 16.26 & 12.36 & 12.36\\
\hline
M3 & 13.21 & 13.21 & 11.32 & 11.32 & 9.67 & 9.67\\
\hline
M4 & 20.70 & 20.70 & 21.00 & 21.00 & 20.58 & 20.58\\
\Xhline{2\arrayrulewidth}
\hline
\hline
\Xhline{2\arrayrulewidth}
& \multicolumn{6}{c}{Case (ii)}\\
\Xhline{1\arrayrulewidth}

\hline
M* &  1.95 & 1.95 & \textbf{1.38} & \textbf{1.45} & \textbf{1.69} & \textbf{1.69}\\
\hline
M1 &  \textbf{1.93} & \textbf{1.93} & 2.16 & 2.16 & 1.95 & 1.95\\
\hline
M2 &  1.94 & 1.94 & 2.07 & 2.07 & 1.99 & 1.99\\
\hline
M3 &  2.37 & 2.37 & 2.06 & 2.06 & 2.23 & 2.23\\
\hline
M4 & 2.15 & 2.15 & 2.00 & 2.00 & 2.24 & 2.24\\
\Xhline{4\arrayrulewidth}
\hline
\multicolumn{7}{l}{
\multirow{3}{*}{
\parbox{47pc}{
M*, the proposed method; M1, $\ell_1$-minimization in \cite{Zhao2014Direct}; M2, $\ell_1$-penalized quadratic loss in \cite{Yuan2017Differential};
M3, penalized Gaussian likelihood in \cite{Chandrasekaran2012Latent}; M4, constrained Gaussian likelihood in \cite{Xu2017Speeding};
detailed descriptions of each method are given in Table \ref{tab:1} and the choice of tuning parameters is discussed
in \textsection\ref{sec:5.1}.
}
}
}
\end{tabular}
\end{table}
\end{landscape}

\end{rotatepage}

\begin{rotatepage}
\begin{landscape}
\begin{table}[p]
\centering
\caption{Simulation results for five algorithms corresponding to the data generating process in Appendix \ref{simu:advantage}.}\label{tab:10}
\begin{tabular}{*{7}{c}}
\Xhline{4\arrayrulewidth}
\hline
& \multicolumn{6}{c}{$n = 10000$, $d = 30$, $r = 1$}\\
\Xhline{4\arrayrulewidth}

\hline
& \multicolumn{2}{c}{Control - Test 1} & \multicolumn{2}{c}{Control - Test 2} & \multicolumn{2}{c}{Control - Test 3}\\

\hline
\\[-0.9em]
Method & $\|\hS - S^\star\|_F$& $\frac{1}{\surd{\sigma_{\max}(R^\star)}}\|\hDelta - \Delta^\star\|_F$ & $\|\hS - S^\star\|_F$&  $\frac{1}{\surd{\sigma_{\max}(R^\star)}}\|\hDelta - \Delta^\star\|_F$ & $\|\hS - S^\star\|_F$ &  $\frac{1}{\surd{\sigma_{\max}(R^\star)}}\|\hDelta - \Delta^\star\|_F$\\
\Xhline{4\arrayrulewidth}

\hline
& \multicolumn{6}{c}{Case (i)}\\
\Xhline{1\arrayrulewidth}
\hline
M* & \textbf{9.01} & \textbf{5.76} & \textbf{8.11} & \textbf{5.39} & \textbf{7.99} & \textbf{5.27}\\
\hline
M1 &  19.80 & 12.07 & 21.11& 12.63 & 23.08 & 13.62\\
\hline
M2 & 19.04 & 11.54 & 17.51 & 10.37 & 20.08 & 11.86\\
\hline
M3 & 18.39 & 11.22 & 19.56 & 11.72 & 21.28 & 12.58\\
\hline
M4 & 16.18 & 9.70 & 16.38 & 9.55 & 16.26 & 9.36\\
\Xhline{2\arrayrulewidth}
\hline\hline
\Xhline{2\arrayrulewidth}
& \multicolumn{6}{c}{Case (ii)}\\
\Xhline{1\arrayrulewidth}
\hline
M* &  \textbf{3.20} & \textbf{1.06} & \textbf{2.53} & \textbf{1.05} & \textbf{3.22} & \textbf{1.00}\\
\hline
M1 &  4.72 & 1.41 & 4.53 & 1.30 & 4.87 & 1.50\\
\hline
M2 &  5.08 & 1.53 & 4.86 & 1.46 & 5.14 & 1.54\\
\hline
M3 &  7.03 & 3.18 & 6.56 & 3.20 & 6.43 & 2.89\\
\hline
M4 & 3.77 & 1.91 & 5.36 & 1.08 & 5.72 & 1.04\\
\Xhline{4\arrayrulewidth}
\hline
\multicolumn{7}{l}{
\multirow{3}{*}{
\parbox{47pc}{
M*, the proposed method; M1, $\ell_1$-minimization in \cite{Zhao2014Direct}; M2, $\ell_1$-penalized quadratic loss in \cite{Yuan2017Differential};
M3, penalized Gaussian likelihood in \cite{Chandrasekaran2012Latent}; M4, constrained Gaussian likelihood in \cite{Xu2017Speeding};
detailed descriptions of each method are given in Table \ref{tab:1} and the choice of tuning parameters is discussed
in \textsection\ref{sec:5.1}.
}
}
}
\end{tabular}
\end{table}

\end{landscape}

\end{rotatepage}

\begin{rotatepage}
\begin{landscape}
\begin{table}[p]
\centering
\caption{Simulation results for five algorithms corresponding to the data generating process in Appendix \ref{simu:advantage}.}\label{tab:11}
\begin{tabular}{*{7}{c}}
\Xhline{4\arrayrulewidth}
\hline
& \multicolumn{6}{c}{$n = 20000$, $d = 30$, $r = 2$}\\
\Xhline{4\arrayrulewidth}
\hline
& \multicolumn{2}{c}{Control - Test 1} & \multicolumn{2}{c}{Control - Test 2} & \multicolumn{2}{c}{Control - Test 3}\\

\hline
\\[-0.9em]
Method & $\|\hS - S^\star\|_F$& $\frac{1}{\surd{\sigma_{\max}(R^\star)}}\|\hDelta - \Delta^\star\|_F$ & $\|\hS - S^\star\|_F$&  $\frac{1}{\surd{\sigma_{\max}(R^\star)}}\|\hDelta - \Delta^\star\|_F$ & $\|\hS - S^\star\|_F$ &  $\frac{1}{\surd{\sigma_{\max}(R^\star)}}\|\hDelta - \Delta^\star\|_F$\\
\Xhline{4\arrayrulewidth}

\hline
& \multicolumn{6}{c}{Case (i)}\\
\Xhline{1\arrayrulewidth}

\hline
M* & \textbf{9.18} & \textbf{4.77} & \textbf{8.65} & \textbf{4.59} & \textbf{8.31} & \textbf{4.46}\\
\hline
M1 & 26.79 & 13.13 & 31.90 & 15.53 & 31.80 & 14.91 \\
\hline
M2 & 22.00 & 10.30 & 25.73 & 12.11 & 24.66 & 11.14\\
\hline
M3 & 24.95 & 12.21 & 29.64 & 14.41 & 29.52 & 13.82\\
\hline
M4 & 12.91 & 5.98 & 13.03 & 5.99 & 13.32 & 5.83\\
\Xhline{2\arrayrulewidth}
\hline
\hline
\Xhline{2\arrayrulewidth}
& \multicolumn{6}{c}{Case (ii)}\\
\Xhline{1\arrayrulewidth}

\hline
M* &\textbf{3.83} & \textbf{0.97} & \textbf{2.86} & 0.83 & \textbf{4.63} & 1.03 \\
\hline
M1 &  6.12 & 1.49 & 6.27 & 1.58 & 7.42 & 1.84\\
\hline
M2 &  6.82 & 1.83 & 6.63 & 1.68 & 7.65 & 1.98\\
\hline
M3 &  10.44 & 4.02 & 10.01 & 3.91 & 11.62 & 4.31\\
\hline
M4 & 6.99 & 1.83 & 6.73 & \textbf{0.77} & 7.68 & \textbf{0.78}\\
\Xhline{4\arrayrulewidth}
\hline
\multicolumn{7}{l}{
\multirow{3}{*}{
\parbox{47pc}{
M*, the proposed method; M1, $\ell_1$-minimization in \cite{Zhao2014Direct}; M2, $\ell_1$-penalized quadratic loss in \cite{Yuan2017Differential};
M3, penalized Gaussian likelihood in \cite{Chandrasekaran2012Latent}; M4, constrained Gaussian likelihood in \cite{Xu2017Speeding};
detailed descriptions of each method are given in Table \ref{tab:1} and the choice of tuning parameters is discussed
in \textsection\ref{sec:5.1}.
}
}
}
\end{tabular}
\end{table}

\end{landscape}

\end{rotatepage}

\putbib[paper]
\end{bibunit}

\end{document}